\documentclass[a4paper,DIV12,12pt]{article}
\usepackage[a4paper, top=20mm, bottom =20mm, textwidth=165mm]{geometry}
\usepackage{amsthm,amsmath,amssymb,amsfonts,tikz,amscd,dsfont}
\usepackage{enumitem}
\usepackage{color}
\usepackage[pdftex, bookmarksnumbered=true]{hyperref}
\usepackage{mathrsfs}
\usepackage[utf8]{inputenc}

\theoremstyle{plain}
\newtheorem{define}{Definition}[section]
\newtheorem{notati}[define]{Notation}
\newtheorem{lemma}[define]{Lemma}

\newtheorem{proposition}[define]{Proposition}
\newtheorem*{proposition*}{Proposition}
\newtheorem{corollary}[define]{Corollary}
\newtheorem{theorem}[define]{Theorem}
\newtheorem{remark}[define]{Remark}
\newtheorem{example}[define]{Example}

\numberwithin{equation}{section}

\newcommand{\ri}{{\mathrm{i}}}
\newcommand{\re}{{\mathrm{e}}}

\title{Highest-weight vectors and three-point functions in GKO coset decomposition}

\author{Mikhail Bershtein \and Boris Feigin \and Aleksandr Trufanov}

\date{March 2025}

\begin{document}

\maketitle
\begin{abstract}
	We revisit the classical Goddard-Kent-Olive coset construction. We find the formulas for the highest weight vectors in coset decomposition and calculate their norms. We also derive formulas for matrix elements of natural vertex operators between these vectors. This leads to relations on conformal blocks. Due to the AGT correspondence, these relations are equivalent to blowup relations on Nekrasov partition functions with the presence of the surface defect. These relations can be used to prove Kyiv formulas for the Painlev\'{e} tau-functions (following Nekrasov's method).
\end{abstract}
\tableofcontents
\section{Introduction}

\paragraph{Coset construction} Goddard-Kent-Olive coset construction is one of the most basic and important constructions in the theory of vertex algebras. This construction can be viewed as an affine analog of the decomposition of tensor product of representations of \(\mathfrak{sl}(2)\). Let \(\widehat{\mathfrak{sl}}(2)_k\) denotes affine (vertex) algebra \(\mathfrak{sl}_2\) where central element acts by number \(k\in \mathbb{C}\). On the tensor product of representations of \(\widehat{\mathfrak{sl}}(2)_1\) and \(\widehat{\mathfrak{sl}}(2)_{k+1}\) there is a natural action of diagonal \(\widehat{\mathfrak{sl}}(2)_{k+1}\). It was observed in \cite{Goddard:1986} that multiplicity spaces have a natural structure of representation of the Virasoro algebra with central charge \(c=13-6(\frac{k+2}{k+3}+\frac{k+3}{k+2})\). Moreover, there is an isomorphism of vertex algebras (assume that \(k \not \in  \mathbb{Q}\))
\begin{equation}\label{eq:coset algebras}
	\widehat{\mathfrak{sl}}(2)_1 \otimes \widehat{\mathfrak{sl}}(2)_k \simeq \widehat{\mathfrak{sl}}(2)_{k+1} \overline{\otimes} \mathrm{Vir}_{c}
\end{equation}
The sign \(\overline{\otimes}\) on the right side of this isomorphism stands for the extension by the sum of degenerate representations. Due to this extension, each highest weight vector (primary field) \(v_\lambda\) of $\widehat{\mathfrak{sl}}(2)_1 \otimes \widehat{\mathfrak{sl}}(2)_k $ corresponds to infinitely many highest weight vectors \(u_l(\lambda)\), \(l \in \frac12\mathbb{Z}\) of \(\widehat{\mathfrak{sl}}(2)_{k+1} {\otimes} \mathrm{Vir}_{c}\). See decompositions \eqref{eq:W1 W2} in the main text. 

Usually in conformal fields theory, we start with computation of correlation functions of primary fields. Therefore it is natural to ask for correlation functions of \(u_l(\lambda)\). The main result of the paper is a computation of a three-point function of \(u_l(\lambda), u_n(\nu), u_m(\mu)\) given in Theorem~\ref{Th:main}. To be more precise we compute a ratio of this three-point function with the three-point function of \(u_0(\lambda)=v_\lambda, u_0(\nu)=v_\nu, u_0(\mu)=v_\mu\). The resulting formula has a structure similar to DOZZ formula for the three-point function in the Liouville theory~\cite{DO:1994},\cite{ZZ:1996}. 

\paragraph{Motivations} One of the approaches to DOZZ formula \cite{Teschner:1995Liouville} is based on the interplay of two Virasoro algebras with the same central charge in Liouville theory, namely chiral and antichiral algebras. Similarly, in the proof of Theorem~\ref{Th:main} we use interplay of $\widehat{\mathfrak{sl}}(2)_{k+1}$ and $\mathrm{Vir}_{c}$ symmetries in which relation between \(k\) and \(c\) is very important. It appears that right side of \eqref{eq:coset algebras} is a simplest case of corner vertex algebras \(\mathcal{D}_{n,k}^\psi\) \cite{Creutzig:2020vertex}, \cite{Creutzig:2021qft}. We hope that methods used in this paper can be used to study more general corner algebras, at least corresponding to \(\mathfrak{sl}(2)\). Another possible generalization is the study of generic \(\mathfrak{sl}(2)\) coset of the form \(\mathfrak{sl}(2)_{k_1}\oplus \mathfrak{sl}(2)_{k_1}/\mathfrak{sl}(2)_{k_1+k_2}\). In the last case, the coset algebra is $\mathcal{D}(2|1,\alpha)$ $W$-algebra~\cite{Feigin:2001coset}.

Our first motivation was the study of relations on conformal blocks. Conformal blocks of Virasoro and  \(\widehat{\mathfrak{sl}}(2)\) algebras are determined by the symmetry up to three-point functions. Therefore, having found three-point functions we can now write relations on conformal blocks, which have the form
\begin{equation}\label{eq:Psi-PsiF}
	\Psi_k(\dots)=\sum C \cdot \Psi_{k+1}(\dots) \cdot \mathrm{F}_c(\dots).
\end{equation}
Here \(\Psi_k\) denotes $\widehat{\mathfrak{sl}}(2)_k$ conformal blocks, \( \mathrm{F}_c\) denotes \(\mathrm{Vir}_{c}\) conformal block, \(\dots\) stands for bunch of arguments on which conformal block depends and \(C\) is a coefficient which comes from three-point functions. We omitted some elementary function that can come from $\widehat{\mathfrak{sl}}(2)_1$ blocks. 

Due to AGT correspondence the relation \eqref{eq:Psi-PsiF} has interpretation in terms of 4d supersymmetric gauge theory. Namely, the function \(\mathrm{F}_c\) is equal (again, up to elementary function) to Nekrasov partition function for \(SU(2)\) gauge theory and function \(\Psi_k\) is equal to Nekrasov partition function with surface defect. The relation \eqref{eq:Psi-PsiF} itself takes the meaning of the blowup relation \cite{Nakajima:2005instanton}, to be more precise this is the blowup relation with the presence of surface defect suggested in \cite{Nekrasov:2023blowups}, \cite{Jeong:2020}. Hence our Theorem~\ref{Th:main} gives the proof of blowup relations with surface defect (for \(SU(2)\) group). Such method of the proof of blowup relations follows \cite{Bershtein2016coupling}, see also \cite{Arakawa:2022urod}.

It is worth mentioning that coset (or blowup) relations \eqref{eq:Psi-PsiF} have remarkable application in the theory of isomonodromic deformations as was shown in \cite{Nekrasov:2023blowups}, \cite{Jeong:2020}. In CFT language this idea can be restated as a \(k\rightarrow \infty\) limit of algebras~\eqref{eq:coset algebras} or conformal blocks~\eqref{eq:Psi-PsiF}. The algebra \(\widehat{\mathfrak{sl}}(2)_{k+1}\) in this limit becomes classical, so we get vertex algebra with big center \cite{Feigin:2017Extensions}, \cite{Creutzig:2021qft}, \cite{Feigin:2022vertex}. The spectrum of the center is the space of \(\mathfrak{sl}(2)\) connections. The conformal blocks \(\Psi_k\) satisfy Knizhnik-Zamolodchikov equations which in classical limit goes to isomonodromic deformation equations \cite{Reshetikhin:1992knizhnik}, \cite{Harnad:1994quantum}, \cite{Nekrasov:2023blowups}. Hence one got a proof of Kyiv formula for isomonodromic tau function \cite{Gamayun:2012Conformal}. 

Note also one more application of our main theorem about three-point functions. Namely, such quantities can be computed using free field  (Wakimito \cite{Wakimoto:1986}) realization of \(\widehat{\mathfrak{sl}}(2)_k\). This leads to Selberg-type integrals, and as a corollary of the Theorem \ref{Th:main} we found new integrals of this type. A particular case of our integrals is a Forrester integral which was conjectured in \cite{Forrester:1995normalization} and proven in \cite{Petrov:2014}. 

\paragraph{Results and Plan of the paper} In Section \ref{sec:sl(2)} we recall standard definition and properties of $\widehat{\mathfrak{sl}}(2)$ and Virasoro algebras, their representations, and vertex operators. For $\widehat{\mathfrak{sl}}(2)$ vertex operators, we use both the definition by commutation relations~\cite{Awata:1992} and the definition by coinvariants, where the space depends on the choice of the Borel subalgebra at each point~\cite{Feigin:1994}.

In section \ref{sec:coset} we revisit GKO coset construction. The first result is an explicit formula for the highest weight vectors \(u_l(\lambda)\), see Theorem~\ref{Th:FormUm}. This formula is written in free field realization, i.e. we consider Wakimoto representation of $\widehat{\mathfrak{sl}}(2)_k$. \footnote{While this paper was in preparation, the authors became aware of the paper \cite{Hadasz:2023decomposition} which has some overlap with our results. In particular, a formula for the vectors \(u_l(\lambda)\) in free field realization was found there, which is different from our formula. It would be interesting to compare these formulas.} The existence of the formula in free field realization is a standard feature in the representation theory of vertex algebras, c.f. formulas for Virasoro singular vectors in terms of Jack symmetric functions \cite{Mimachi:1995} or formulas for the highest weight vector in decomposition \cite{Belavin:2013instanton}.

The next result is Theorem \ref{Th:Norm} in which the norms of vectors \(\|u_l(\lambda)\|\) are computed. The computation is based on the recursion which is derived using degenerate fields \(I(z), J(z)\). In terms of left side of isomorphism \eqref{eq:coset algebras} these operators are products of spin \(\frac12\) degenerate fields for \(\widehat{\mathfrak{sl}}(2)_{1}\) and \(\widehat{\mathfrak{sl}}(2)_{k}\), while in terms of the right side of \eqref{eq:coset algebras} these operators are products of identity operators for \(\widehat{\mathfrak{sl}}(2)_{k+1}\) and $\Phi_{2,1}$ vertex operators for $\mathrm{Vir}_c$.

Section \ref{sec:main} is devoted to the proof of the Theorem~\ref{Th:main}. The idea is to consider a four-point conformal block with an insertion of a degenerate field \(b(x,z)\). Essentially, we mimic a standard approach to study theories with chiral and antichiral symmetries \cite{zamolodchikov:1989}, \cite{Teschner:1995Liouville}, \cite{Teschner:1997}, in purely chiral setting this was also used in \cite{Bershtein:2015Bilinear}. The inserted field \(b(x,z)\) in terms of left side of \eqref{eq:coset algebras} is a product of operator of spin \(\frac12\) for \(\widehat{\mathfrak{sl}}(2)_{1}\) and identity for \(\widehat{\mathfrak{sl}}(2)_{k}\), while in terms of the right side of \eqref{eq:coset algebras} is a product of spin \(\frac12\) for \(\widehat{\mathfrak{sl}}(2)_{k+1}\) and $\Phi_{1,2}$ for $\mathrm{Vir}_c$.

In the end of Section~\ref{sec:main} we discuss relations on conformal blocks and blowup relations, see formulas \eqref{eq:blowup} and \eqref{eq:blowup tor}.

Section \ref{sec:Kyiv} is included for completeness. We deduce Kyiv formulas for the tau function from the coset (or blowup) relations \eqref{eq:Psi-PsiF} closely following \cite{Nekrasov:2023blowups}, \cite{Jeong:2020}. We restrict ourselves to the case of Painlev\'e \(\mathrm{III}_3\). 

Finally, in Section \ref{sec:selbint} we use free field realizations and Theorem~\ref{Th:main} for the computation of Selberg-type integrals. We first prove Theorem~\ref{Th:U nu n} which is an operator analog of the formula for the highest vector \(u_l(\lambda)\). The integrals are computed in Theorem~\ref{Th:Selberg}.

\paragraph{Acknowledgements} M.B. is grateful to Nikita Nekrasov for explanations about relation between blowup relations with surface defect and Kyiv formula in 2013, which eventually lead to this work. We are grateful to  T.~Creutzig, N.~Genra, M.~Lashkevich, A.~Litvinov, H.~Nakajima, N.~Nekrasov, M.~Noumi,  F.~Petrov, A.~Shchechkin, J.~Teschner,  A.~Zamolodchikov for useful discussions. The results of the paper were reported at conferences and seminars in Tokyo (October 2022, January 2023), Edinburgh (September 2023), Trieste (December 2023), Moscow/Zoom (December 2023), Hamburg/Toronto/Zoom (March 2024); we are grateful to the organizers and participants for their interest and remarks.

The work of M.B is partially supported by the European Research Council under the European Union’s Horizon 2020 research and innovation programme under grant agreement No 948885. M.B. is very grateful to Kavli IPMU  and especially Y.~Fukuda, M.~Kapranov, K.~Kurabayashi, H.~Nakajima, A.~Okounkov, T.~Shiga, K.~Vovk, for the hospitality during 2022--2023 years.

\section{$\widehat{\mathfrak{sl}}(2)$ and Virasoro algebras} \label{sec:sl(2)}
\subsection{Definition}
Let \(\widehat{\mathfrak{sl}}(2)=\mathfrak{sl}(2)\otimes \mathbb{C}[[t,t^{-1}]\oplus \mathbb{C}K\) denotes the central extension of the of the algebra of Laurent series with coefficients in $\mathfrak{sl}_2$. It has topological basis \(e_m=e\otimes t^m\), \(h_m=h\otimes t^m\), \(f_m=f\otimes t^m\) \((m\in \mathbb{Z})\), and \(K\) with commutation relations 
\begin{subequations}
	\begin{align}
		[e_m, f_l] &= h_{m + l} + m\delta_{m + l, 0} K,\\
		[h_m, h_l] &= 2m\delta_{m + l, 0} K,\\
		[h_m, e_l] &= 2e_{m+l},\\
		[h_m, f_l] &=-2f_{m+l}.
	\end{align}
\end{subequations}
%
\noindent Here and below all commutators which are not written are equal to zero. In particular, $K$ is a central element. 

It is convenient to  consider currents
\begin{equation}
    e(z) = \sum_{l\in \mathbb{Z}} e_l z^{-l-1}, ~~~ f(z) = \sum_{l\in \mathbb{Z}} f_l z^{-l-1}, ~~~ h(z) = \sum_{l\in \mathbb{Z}} h_l z^{-l-1}.
\end{equation}
\begin{define}
	Verma module $\mathcal{M}_{\lambda,k}$  is a module over algebra $\widehat{\mathfrak{sl}}(2)$ which is generated freely by $e_{n}, h_{n}, f_{n+1}$, $(n \in \mathbb{Z}_{<0})$ acting on a highest weight vector $v_{\lambda,k}$  such that
	\begin{align}
 	   \label{eq:efh v}&f_n v_{\lambda,k}= e_{n-1} v_{\lambda,k} = h_n v_{\lambda,k} = 0,\; \forall n>0;
 	     \\
 	   \label{eq:h0K v}
 	   &h_0 v_{\lambda,k}= \lambda v_{\lambda,k}, ~~ K v_{\lambda,k} = k v_{\lambda,k}.
	\end{align}
\end{define}
Let \(\widehat{\mathfrak{b}}=\langle K, e_{n}, h_{n}, f_{n+1}| n \in \mathbb{Z}_{\ge 0}  \rangle \) denotes Borel subalgebra in \(\widehat{\mathfrak{sl}}(2)\) (note that \(K, e_{n}, h_{n}, f_{n+1}\) form a \emph{topological} generating set of \(\widehat{\mathfrak{b}}\), i.e. infinite sums are allowed). The formulas \eqref{eq:efh v}, \eqref{eq:h0K v} define one dimensional representation of \(\widehat{\mathfrak{b}}\), which we denote by \(\mathbb{C}_{\lambda,k}\). The Verma module can be also written as an induced module $\mathcal{M}_{\lambda,k}=\operatorname{Ind}^{\widehat{\mathfrak{sl}}(2)}_{\widehat{\mathfrak{b}}}\mathbb{C}_{\lambda,k}$.
\begin{remark}
	We will often use $\kappa = k + 2$ instead of $k$.
\end{remark}
\begin{remark}
	We say that affine algebra $\widehat{\mathfrak{sl}}(2)_k$ acts on the module $V$  if the central elements $K$ acts by $k\in \mathbb{C}$ on $V$.
\end{remark}
	The Verma module $\mathcal{M}_{\lambda,k}$ has unique irreducible quotient. We denote it by $\mathcal{L}_{\lambda,k}$ .
\begin{theorem} \phantomsection \cite{KacKazhdan:1979} \label{Th:KacKazhdan}
	Verma module $\mathcal{M}_{\lambda,k}$  is irreducible iff for any $m,n>0$
	\begin{subequations}
	\label{eq:generic}	
		\begin{align}
		    m\lambda + (m - 1)(k - \lambda) + (2m - 1) - n&\neq0\\
		    (m - 1)\lambda + m (k - \lambda) + (2m - 1) - n&\neq0\\
		    k + 2 \neq 0
		\end{align}
	\end{subequations}
\end{theorem}
We will call pair $(\lambda,k)$ generic if \(\lambda\) and \(k\) are linearly independent over \(\mathbb{Q}\). It is typically sufficient to require that they do not satisfy conditions \eqref{eq:generic}. Similarly, we call \(k\) generic if it is irrational.

\subsection{Virasoro algebra and its modules}
\begin{define}
	A Virasoro algebra is Lie algebra with basis $L_n$ ($n\in\mathbb{Z}$), $C$ and commutation relations
	\begin{equation}\label{eq:Virasaro rel}
	    [L_n,L_m] = (n - m)L_{n+m} + \frac{C(n^3-n)}{12}\delta_{n+m,0}.
	\end{equation}
\end{define}
\noindent In particular, the generator $C$ is central.

It is convenient to consider a current
\begin{equation}
    L(z) = \sum_{n\in\mathbb{Z}} L_n z^{-n-2}.
\end{equation}
\begin{define}
	Verma module $\mathbb{M}_{P,b}$ is module over Virasoro algebra which is freely generated by $L_{n}$, \(n<0\) acting on a vector highest weight vector vector $v_{P,b} = |\Delta(P,b)\rangle$ such that
	\begin{align}
	    &L_n v_{P,b} = 0,~~ \forall n>0;\\
	    &L_0 v_{P,b} = \Delta(P,b) v_{P,b}, ~~~ C v_{P,b} = c(b) v_{P,b};
	\end{align}
	where $\Delta(P,b) = \frac{1}{4} \left(b+b^{-1}\right)^2-P^2$ and $c(b) = 1 + 6 \left(b+b^{-1}\right)^2 $.
\end{define}
\begin{define}
	A vector $u \in \mathbb{M}_{P,b}$ is called singular if $L_{n} u = 0$, $\forall n>0$.
\end{define}
\begin{theorem}[\cite{FeiginFuchs:1990}\cite{kac:1990}] \phantomsection \label{Th:Delta_mn}
	The Verma module $\mathbb{M}_{P,b}$ over Virasoro algebra is irreducible iff 
    $P \neq P_{m,n}(b)$ for any $m,n \in \mathbb{Z}_{>0}$. Here
\begin{equation}
    P_{m,n}(b) = \frac{m b^{-1} + n b}{2}.
\end{equation}
 If $P = P_{m,n}(b)$ then there is a singular vector $u \in \mathbb{M}_{P,b}$ such that 
 \begin{equation}
     L_0 u = (\Delta(P_{m,n}(b),b) + m n) u.
 \end{equation}
\end{theorem}

We will use notation $\Delta_{m,n} = \Delta(P_{m,n}(b),b)$.

\begin{example}\phantomsection \label{Exa:Phi 12 21}
	An important case of the reducible Verma module is $\mathbb{M}_{P_{2,1},b}$. The explicit formula for singular vector reads
	\begin{equation}
	 u_{2,1} = (L_{-1}^2 + b^{-2}L_{-2})v_{P_{2,1},b}.
	\end{equation}
	In case of $P = P_{1,2}$ the formula is similar (obtained by $ b \leftrightarrow b^{-1}$ symmetry)
	\begin{equation}
	 u_{1,2} = (L_{-1}^2 + b^{2}L_{-2})v_{P_{1,2},b}.
	\end{equation}
\end{example}

\subsection{Sugawara construction}
This construction provides an action of the Virasoro algebra on $\widehat{\mathfrak{sl}}(2)$ highest weight modules, such as $\mathcal{M}_{\lambda,k}$ and $\mathcal{L}_{\lambda,k}$. 

We will use normal ordering of fields. See, e.g. \cite[Ch. 2]{FB:2004}
\begin{define}
	Let  $A(z) =\sum_n A_n z^{-n-1}, B(w)=\sum_n B_n w^{-n-1}$ be two fields. \emph{Normally ordered product} is defined by a formula
	\begin{equation}
	    :\!A(z) B(w)\!: = A_+(z) B(w) + B(w) A_{-}(z),
	\end{equation}
	where $A_+(z) = \sum_{m\ge 0} A_{-m-1} z^{m}$ and  $A_-(z) = \sum_{m < 0} A_{-m-1} z^{m}$.
\end{define}
\begin{theorem} \phantomsection
	Series coefficients of 
	\begin{equation}\label{Sug}
	  L(z) =\sum L_n z^{-n-2} =\frac{1}{2(k+2)} :\!e(z)f(z) + f(z)e(z) + \frac{1}{2}h^2(z)\!:  
	\end{equation}
	acting on any highest weight representation of the level $k$  satisfy Virasoro algebra relations~\eqref{eq:Virasaro rel} with central charge $c = \frac{3k}{k + 2}$.
\end{theorem}
In particular, the theorem works for modules $\mathcal{M}_{\lambda,k}$ and $\mathcal{L}_{\lambda,k}$. The Virasoro algebra defined by formula (\ref{Sug}) is called the Sugawara Virasoro algebra; see \cite[sec. 12]{kac:1990}. We define the character using the Sugawara operator \(L_0\).
\begin{proposition}
	The character of the Verma module is 
	\begin{equation}
		\text{Tr}_{\mathcal{M}_{\lambda, k}}(q^{L_0} z^{h_0}) =  \frac{q^{\frac{\lambda(\lambda + 2)}{4(k + 2)}}z^{\lambda}}{\prod_{m \in \mathbb{Z}_{>0}} (1 - q^m)(1 - q^{m - 1} z^2)(1 - q^m z^{-2} )}.
	\end{equation}
\end{proposition}

\subsection{Integrable modules on level $1$} \label{ssec:level 1}
\begin{define}
	Heisenberg algebra is a Lie algebra with basis $a_n$ ($n\in\mathbb{Z}$) and $\mathbf{1}$ with commutation relations
	\begin{equation}
		[a_n, a_m] = n\delta_{n+m,0}\mathbf{1}. 
	\end{equation}
\end{define}
The operator $\mathbf{1}$ will act by \(1\) on any representation that we consider. It is convenient to introduce bosonic field
\begin{equation}\label{eq:bosonic field k=1}
    \phi(z) = \sum_{n\neq 0} \frac{1}{-n}a_n z^{-n} +a_0 \log(z) + Q,
\end{equation}
where $Q$ is an additional generator such that $[a_n, Q] = \delta_{n,0}$.

\begin{define}
	The Fock module $\mathcal{F}_\alpha$ is a Heisenberg algebra module freely generated by \(a_n\), \(n<0\) acting on a highest weight vector $|\alpha\rangle = v_{\frac{\alpha}{\sqrt{2}}}$ such that
	\begin{equation}
	    a_{n} |\alpha\rangle  = 0,~ \text{if} ~ n>0;~~~   a_{0} |\alpha\rangle  = \alpha |\alpha\rangle .
	\end{equation}
\end{define}
There is a simple realization of $\widehat{\mathfrak{sl}}(2)_1$ using Heisenberg algebra.
\begin{theorem}[\cite{Frenkel:1981}]\phantomsection\label{Th:Frenkel}
	The direct sums of the Fock modules $\bigoplus_{n\in \mathbb{Z}} \mathcal{F}_{\sqrt{2}n}$ and $\bigoplus_{n\in \mathbb{Z}+\frac{1}{2}}  \mathcal{F}_{\sqrt{2}n}$ have a structure of $\widehat{\mathfrak{sl}}(2)_1$ modules given by formulas
	\begin{subequations}\label{eq:ehf level 1}
		\begin{align}
		    e(z)&=\colon\!\!\exp(\sqrt{2}\phi(z))\!\colon,\\
		    h(z)&=\sqrt{2}\partial\phi(z),\\
		    f(z)&=\colon\!\!\exp(-\sqrt{2}\phi(z))\!\colon.
		\end{align}
	\end{subequations}
	Moreover, these sums are irreducible $\widehat{\mathfrak{sl}}(2)_1$-modules
	\begin{equation}\label{int1}
	    \mathcal{L}_{0,1} = \bigoplus_{n\in \mathbb{Z}} \mathcal{F}_{\sqrt{2}n},~~~   \mathcal{L}_{1,1} = \bigoplus_{n\in \mathbb{Z}+\frac{1}{2}} \mathcal{F}_{\sqrt{2}n}.
	\end{equation}
\end{theorem}
Here and below in exponents of Heisenberg algebra we use bosonic normal ordering.

The modules $\mathcal{L}_{0,1}$, $\mathcal{L}_{1,1}$ are integrable \cite{kac:1990}. The highest weight vectors $v_n \in \mathcal{F}_{\sqrt{2}n}$, \(n \in \frac12\mathbb{Z}\) in components of decompositions \eqref{int1}  are called extremal vectors. It follows from the Theorem~\ref{Th:Frenkel}  that \(v_0, v_{1/2}\) are highest weight vectors of \(\mathcal{L}_{0,1}\) and \(\mathcal{L}_{1,1}\) correspondingly. Other extremal vectors can be found by the formulas
\begin{equation}
	e_{2n+1}v_n=v_{n+1}, \quad f_{2n-1}v_n=v_{n-1}.
\end{equation}
The following formulas for characters follows from Theorem \ref{Th:Frenkel}.

\begin{corollary}
	We have
	\begin{equation}\label{eq:char level 1}
		\text{Tr}_{\mathcal{L}_{0, 1}} (q^{L_0} x^{h_0}) = \frac{\sum_{n \in \mathbb{Z}} q^{n^2} x^{2n}}{\prod_{m = 1}^{\infty} (1 - q^m)}
		;~~~ \text{Tr}_{\mathcal{L}_{1, 1}} (q^{L_0} x^{h_0}) = \frac{\sum_{n \in \mathbb{Z}+\frac{1}{2}} q^{n^2} x^{2n}}{\prod_{m = 1}^{\infty} (1 - q^m)}.
	\end{equation}
\end{corollary}
\subsection{Wakimoto module} \label{ssec:Wakimoto}
\begin{define}
 Let us introduce algebra $Heis_{\alpha,\beta,\gamma}$ with basis $\beta_n,\gamma_n,\alpha_n$ (\(n \in \mathbb{Z}\)), $\mathbf{1}$ and commutation relations
\begin{equation}
    [\beta_n,\gamma_m]=\delta_{n+m,0}\mathbf{1},\quad [\alpha_n,\alpha_m]=n \delta_{n+m,0}\mathbf{1}.
\end{equation}
\end{define}
\noindent  Recall that all other commutators are equal to zero. The operator $\mathbf{1}$ will act by \(1\) on any representation which we consider.

It is convenient to consider currents
\begin{equation}
    \gamma(z)=\sum_{n\in\mathbb{Z}} \gamma_n z^{-n},~~ \beta(z)=\sum_{n\in\mathbb{Z}} \beta_n z^{-n-1},~~ \partial \varphi(z)=\sum_{n\in\mathbb{Z}} \alpha_n z^{-n-1}.
\end{equation}
\begin{define}
	A module $\mathcal{F}_{\lambda, k}$ is module over algebra  $Heis_{\alpha, \beta, \gamma}$ which is freely generated by $\alpha_n, \beta_{n}, \gamma_{n+1}$ ($n<0$) acting on highest weight vector $v$ such that
	\begin{equation}
	    \beta_{n-1} v=\gamma_n v =a_n v =0,\; n>0;~~ \quad a_0 v=\frac{\lambda}{\sqrt{2(k+2)}} v. 
	\end{equation}
\end{define}

\begin{proposition}[\cite{Wakimoto:1986}]
	There is an action of $\widehat{\mathfrak{sl}}(2)_k$ on $\mathcal{F}_{\lambda, k}$ given by the formulas
	\begin{equation}\label{eq:wakimoto}
		\begin{aligned}
		e(z)=\beta(z),\quad h(z)=-2:\! \gamma(z)\beta(z)\!: + \sqrt{2(k + 2)} \partial \varphi(z),\\
		f(z)=-:\!\gamma^2(z)\beta(z)\!: + \sqrt{2(k + 2)}\partial\varphi(z) \gamma(z)+ k\partial\gamma(z).
		\end{aligned}
	\end{equation}
\end{proposition}
  The $\widehat{\mathfrak{sl}}(2)_k$-module $\mathcal{F}_{\lambda, k}$ is called Wakimoto module. See also \cite[Ch.11--12]{FB:2004}. The following proposition easily follows from the equality of the characters.
\begin{proposition}\phantomsection \label{Prop:wakverm}
	The Wakimoto module $\mathcal{F}_{\lambda, k}$ is isomorphic to the Verma module $\mathcal{M}_{\lambda,k}$ for generic $\lambda, k$.
\end{proposition}

\subsection{Vertex operators} \label{ssec:operstatecorr}
There are (at least) two standard approaches to the definition of vertex operators in conformal field theory. The first one is based on the definition of conformal blocks via coinvariants and operator-state correspondence. In the second approach vertex operators are defined using commutation relation with the algebra generators. We will use both approaches.
We start from the first one, mainly following the paper \cite{Feigin:1994}. We restrict ourselves to the genus zero conformal block. We fix a global coordinate \(t\) on \(\mathbb{P}^1\).


\subsubsection{Generic representations} We will need a slight generalization of the representations considered above. Let \(B\subset SL(2)\) be a Borel subgroup of invertible upper triangular matrices. Then the flag manifold \(\mathbb{P}^1=SL(2)/B\) parametrize all Borel subalgebras in the Lie algebra \(\mathfrak{sl}(2)\). To be more precise for any $x\in \mathbb{A}^1 \subset\mathbb{P}^1$ we can consider Borel subalgebra $\mathfrak{b}_{x}  \subset \mathfrak{sl}(2)$ given by 
\begin{equation}
    \mathfrak{b}_{x} = \left\{
    \begin{pmatrix}
        c_1 - x c_2 & c_2 \\
        2xc_1-x^2 c_2 & x c_2 - c_1
    \end{pmatrix}|c_1, c_2\in \mathbb{C}\right\}.
\end{equation}
This subalgebra is conjugated from \(\mathfrak{b}=\langle h,e\rangle\) by the \(\re^{x f} \in SL(2)\). Let \(h_x,e_x\) be generators of \(\mathfrak{b}_x\) such that \([h_x,e_x]=2e_x\), \(e_x \in [\mathfrak{b}_x,\mathfrak{b}_x]\). For example, for one can take \(h_x=\re^{x f} h  \re^{-x f}=h+2xf\), \(e_x=\re^{x f} e  \re^{-x f}=e-xf-x^	2f\) for \(x \neq \infty\) and \(h_\infty=-h\), \(e_\infty=-f\) for \(x = \infty\).

Let \(z \in \mathbb{P}^1\). Consider the Lie algebra $\mathfrak{sl}(2)\otimes\mathbb{C}[[t - z, (t - z)^{-1}]$  of \emph{Laurent series} with coefficients in \(\mathfrak{sl}(2)\).  Let us denote by $ \widehat{\mathfrak{sl}}(2)_{z}$ its central extension with the commutator given by 
\begin{equation}
	[x \otimes f(t-z), y \otimes g(t-z)]=[x,y]\otimes f(t-z)g(t-z)+K \operatorname{Tr}(xy) \operatorname{Res}_{t=z} (g df). 
\end{equation}
We denote $\widehat{\mathfrak{b}}_{x,z} = \mathfrak{b}_{x}\oplus \mathfrak{sl}(2)\otimes(t - z)\mathbb{C}[[t - z]]\oplus \mathbb{C}K\subset \widehat{\mathfrak{sl}}(2)_{z}$ the Borel subalgebra in $ \widehat{\mathfrak{sl}}(2)_{z}$. Let $\mathbb{C}_{\lambda,k}$ denotes the $1$-dimensional $\widehat{\mathfrak{b}}_{x,z}$-module 
generated by vector \(1_{\lambda,k}\) such that 
\begin{equation}
	h_x 1_{\lambda,k}=\lambda 1_{\lambda,k},\; K 1_{\lambda,k}=k1_{\lambda,k},\quad \big( \mathbb{C} e_x \oplus   \mathfrak{sl}(2)\otimes(t - z)\mathbb{C}[[t - z]]\big)1_{\lambda,k}=0.
\end{equation}
Let us define $\mathcal{M}_{\lambda, k}(x,z)=\operatorname{Ind}^{\widehat{\mathfrak{sl}}(2)_{z}}_{\widehat{\mathfrak{b}}_{x,z}}\mathbb{C}_{\lambda,k}$. 

We have an isomorphism of Lie algebras
\begin{equation}
	\alpha_{x,z}:
	\widehat{\mathfrak{sl}}(2) \rightarrow \widehat{\mathfrak{sl}}(2)_z,
\end{equation} 
which maps $\widehat{\mathfrak{b}}$ to $\widehat{\mathfrak{b}}_{x,z}$. Such isomorphism is not unique, one of the possible choices which we will mainly use in the paper is 
\begin{equation}
	\alpha^{(0)}_{x,z}(e_n)=(e-xh-x^2f)\otimes (t-z)^n,\quad \alpha^{(0)}_{x,z}(h_n)=(h+2xf)\otimes (t-z)^n, \quad \alpha^{(0)}_{x,z}(f_n)=f\otimes (t-z)^n.
\end{equation}
Of course these formulas do not work if \(x\) or \(z\) are equal to infinity. So in the neighbourhood of \((\infty,\infty)\) we use another isomorphism
\begin{equation}
	\alpha^{(\infty)}_{y,w}(e_n)=(y^2e+yh-f)\otimes (s-w)^n,\; \alpha^{(\infty)}_{y,w}(h_n)=(-h-2ye)\otimes (s-w)^n, \; \alpha^{(\infty)}_{y,w}(f_n)=-e\otimes (s-w)^n.
\end{equation}
Here \(s=t^{-1}\), \(w=z^{-1}\) and \(y=-x^{-1}\).

We use the same notation for the corresponding maps between Verma modules $\alpha_{x,z}: \mathcal{M}_{\nu,k}\rightarrow \mathcal{M}_{\nu,k}(x,z)$ mapping $v_{\nu, k}\mapsto v_{\nu, k}(x,z)$. 

Let \(\mathbb{C}[\mathbb{P}^1{\setminus} \{z_1,z_2,z_3\}]\)  denotes the space of rational functions in \(t\) with possible poles at \(z_1, z_2, z_3\). We have a Lie algebras homomorphism
\begin{equation}
    \mathfrak{sl}(2)\otimes \mathbb{C}[\mathbb{P}^1{\setminus} \{z_1,z_2,z_3\}] \rightarrow \widehat{\mathfrak{sl}}(2)_k\oplus \widehat{\mathfrak{sl}}(2)_k\oplus\widehat{\mathfrak{sl}}(2)_k,
\end{equation}
which map any element to the direct sum of series expansions at points \(z_i\). 

The following proposition is standard (see e.g. \cite[Lemma 3.1]{Feigin:1994}).
\begin{proposition}\phantomsection \label{Prop:coinv generic}
	Let \(x_1,x_2,x_3 \in \mathbb{P}^1\) and \(z_1,z_2,z_3 \in \mathbb{P}^1\) be tuples of different points. 
	There is a unique up to constant homomorphism of $\mathfrak{sl}(2)\otimes \mathbb{C}[\mathbb{P}^1{\setminus} \{\vec{z}\}]$-modules
	\begin{equation}
	     \mathtt{m}_k \colon  \mathcal{M}_{\lambda,k}(x_1,z_1) \otimes  \mathcal{M}_{\nu, k}(x_2,z_2) \otimes  \mathcal{M}_{\mu, k}(x_3,z_3) \rightarrow \mathbb{C}.
	\end{equation}  
	Here \(\vec{z}=\{z_1,z_2,z_3\}\) and \(\vec{x}=\{x_1,x_2,x_3\}\). Moreover, the homomorphism is uniquely determined by its value on the product of highest vectors \(v_{\lambda,k}\otimes v_{\mu,k} \otimes v_{\nu,k}\).
\end{proposition}
We define map $\mathtt{m}_k(\vec{x},\vec{z})$ as a composition \(
\mathtt{m}_k(\vec{x},\vec{z})=\mathtt{m}_k\circ(\alpha_{x_1,z_1}\otimes \alpha_{x_2,z_2}\otimes \alpha_{x_3,z_3}) \)
\begin{equation}
	\mathtt{m}_k(\vec{x},\vec{z}) \colon \mathcal{M}_{\lambda, k}\otimes  \mathcal{M}_{\nu, k}\otimes  \mathcal{M}_{\mu, k}  \rightarrow \mathbb{C}.
\end{equation}   
Using \(PGL(2)\) action on \(\mathbb{P}^1\) (global conformal transformations) we can  assume that \(\vec{z}=\{0,z,\infty\}\), and
similarly we can assume that \(\vec{x}=\{0,x,\infty\}\) using \(SL(2)\) action on the space of Borel subalgebras \(\mathbb{P}^1=Sl_2/B\). We write simply 
\begin{equation}
	\mathtt{m}_k(x,z)=\mathtt{m}_k\circ (a^{(0)}_{0,0}\otimes a^{(0)}_{x,z}\otimes a^{(\infty)}_{0,0}).
\end{equation}
Note that notation \(a^{(\infty)}_{0,0}\) means that \(y=w=0\), i.e. \(x=z=\infty\).

For generic $\mu$ the  map \(\mathtt{m}_k(x,z)\) could be rewritten as 
    \begin{equation}
    	\mathtt{Y}_k(x,z): \mathcal{M}_{\nu,k} \rightarrow \operatorname{Hom}_{\mathbb{C}}( \mathcal{M}_{\lambda,k}, 
    	\overline{\mathcal{M}_{\mu,k}}).
    \end{equation} 
Here and below \(\overline{\mathcal{M}_{\mu,k}}\) stands for completion of \(\mathcal{M}_{\mu,k}\) with respect to natural gradation.
    
Consider $\mathcal{V}_{\nu}(x,z) =  \mathtt{Y}(v_{\nu, k}|x,z)$. 
\begin{proposition}\phantomsection \label{Prop: Awata Yamada}
	The operator $\mathcal{V}_{\nu}(x,z)$ enjoys commutation relations
	\begin{subequations}\label{eq:Vnu comm rel}
		\begin{align}
		        [f_n,\mathcal{V}_{\nu}(x,z)] &= z^n\partial_x \mathcal{V}_{\nu}(x,z),\\
		        [h_n,\mathcal{V}_{\nu}(x,z)] &= z^n(- 2 x\partial_x + \nu)\mathcal{V}_{\nu}(x,z),\\
		        [e_n,\mathcal{V}_{\nu}(x,z)] &= z^n(-  x^2\partial_x + \nu x )\mathcal{V}_{\nu}(x,z).    
		\end{align}
	\end{subequations}
\end{proposition}
This Proposition follows from a direct computation (cf. \cite[Prop. 3.1]{Feigin:1994}). On the other hand this proposition can be considered as a definition of the vertex operator, (see \cite{Awata:1992}).

In order to write down \(\mathtt{Y}(v|x,z)\) for arbitrary \(v \in  \mathcal{M}_{\nu,k}\) we will need notations
\begin{equation}
    e(x,z) = \re^{x f_0} e(z)  \re^{-x f_0} =  e(z) - x h(z) - x^2 f(z),
\end{equation}
\begin{equation}
    h(x,z) = \re^{x f_0} h(z)  \re^{-x f_0} =  h(z) + 2 x f(z), ~~   f(x,z) = \re^{x f_0} f(z)  \re^{-x f_0} = f(z).
\end{equation}
\begin{proposition}
    The map  $\mathtt{Y}_k(x,z):\mathcal{M}_{\nu,k} \rightarrow \text{End}( \mathcal{M}_{\lambda,k}, \mathcal{M}_{\mu,k})$
    is defined by formula
    \begin{multline}\label{eq:operator-state}
              \mathtt{Y}_k(e_{-k_t}\dots e_{-k_2}e_{-k_1}h_{-j_s}\dots h_{-j_2}h_{-j_1}f_{-i_r}\dots f_{-i_2}f_{-i_1}f_0^m
              \, v_{\nu, k}; x,z) 
              \\
              = ~  \frac{ :\!\partial^{k_t-1}_z e(x,z)\dots\partial^{i_2-1}_z f(x,z)\partial^{i_1-1}_z f(x,z)\!\partial^m_x\mathcal{V}_{\nu}(x,z)\!:}{\prod^t_{q=1}(k_q - 1)!\prod^s_{q=1} (j_q - 1)!\prod^r_{q=1}(i_q - 1)!}.
    \end{multline}
\end{proposition}

\subsubsection{Degenerate representations} Assume now that \(k\) is generic  and \(\nu \in \mathbb{Z}_{\ge 0}\). Then the representation \(\mathcal{M}_{\nu,k}\) has singular vector \(f_0^{\nu+1}v_{\nu,k}\). Denote by $\mathcal{L}_{\nu,k}$ the irreducible quotient. Since \(k\) is generic we have $\mathcal{L}_{\nu,k}\simeq \operatorname{Ind}^{\widehat{\mathfrak{sl}}(2)}_{\mathfrak{sl}(2)[[t]]\oplus \mathbb{C}K}\mathcal{L}_\nu$ where \(\mathcal{L}_\nu\) is \(\nu+1\) dimensional representation of \(\mathfrak{sl}(2)\). 

Similarly to the construction above we can define modules \(\mathcal{L}_{\nu,k}(x,z)\). Contrary to generic case above the modules \(\mathcal{L}_{\nu,k}(x,z)\) do not depend on \(x\) as $\widehat{\mathfrak{sl}}(2)_{z}$ modules (up to isomorphism). The analogue of Proposition \ref{Prop:coinv generic} now states that 	there is a unique up to constant homomorphism of $\mathfrak{sl}(2)\otimes \mathbb{C}[\mathbb{P}^1{\setminus} \{\vec{z}\}]$ modules
\begin{equation}
	\mathtt{m}_k \colon  \mathcal{M}_{\lambda,k}(x_1,z_1) \otimes  \mathcal{L}_{\nu, k}(x_2,z_2) \otimes  \mathcal{M}_{\mu, k}(x_3,z_3) \rightarrow \mathbb{C}.
\end{equation}  
if and only if \(\lambda,\mu,\nu\) satisfy fusion rule. The corresponding fusion rule reads 
\begin{equation}
	-\nu \le \lambda-\mu \le \nu, \text{ and } \lambda-\mu+\nu \in 2 \mathbb{Z}.
\end{equation}
This condition is standard; it can also be deduced from more nontrivial fusion rules for admissible representations in \cite[Th. 3.2]{Feigin:1994} or integrable representations \cite[Th 1]{Tsuchiya:1988vertex}. Composing with isomorphisms \(\alpha_{x,z}\) we get a map
\begin{equation}
	\mathtt{m}_k(\vec{x},\vec{z}) \colon  \mathcal{M}_{\lambda,k} \otimes  \mathcal{L}_{\nu, k} \otimes  \mathcal{M}_{\mu, k} \rightarrow \mathbb{C}.
\end{equation}  

For generic \(\mu\) we have a map \(\mathtt{Y}_k(x,z): \mathcal{L}_{\nu,k} \rightarrow \operatorname{Hom}_{\mathbb{C}}( \mathcal{M}_{\lambda,k}, 
\overline{\mathcal{M}_{\mu,k}})\). Let \(X_\nu(x,z)= \mathtt{Y}(v_{\nu,k}|x,z)\). This operator satisfies the relations as in Proposition \ref{Prop: Awata Yamada}. The main difference that vertex operator here has \emph{finitely many terms} in \(x\) expansion, namely
\begin{equation}\label{eq:X nu}
	X_\nu(x,z)=\sum_{l=0}^\nu \binom{\nu}{l} X_{l, \nu}(z) x^{l}.
\end{equation}
The summands satisfy the following commutation relations 
%
	\begin{subequations}
	\begin{align}
		[f_n,   X_{l, \nu}(z)] &=  z^n (\nu - l) X_{l+1, \nu}(z),\\
		[e_n,  X_{l, \nu}(z)] &=  z^n  l  X_{l-1, \nu}(z),\\
		[h_n,   X_{l, \nu}(z) ] &= z^n(\nu - 2 l) X_{l, \nu}(z).
	\end{align}
	\end{subequations}
Finally the fusion rules are equivalent to the following proposition.
\begin{proposition}\phantomsection \label{onlyonevertex}
	Let $\nu\in\mathbb{Z}_{\ge 0}$ and \(\lambda\) be such that pair \((\lambda,k)\) is generic. Then, for any $r\in\{0,1,\dots \nu\}$ there exist a unique (up to constant) vertex operator  $X_\nu(x,z): \mathcal{M}_{\lambda-\nu+2r,k}\rightarrow\mathcal{M}_{\lambda,k}$. 
\end{proposition}

\subsubsection{Integrable representations} \label{ssec:integrable}
The similar definitions work for integrable representations of  $\widehat{\mathfrak{sl}}(2)_1$.
Similarly to the above, we can define the representations \(\mathcal{L}_{i,1}(x,z)\), \((i=0,1)\) of $\widehat{\mathfrak{sl}}(2)_{z}$. These representations do not depend on \(x\) up to isomorphism, so these parameters are usually excluded in the construction of vertex operators, but we prefer to keep them.

It was shown in \cite[Th. 1]{Tsuchiya:1988vertex}, \cite[sec. 2]{Tsuchiya:1989} that for $i,j,r\in \{0,1\}$ there exists and unique up to constant homomorphism of $\mathfrak{sl}(2)\otimes \mathbb{C}[\mathbb{P}^1{\setminus} \{\vec{z}\}]$-modules
\begin{equation}\label{eq:fusion level 1}
	\mathtt{m}_1: \mathcal{L}_{i,1}(x_1,z_1)\otimes  \mathcal{L}_{j,1}(x_2,z_2) \otimes  \mathcal{L}_{r,1}(x_3,z_3)\rightarrow \mathbb{C},
\end{equation}
if and only if $i + j +r\in 2\mathbb{Z}$. Composing with isomorphisms \(\alpha_{x,z}\) we get a map
\begin{equation}
	\mathtt{m}_1(\vec{x},\vec{z}) \colon  \mathcal{L}_{i,1} \otimes  \mathcal{L}_{j, 1} \otimes  \mathcal{L}_{r, 1} \rightarrow \mathbb{C}.
\end{equation}

This allows to define $Y_1(z):\mathcal{L}_{j,1} \rightarrow \operatorname{Hom}(\mathcal{L}_{i,1}, \mathcal{L}_{r,1})$. It is clear that $Y_1(v_{0}|z) = \operatorname{Id}_{\mathcal{L}_{i,1}}$. Let us define 
\begin{equation}
	b(x,z)=Y_1(v_{{1}/{2}}|x,z),\quad b(x,z) = b_0(z) + x b_1(z).
\end{equation} 
The components \(b_0(z), b_1(z)\) satisfy relations similar to \eqref{eq:Vnu comm rel}. 
In particular using commutation relations with \(h_n\) and realization of $\widehat{\mathfrak{sl}}(2)_1$ modules in terms of Heisenberg algebra (see Theorem \ref{Th:Frenkel}) we can write explicit formulas for \(b_0(z), b_1(z)\).
\begin{proposition}
	The operators \(b_0(z), b_1(z)\) are equal to exponents ot the bosonic field \(\phi(z)\) defined in \eqref{eq:bosonic field k=1}
    \begin{equation}\label{eq:b01}
    	b_{0}(z) = :\!\exp (\frac{\phi(z)}{\sqrt{2}})\!:(-1)^{\frac{h_0(h_0-1)}{2}},\quad b_{1}(z) = :\!\exp (-\frac{\phi(z)} {\sqrt{2}})\!:(-1)^{\frac{h_0(h_0-1)}{2}} . 
	\end{equation}
 Recall that ${h_0} = {a_0}{\sqrt{2}}$. 
\end{proposition}


\subsubsection{Bosonization of vertex operators}
Under certain conditions the vertex operators \(\mathcal{V}_\nu(x,z)\) can be written in terms of currents \(\beta(z),\gamma(z), \varphi(z)\) used in sec. \ref{ssec:Wakimoto}. 
Let us define operators $\mathcal{O}^{(N)}_{r; \nu}(z)$ which act between two Wakimoto modules $\mathcal{F}_{\lambda,k} \rightarrow \mathcal{F}_{\lambda + \nu - 2N,k}$ by formula:
\begin{equation}\label{omnnu}
    \mathcal{O}^{(N)}_{r, \nu}(z) = \int_{0\le t_N\le \dots \le t_1\le z} \prod_{i = 1}^N dt_i :\!\exp\left(\frac{\nu}{\sqrt{2\kappa}}\varphi(z)\right)\!:\gamma(z)^r ~ \prod_{i=1}^N S(t_i),
\end{equation}
where $S(t) = :\!\exp\left(\frac{-2}{\sqrt{2\kappa}}\varphi(t)\right)\!:\beta(t)$ is a screening field.

Here for simplicity we assumed the integration is performed over the relative cycle which is \(N\)-dimensional simplex. Such choice works under certain inequalities on the parameters \(\lambda, \nu, k\), for other values the vertex operators can be defined via analytic continuation. See \cite{EFK:1998} for the detailed discussion of the contour. The following proposition is standard (see e.g. \cite[Th, 5.7.4]{EFK:1998}, \cite{Gerasimov:1990}, \cite{Awata:1991})
\begin{proposition}\phantomsection \label{Prop:O comm rel}
	Commutation relations between operators  $\mathcal{O}^{(N)}_{r; \nu}(z)$ and $\widehat{\mathfrak{sl}}(2)$ generators have the form
	\begin{subequations}\label{eq:O comm rel}
		\begin{align}
			\label{eq:O comm rel e}
		    [e_n,  \mathcal{O}^{(N)}_{r; \nu}(z)] &=  z^n  r \mathcal{O}^{(N)}_{r-1; \nu}(z),
		    \\
		    [h_n,  \mathcal{O}^{(N)}_{r; \nu}(z) ] &= z^n(\nu-2 r)\mathcal{O}^{(N)}_{r; \nu}(z),
		    \\
		    \label{eq:O comm rel f}
		    [f_n,    \mathcal{O}^{(N)}_{r; \nu}(z)] &=  z^n (\nu-r)\mathcal{O}^{(N)}_{r + 1; \nu}(z).
		\end{align}
	\end{subequations}
\end{proposition}
\begin{proposition}\phantomsection \label{Prop:O&Vnu 1}
	If $\mu, k$ are generic and $\lambda + \nu - \mu = 2 N$, $N \in \mathbb{Z}_{\ge 0}$ then the operator $\mathcal{V}_{\nu}(x,z) \colon  \mathcal{M}_{\lambda,k}\rightarrow  \mathcal{M}_{\mu,k}$ has the form 
	\begin{equation} \label{eq:Vnu=O 1}
		\mathcal{V}_{\nu}(x,z)=C \sum_{r\ge 0} \binom{\nu}{r} \mathcal{O}^{(N)}_{r; \nu  }(z) x^{r}.
	\end{equation}
\end{proposition}

Here \(C\) is the scalar factor which cannot be fixed by commutation relations, see Proposition \ref{Prop:coinv generic}. Note that the operator \(X_\nu(x,z)\) above can be written using formula \eqref{eq:Vnu=O 1} for \(\nu \in \mathbb{Z}_{\ge 0}\).

\begin{proof}
	Due to uniqueness of \(\mathcal{V}_\nu\) it is sufficient to check that right side of satisfies commutation relations \eqref{eq:Vnu comm rel}. This is direct computation using Proposition \ref{Prop:O comm rel}.	
	%
	%
\end{proof}

\begin{proposition}\phantomsection \label{Prop:O&Vnu 2}
	Let $\mu, k$ be generic and $\lambda - \nu - 2 - \mu = 2 N,~ N \in \mathbb{Z}_{\ge 0}$. Then the expansion of the operator   $\mathcal{V}_{\nu}(x,z) \colon  \mathcal{M}_{\lambda,k}\rightarrow  \mathcal{M}_{\mu,k}$ has the form \(\mathcal{V}_{\nu}(x,z)=C \sum_{r \in \mathbb{Z}} 
	\mathcal{V}_{\nu}^{[r]}(x,z)  x^{\nu+1+r}\) where 
	\begin{equation} \label{eq:Vnu=O 2}
		\mathcal{V}_{\nu}^{[r]}(x,z)=
		(-1)^r\mathcal{O}^{(N)}_{r; -\nu - 2}(z),\;\; 
		\text{ for }r\ge 0.
	\end{equation}
\end{proposition}
Note that, contrary to the formula \eqref{eq:Vnu=O 1} we cannot write that whole vertex operator \(\mathcal{V}_{\nu}\) is simply the sum of \(\mathcal{O}^{(N)}_{r; -\nu - 2}\). This is because there are also summands \(\mathcal{V}_{\nu}^{[r]}(x,z)\) with \(r<0\) which do not have formula like \eqref{omnnu}. Also note the reflection \(\nu \leftrightarrow -\nu-2\) in the indices in the formula \eqref{eq:Vnu=O 2}.
\begin{proof}
	Same as above, the equations \eqref{eq:Vnu comm rel} lead to the system of relations on commutators of \(\mathcal{V}_{\nu}^{[r]}(z)\) and \(e_n,h_n,f_n\). This system  coincides with relations \eqref{eq:O comm rel} (but for all \(r \in \mathbb{Z}\)) and has unique solution. One can additionally note that the system of equation on operators  \(\mathcal{V}_{\nu}^{[r]}(z)\) with \(r\ge 0\) is closed and these operators can be uniquely determined without using \(r< 0\) ones.
%
%
\end{proof}

\subsection{Virasoro vertex operators} \label{ssec:Virasoro}

In the Virasoro case we follow the same pattern. We keep the notation \(t\) for some fixed global coordinate on \(\mathbb{P}^1\). For any \(z\in \mathbb{P}^1\) let 	\(\mathrm{Vir}_z\) denotes Lie algebra with basis \((t-z)^n \partial_t\) and central element \(C\) with the commutator defined by 
\begin{equation}
	[f(t) \partial_t, g(t) \partial_t]= (f g'-f' g) \partial_t + \frac{1}{12}C\operatorname{Res}_{t=z} (f''' g dt).
\end{equation}
There is an isomorphism \(\alpha_z\) between the abstract Virasoro algebra and \(\mathrm{Vir}_z\), \(L_n \mapsto -(t-z)^{n+1}\partial_t\). Using this isomorphism, we can define highest weight modules $\mathbb{M}_{P,b}(z)$ over \(\mathrm{Vir}_z\)

Let \(\mathrm{Vect}(\mathbb{P}^1){\setminus} \{z_1,z_2,z_3\})\) denotes Lie algebra of meromorphic vector fields with poles only at \(z_1,z_2,z_3\). There is a Lie algebra map
\begin{equation}
	\mathrm{Vect}(\mathbb{P}^1){\setminus} \{z_1,z_2,z_3\}) \rightarrow \overline{\mathrm{Vir}_{z_1}} \oplus \overline{\mathrm{Vir}_{z_2}} \oplus \overline{\mathrm{Vir}_{z_3}}, 
\end{equation}
which map any element to the direct sum of series expansions at points \(z_i\). The completions \(\overline{\mathrm{Vir}_{z}}\) denote Lie algebra with elements of the form \( f(t-z)\partial_t + \alpha C\), where \(f \in \mathbb{C}[[t - z, (t - z)^{-1}]\) is a Laurent series. Clearly \(\overline{\mathrm{Vir}_{z}}\) acts on $\mathbb{M}_{P,b}(z)$.

For generic \(P_1, P_2, P_3\) there exists a unique up to constant homomorphism of \(\mathrm{Vect}\big(\mathbb{P}^1{\setminus} \{z_1,z_2,z_3\}\big)\) modules 
 \begin{equation}
 	\mathtt{m}^{\mathrm{Vir}}
 	\colon \mathbb{M}_{P_1,b}(z_1)\otimes \mathbb{M}_{P_2,b}(z_2)\otimes \mathbb{M}_{P_3,b}(z_3) \rightarrow \mathbb{C}.
 \end{equation}
Moreover, the homomorphism is uniquely determined by its value on the product of highest vectors \(v_{P_1,b}\otimes v_{P_2,b} \otimes v_{P_3,b}\). Composing with \(\otimes \alpha_{z_i}\) and moving points to \(\{0,z,\infty\}\) we get a map 
\begin{equation}
	\mathtt{m}^{\mathrm{Vir}}(z)
	\colon \mathbb{M}_{P_1,b}\otimes \mathbb{M}_{P_2,b}\otimes \mathbb{M}_{P_3,b} \rightarrow \mathbb{C}.
\end{equation}
For generic \(P_3\) this map leads to an operator state correspondence map \(\mathtt{Y}^{\mathrm{Vir}}(z): \mathbb{M}_{P_2,b} \rightarrow \operatorname{Hom}_{\mathbb{C}}( \mathbb{M}_{P_1,b}, 
\overline{\mathbb{M}_{P_3,b}})\). Let \(\Phi_\Delta(z)=\mathtt{Y}^{\mathrm{Vir}}(v_{P,b}|z)\), here \(\Delta=\Delta(P,b)\). The \(\mathrm{Vect}\big(\mathbb{P}^1{\setminus} \{z_1,z_2,z_3\}\big)\) invariance of \(\mathtt{m}^{\mathrm{Vir}}\) leads to the following commutation relations.
\begin{proposition}
	The operator $\Phi_\Delta(z)$ satisfies 	
	\begin{equation} \label{eq:Phi_Delta L_n}
	    [L_n, \Phi_\Delta(z)] =  z^{n + 1}\partial\Phi_\Delta(z) + (n + 1)z^n \Delta \Phi_\Delta(z).
	\end{equation}
\end{proposition}
The commutation relations \eqref{eq:Phi_Delta L_n} can serve as a definition of Virasoro vertex operator.

For the \(\Delta=\Delta_{m,n}\) Verma module becomes reducible (see Theorem \ref{Th:Delta_mn}). We can replace Verma modules by their irreducible quotients \(\mathbb{L}_{P_{m,n},b}\). The corresponding vertex operators will be denoted by \(\Phi_{m,n}(z)\). The singular vector condition on \(v_{P_{m,n},b}\) corresponds to the equation on the vertex operator. In the important cases  \(\Phi_{1,2}(z)\) and \(\Phi_{2,1}(z)\) we have (cf. Example \ref{Exa:Phi 12 21})
\begin{subequations}
	\begin{align}
		\partial_z \Phi_{2,1}(z) + b^{-2} :\! L(z) \Phi_{2,1}(z)\!: = 0,
		\\
		\label{eq: Phi12}
		\partial_z \Phi_{1,2}(z) + b^{2} :\! L(z)\Phi_{1,2}(z)\!: = 0.
	\end{align}
\end{subequations}
These differential equations lead to the famous Virasoro fusion rules \cite{BPZ:1984}.
%
%

\section{Coset construction} \label{sec:coset}
\subsection{Definition, decomposition}
    
Let us consider the algebra $\widehat{\mathfrak{sl}}(2)_1\oplus \widehat{\mathfrak{sl}}(2)_k$ and its module $\mathcal{L}_{i,1}\otimes \mathcal{M}_{\lambda,k}$. There is an action of diagonal $\widehat{\mathfrak{sl}}(2)_{k + 1}$ on this space. Let us denote the generators of this algebra by $e_n^\Delta, f_n^\Delta, h_n^\Delta$. 

Note that there are three different actions of the Virasoro algebra on $\mathcal{L}_{i,1}\otimes \mathcal{M}_{\lambda,k}$ which come from the Sugawara construction. We will call the corresponding currents by $L^{(1)}(z), L^{(2)}(z)$ and $L^\Delta(z)$.

\begin{theorem}[\cite{Goddard:1986}] \phantomsection
	The modes of $L^{GKO}(z) = L^{(1)}(z) +  L^{(2)}(z) - L^\Delta(z)$ satisfy the relations of the Virasoro algebra and commute with diagonal affine algebra $\widehat{\mathfrak{sl}}(2)_{k + 1}$. 
\end{theorem}
Let us use the notation $\mathrm{Vir}^{\mathrm{coset}}$ for the algebra generated by $L^{GKO}_n$. The following result is standard, we follow \cite{Bershtein2016coupling} for the statement and sketch of the proof.

\begin{theorem}\phantomsection \label{Th:cosetdec} 
	For generic $\lambda, k$ there are the following decompositions of  $\mathcal{L}_{i,1}\otimes \mathcal{M}_{\lambda,k}$ as  \(\mathrm{Vir}^{\mathrm{coset}} \oplus \widehat{\mathfrak{sl}}(2)_{k + 1} \) modules
	\begin{subequations}\label{eq:W1 W2}
		\begin{align}
			\label{eq:W1}
		\mathcal{L}_{0, 1} \otimes  \mathcal{M}_{\lambda, k} &= \bigoplus\nolimits_{l\in \mathbb{Z}} \mathbb{M}_{P(\lambda) + lb,b} \otimes \mathcal{M}_{\lambda+2l, k + 1},
		\\
		\label{eq:W2}
		\mathcal{L}_{1, 1} \otimes  \mathcal{M}_{\lambda, k} &= \bigoplus\nolimits_{l \in \mathbb{Z}+\frac{1}{2}} \mathbb{M}_{P(\lambda) + lb,b} \otimes \mathcal{M}_{\lambda+2l, k + 1}.
		\end{align}
	\end{subequations}	
	where $b = b_{GKO}=-\ri \sqrt{\frac{k+2}{k+3}}$ and $P(\lambda) = P_{GKO}(\lambda)=-\frac{\lambda +1}{2 (k+2) } b_{GKO}$.
\end{theorem}
\begin{proof}[Sketch of the proof] The proof is based on the same two arguments as the proof for the more difficult case of admissible representations, see \cite[Theorem 10.2]{Iohara:2011}. First, one checks the identity of characters for the modules on the left and right sides. Second, one shows that there are no extensions between summands on the right side.
\end{proof}

\begin{define}\phantomsection \label{Def:Shap sl2} 
	The Shapovalov form on Verma module $\mathcal{M}_{\lambda,k}$ of \(\widehat{\mathfrak{sl}}(2)\) algebra is defined by the following properties
	\begin{equation}
		\langle v_{\lambda,k},v_{\lambda,k}\rangle = 1,\qquad e_n^\dagger = f_{-n},\;\; h_n^\dagger = h_{-n},\;\; f_n^\dagger = e_{-n},
	\end{equation}
	where $v_{\lambda,k}$ is the highest weight vector in $\mathcal{M}_{\lambda,k}$.
\end{define}
The Shapovalov form descents to the irreducible quotient $\mathcal{L}_{\lambda,k}$. We define form on the tensor product of the modules by the formula
\begin{equation}
    \langle x_1\otimes x_2, y_1\otimes y_2\rangle = \langle x_1,  y_1\rangle \langle x_2,  y_2\rangle .
\end{equation}
Hence we have Shapovalov form on the modules $\mathcal{L}_{i,1}\otimes \mathcal{M}_{\lambda,k}$.

\begin{define}\phantomsection \label{Def:Shap Vir}
    Shapovalov form on the Verma module $\mathbb{M}_{P,b}$ of the Virasoro algebra is defined by the following properties
      \begin{equation}
      \langle v_{P,b}, v_{P,b}\rangle = 1, ~~~ L_n^\dagger = L_{-n},
  \end{equation}
  where $ v_{P,b}$ is the highest weight vector in $\mathbb{M}_{P,b}$.
\end{define}

\subsection{Formula for the highest weight vectors}
\begin{notati}
	Let us denote the highest weight vector in coset decomposition by 
    \begin{equation}
    	u_l(\lambda) \in \mathbb{M}_{P + lb,b} \otimes \mathcal{M}_{2l +\lambda, k + 1}\subset (\mathcal{L}_{0, 1} \oplus \mathcal{L}_{1, 1}) \otimes  \mathcal{M}_{\lambda, k}.
	\end{equation}
\end{notati}
Recall (see sec. \ref{ssec:level 1}) notation \(v_l\) for extremal vectors in \(\mathcal{L}_{i,1}\). For the tensor products of extremal and highest weight vectors we will use notation
\begin{equation}
	v_l(\lambda)  = v_{l}\otimes v_{\lambda,k}.
\end{equation}
Let us fix normalization of $u_l(\lambda)$ through the Shapovalov scalar product with $v_l(\lambda)$
\begin{equation}\label{eq:ul normalization}
	\langle u_l(\lambda), v_l(\lambda)\rangle = 1.
\end{equation}
\begin{example}
	The simplest examples of the highest weight vectors $u_{l}(\lambda)$ have the form
	\begin{subequations}\label{eq:ul Verma}
		\begin{align}
			u_{0}(\lambda)&=v_{0}\otimes v_{\lambda,k},
			\\
			\label{eq:ul Verma 1/2}
			u_{-{1}/{2}}(\lambda)&=f_{0}v_{{1}/{2}}\otimes
			v_{\lambda,k}-\frac{1}{\lambda}v_{{1}/{2}}\otimes f_{0}v_{\lambda,k}
			\\
			\label{eq:ul Verma 1}			
			u_{-1}(\lambda)&=v_{-1}\otimes  v_{\lambda,k} - \frac{1}{(\lambda{-}1) \lambda}
			e_{-1} v_0\otimes f_{0}^2v_{\lambda,k} \, -  \,
		 	\frac{1}{\lambda} h_{-1}v_0\otimes f_{0}v_{\lambda,k}-\frac{1}{(k{+}\lambda{+}2)}v_0\otimes f_{-1}v_{\lambda,k} \notag
		 	\\	 	
		 	&  + \frac{1}{\lambda ( k{+}\lambda{+}2)}
		 	v_0\otimes h_{-1}f_{0}v_{\lambda,k}    
		 	+ \frac{1}{(\lambda{-}1) \lambda ( k{+}\lambda{+}2)}
		 	v_0\otimes e_{-1} f_{0}^2 v_{\lambda,k}.
		\end{align}
	\end{subequations}
\end{example}    
Recall that for general $\lambda,k$ Verma module $\mathcal{M}_{\lambda,k}$ is isomorphic to Wakimoto module $\mathcal{F}_{\lambda,k}$ (see Proposition \ref{Prop:wakverm}). Our first goal is to find explicit formulas for vectors \(u_{l}(\lambda)\), or actually to their images in \(\mathcal{L}_{i,1}\otimes \mathcal{F}_{\lambda,k}\). Let us define current $g(z)$ acting on \(\mathcal{L}_{i,1}\otimes \mathcal{F}_{\lambda,k}\)
\begin{equation}
	g(z) =  e^{(1)}(z) \gamma(z), ~~~ g_0 = \sum_{n \in \mathbb{Z}}  e^{(1)}_n \gamma_{-n}.
\end{equation}
\begin{theorem} \phantomsection \label{Th:FormUm} 
	For generic \(\lambda,k\) and $2l\in \mathbb{Z}_{\le 0}$ then $l$-th highest weight vector in decompositions \eqref{eq:W1 W2} with respect to  $\mathrm{Vir}^{\mathrm{coset}}\oplus \widehat{\mathfrak{sl}}(2)_{k+1}$ is given by
	\begin{equation}\label{eq:ul Wakimoto}
		u_{l}(\lambda) = \re^{-g_0} v_{l}(\lambda).
	\end{equation}
\end{theorem}

\begin{remark}
	In the proof below we will see that the right side of \eqref{eq:ul Wakimoto} is defines the highest weight vector in \(\mathcal{L}_{i,1}\otimes \mathcal{F}_{\lambda,k}\) for arbitrary \(\lambda,k\). On the other hand, in order to obtain highest weight vector in \(\mathcal{L}_{i,1}\otimes \mathcal{M}_{\lambda,k}\) we use Proposition~\ref{Prop:wakverm}, i.e. condition that \(\lambda,k\) are generic is essential for the proof. 
	Furthermore, the definition of vectors \(u_{l}(\lambda)\) was based on decomposition \eqref{eq:W1 W2}, i.e. only for generic values of parameters. See also Remark~\ref{Rem:RT proof} below.
	
%
\end{remark}

\begin{example}
	Let us also present the examples of highest weight vectors $u_{l}(\lambda)$ after Wakimoto realization
	\begin{subequations}
		\begin{align}
			u_{0}(\lambda)&=v_{0}\otimes v_{\lambda,k},
			\\
			u_{-{1}/{2}}(\lambda)&=v_{-{1}/{2}}\otimes
			v_{\lambda,k}-v_{{1}/{2}}\otimes \gamma_{0}v_{\lambda,k}
			\\
			u_{-1}(\lambda)&= v_{-1}\otimes  v_{\lambda,k}- v_{0}\otimes  \gamma_{-1} v_{\lambda,k} - h_{-1}v_{0}\otimes  \gamma_0 v_{\lambda,k} -  v_1 \otimes \gamma_0^2 v_{\lambda,k}.
		\end{align}
	\end{subequations}
	Comparing with the formulas \eqref{eq:ul Verma} one can see significant simplification, in particular there are less summands and there are no denominators after Wakimoto realization.
\end{example}    

\begin{remark}\phantomsection \label{Rem:tau} 
	In the theorem above we had \(l\le 0\). Let us comment on the other case \(l>0\). There is an automorphism $\tau$ of algebra $\widehat{\mathfrak{sl}}(2)_k$ 
	\begin{equation}
	    \tau(e_n) = f_{n+1}, ~~ \tau(f_n) = e_{n-1}, ~~ \tau(h_n) = \delta_{n,0}k - h_n.
	\end{equation}
	This automorphism leads to a map between Verma modules $\tau\colon\mathcal{M}_{\lambda,k}\rightarrow\mathcal{M}_{k-\lambda,k}$ such that
	\begin{equation}
	    \tau(v_{\lambda,k})=v_{k-\lambda,k},
	    \quad \tau(x v)=\tau(x) \tau(v),
	\end{equation}
	where $x\in \widehat{\mathfrak{sl}}(2)_k$, $v \in \mathcal{M}_{\lambda,k}$ and $v_{\lambda,k}$ denotes the highest weight vector in $\mathcal{M}_{\lambda,k}$.  There exist similar map  $\tau\colon \mathcal{L}_{i,1} \rightarrow \mathcal{L}_{1-i,1}$, for $i\in\{0,1\}$.
	
	Consider $\tau\otimes\tau: \mathcal{M}_{\lambda,k}\otimes \mathcal{L}_{i,1} \rightarrow  \mathcal{M}_{k-\lambda,k}\otimes \mathcal{L}_{1-i,1}$. 
	We have
	\begin{equation}
	    \tau\otimes\tau(v_l(\lambda)) = v_{\frac{1}{2}-l}(k-\lambda).
	\end{equation}
	Hence the formulas for \(u_l(\lambda)\) with \(l>0\) can be obtained using \(\tau\). An analog of formula \eqref{eq:ul Wakimoto} for \(l>0\) requires a change of Wakimoto realization, obtained by permutation of \(e(z)\) and \(f(z)\), i.e., instead of formula \eqref{eq:wakimoto} we have \(f(z)=\gamma(z)\) but \(e(z)\) is the sum of three terms.

    Another analog of formula \eqref{eq:ul Wakimoto} for \(l>0\) uses conjugated operator \(g_0^\dagger\). It is given in Remark~\ref{Rem:l>0}.
\end{remark}

We will first prove Theorem \ref{Th:FormUm} by a direct computation.
\begin{lemma}\phantomsection \label{Lem:efconj}
Conjugation by \(\re^{g_0}\) acts on the operators of diagonal $\widehat{\mathfrak{sl}}(2)_{k+1}$ as follows
\begin{align}
	&\re^{g_0}e^\Delta(w) \re^{-g_0} = \beta(w),
    \\
    \label{eq:fconj}
    &\re^{g_0}f^\Delta(w) \re^{-g_0} =  f^{\Delta}(w) + h^{(1)}(w)\gamma(w) 
    +\gamma'(w).
\end{align}
\end{lemma}
\begin{proof}
Th exponent of adjoint action has the form
\begin{align}
	\re^{g_0}e^\Delta(w) \re^{-g_0} &=  e^\Delta(w) + [g_0, e^\Delta(w)] + \frac{1}{2}[g_0,[g_0, e^\Delta(w)]]+ \dots\;,
	\\
    \re^{g_0}f^\Delta(w) \re^{-g_0} &=  f^\Delta(w) + [g_0, f^\Delta(w)] + \frac{1}{2}[g_0,[g_0, f^\Delta(w)]]+ \dots\;.
\end{align}
Hence it is sufficient to find all commutators. In order to do this we will use OPE (see \cite[Ch 3]{FB:2004})
\begin{equation}
	g(z) e^\Delta(w) = e^{(1)}(z) \gamma(z) (e^{(1)}(w) + \beta(w)) = \frac{- e^{(1)}(w)}{z - w}+ \text{reg}.
\end{equation}
It means that
\begin{equation}
	[g_0, e^\Delta(w)] = - e^{(1)}(w).
\end{equation}
Moreover, the OPE of $g(z)$ and $e^{(1)}(w)$ has no singular term, hence
\([g_0,[g_0, e^\Delta(w)]] = 0.\). Therefore
\begin{equation}
	\re^{g_0}e^\Delta(w) \re^{-g_0} = e^{(2)}(w)=\beta(w).
\end{equation}
Similarly,
\begin{multline}
     g(z) f^{\Delta}(w) =  e^{(1)}(z) \gamma(z) (f^{(1)}(w) + f^{(2)}(w)) =
     \\
	 = \frac{\gamma(w)}{(z - w)^2} +  \frac{h^{(1)}(w)\gamma(w) + e^{(1)}(w)\gamma^{2}(w)+\gamma'(w)}{z - w} + \text{reg}.  	
\end{multline}
Then using \cite[Prop. 3.3.1]{FB:2004} we have 
\begin{equation}
	[g(z), f^{\Delta}(w)]= \gamma(w) \partial_w \delta(z,w)+ \big(h^{(1)}(w)\gamma(w) + e^{(1)}(w)\gamma^{2}(w)+\gamma'(w)\big)\delta(z,w),
\end{equation}
where \(\delta(z,w)=\sum_{n \in \mathbb{Z}} z^n w^{-n-1}\). Hence,
\begin{equation}
    [g_0, f^{\Delta}(w)] = h^{(1)}(w)\gamma(w) + e^{(1)}(w)\gamma^{2}(w)+\gamma'(w).
\end{equation}
 Taking now the OPE of $g(z)$ and $[g_0, f^{\Delta}(w)]$ we get
\begin{equation}
    [g_0,[g_0,f^{\Delta}(w)]] = - 2 e^{(1)}(w)\gamma^{2}(w).
\end{equation}
It is easy to see that successive commutators are equal to zero. Therefore we obtain formula~\eqref{eq:fconj}.
\end{proof}
Now we are ready to prove the theorem.
\begin{proof}[Proof of the Theorem \ref{Th:FormUm}]
	We decompose the proof into three steps.

	\textbf{Step 1.} Let us show that $\re^{-g_0}v_{-l}(\lambda)$ is the highest weight vector with respect to diagonal $\widehat{\mathfrak{sl}}(2)^\Delta$. It is sufficient to consider action $e^{\Delta}_0$ and $f^\Delta_1$. Using Lemma \ref{Lem:efconj} we have
    \begin{align}
    	&e^\Delta_0 \re^{-g_0} v_{-l}(\lambda) = \re^{-g_0}\re^{g_0} e^\Delta_0 \re^{-g_0} v_{-l}(\lambda) = \re^{-g_0}\beta_0 v_{-l}(\lambda) = 0,
    	\\
   	    &f^\Delta_1 \re^{-g_0} v_{-l}(\lambda) = \re^{-g_0}\re^{g_0} f^\Delta_1 \re^{-g_0} v_{-l}(\lambda) = \re^{-g_0}  \Big( f^{(2)}_1 + f^{(1)}_1 - \sum_{s+r = 1} h^{(1)}_r \gamma_s-\gamma_1\Big)v_{-l}(\lambda) = 0. 
	\end{align}
	
	\textbf{Step 2.} It follows from the decompositions \eqref{eq:W1 W2} that there exist only one up to proportionality vector \(v \in \mathcal{L}_{i, 1} \otimes  \mathcal{M}_{\lambda, k}\) which is the highest weight vector for  $\widehat{\mathfrak{sl}}(2)^\Delta$ and has the same eigenvalues of \(h_0^\Delta \) and \(L_0^{GKO}\) as $u_{l}(\lambda)$. Hence, the vectors $u_{l}(\lambda)$ and $\re^{-g_0} v_{l}(\lambda)$ are proportional.
	
%
	\textbf{Step 3.} It remains to check normalization property 
    \begin{equation}\label{eq:e{-g0}v norm}
        \langle \re^{-g_0}v_{-l}(\lambda), v_{-l}(\lambda)\rangle = 1.
    \end{equation}
    Note that in the series \(\re^{-g_0}v_{-l}(\lambda)=\sum_{j \ge 0}(-g_0)^jv_{-l}(\lambda)/j! \) all summands except the first one has \(h_0^{(1)}\) gradation different from the one of \( v_{-l}(\lambda)\). Hence they are orthogonal to  \(v_{-l}(\lambda)\),
\end{proof}

\subsection{Operator $I(z)$ and another proof of the Theorem \ref{Th:FormUm}}

Let us consider the operator $I(z): (\mathcal{L}_{0,1}\oplus \mathcal{L}_{1,1})\otimes M_{\lambda, k} \rightarrow \overline{(\mathcal{L}_{0,1}\oplus \mathcal{L}_{1,1})\otimes M_{\lambda + 1, k}}$ defined by formula
\begin{equation}\label{eq:defI}
	I(z) =  b_1(z)\otimes\mathcal{O}^{(0)}_{0, 1}(z) - b_0(z)\otimes\mathcal{O}^{(0)}_{1, 1}(z). 
\end{equation}
Recall that notations \(b_0(z), b_1(z), \mathcal{O}^{(N)}_{r, \nu}(z)\) were introduced in Subsection \ref{ssec:operstatecorr} and overline stands for the completion. In particular, according to Proposition  \ref{Prop:O&Vnu 1}, the operators \(\mathcal{O}^{(0)}_{0, 1}, \mathcal{O}^{(0)}_{1, 1}\) used in formula \eqref{eq:defI} are components of \(X_1\) corresponding to degenerate representation. The operator $I(z)$ corresponds to the skew-symmetric tensor product of two $2$-dimensional representations of \(\mathfrak{sl}_2\), see formula \eqref{eq:v_I} below. 


\begin{proposition}\phantomsection \label{Th:Iprop}
	1) The operator $I(z)$ commutes with $\widehat{\mathfrak{sl}}(2)_{k+1}^\Delta$.
	
	2) The operator $I(z)$ is $\mathrm{Vir}^{\mathrm{coset}}$ vertex operator $\Phi_{2,1}(z)$.
\end{proposition}
\begin{proof}
	Both operators \(b_i(z)\) and \(\mathcal{O}^{(0)}_{0, 1}(z)\) are obtained via operator-state correspondence map~$\mathtt{Y}$. Hence the operator $I(z)$ is  obtained via operator-state correspondence map \(\mathtt{Y}\otimes \mathtt{Y}\) applied to the vector
	\begin{equation}\label{eq:v_I}
	    v_I = v_{{1}/{2}}\otimes f_0 v_{1, k} -  v_{-{1}/{2}}\otimes v_{1, k} \in \mathcal{L}_{1,1}\otimes\mathcal{L}_{1,k}.
	\end{equation} 
	The commutation relations of the field \(I(z)\) with algebra generators are equivalent to the highest weight conditions for the vector \(v_I\). 
		
	It is easy to see that
	\begin{equation}
	    e^\Delta_{n\ge 0} v_I = h^\Delta_{n\ge 0} v_I = f^\Delta_{n\ge 0} v_I = 0   
	\end{equation}
	and
	\begin{equation}
	    L^{GKO}_0 v_I = \Delta_{2,1} v_I,\;\; L^{GK0}_{n>0} v_I = 0,
	\end{equation}
	where $\Delta_{2,1} = \Delta(P_{2,1}(b_{\text{GKO}}),b_{\text{GKO}})$ (see Theorem \ref{Th:cosetdec} for notation). Furthermore, it is easy to check that
	\begin{equation}
	    ((L^{GKO}_{-1})^2 + (b_{\text{GKO}})^{-2}L^{GKO}_{-2})v_I = 0.
	\end{equation}
	It means that the field $I(z)$ is the degenerate Virasoro vertex operator.
\end{proof}

\begin{proposition}\phantomsection \label{Prop:I conj}
	The operator $I(z)$ has the form
    \begin{equation}\label{vertcomwithdiag}
    	I(z) =\re^\frac{\varphi(z)}{\sqrt{2\kappa}} \re^{-g_0} b_1(z)  \re^{g_0}.
    \end{equation}
\end{proposition}
\begin{proof}
	Note that
	\begin{equation}
		g(z) b_1(w) = \frac{b_0(w) \gamma(w)  }{z - w} + \text{reg}.
	\end{equation}
	Hence $[g_0, b_1(w)] = b_0(w) \gamma(w)$. Moreover, we have $[g_0,[g_0, b_1(w)]] = 0$ since $g(z) b_0(w)\gamma(w)$ is regular. 
 \end{proof}

Now we give the second proof of Theorem~\ref{Th:FormUm}. 
\begin{proposition}\phantomsection \label{Prop:I vl}
    1) For any \(2l \in \mathbb{Z}_{\le 0}\) we have 
    \begin{equation}\label{inthighvect}
	    \oint_{C_0} \frac{dz}{z^{d(l,\lambda, k)+1}} I(z) 	u_{l}(\lambda) =  (-1)^{l(2l - 1)} u_{l - 1/2}(\lambda+1);
    \end{equation}
    where $d(l,\lambda, k) = (\frac{\lambda}{2\kappa} - l)$ and $C_0$ is a contour encircling $0$.
    
    2) For \(2l\in  \mathbb{Z}_{\le 0}\) we have
    \begin{equation}\label{eq:ul second}
        u_{l}(\lambda) = \re^{-g_0} v_{l}(\lambda).
    \end{equation}
\end{proposition}
The exponent \(d(l,\lambda, k)\) in \eqref{inthighvect} means that the leading term of \(I(z)u_{l}(\lambda)\) coincides with \(u_{l - {1}/{2}}(\lambda+1)\).

\begin{proof}
	Due to Theorem \ref{Th:cosetdec} we have 
	\begin{equation}
		I(z) u_l \in \overline{(\mathcal{L}_{0,1}\oplus \mathcal{L}_{1,1})\otimes M_{\lambda + 1, k}} = 
		 \bigoplus\nolimits_{2l \in \mathbb{Z}} \overline{\mathbb{M}_{P(\lambda+1) + lb,b} \otimes \mathcal{M}_{2l +\lambda+1, k + 1}}. 
	\end{equation}
	Since \(I(z)\) commutes with $\widehat{\mathfrak{sl}}(2)_{k+1}^\Delta$ and \(u_l\) is the highest vector for this algebra with highest weight \(\lambda+2l\) we get \(I(z) u_l \in  \overline{\mathbb{M}_{P(\lambda+1) + (l-\frac12)b,b}}\otimes v_{\lambda+2l,k+1} \). Hence, the leading term of \(I(z)u_{l}(\lambda)\) is equal to \(c_{l,\lambda}u_{l-1/2}(\lambda+1)\) for some (probably zero) \(c_{l,\lambda}\). 
	
	On the other hand, using the formula \ref{vertcomwithdiag} we get
	\begin{equation}
		   \oint_{C_0} \frac{dz}{z^{d(s,\lambda, k)}} I(z) \re^{-g_0} v_{l}(\lambda) = 
		   (-1)^{l(2l - 1)} \re^{-g_0} v_{l-1/2}(\lambda+1).  
	\end{equation}
	Therefore by induction by \(l \le 0\) we obtain
	\begin{equation}
		\re^{-g_0} v_{l}(\lambda)= \left(\prod\nolimits_{2m\in \mathbb{Z}, l<m \le 0} c_{m,\lambda-2l+2m} \right)	u_l(\lambda).
	\end{equation}
	Taking into account correct normalization property \eqref{eq:e{-g0}v norm} we conclude that \(c_{l\lambda}=1\) for \(l\le 0\). This finishes the proof.
%
%
%
%
%
%
%
\end{proof}

Note that \(c_{l,\lambda}\) appearing in the proof above is non-trivial for \(l>0\). We will discuss them belowm see Example~\ref{Exa:I}.

It follows from the formula \eqref{inthighvect} that for \(l \le 0\)
\begin{equation}\label{actI}
  \left\langle I(1) u_{l}(\lambda),  u_{l - {1}/{2}}(\lambda+1) \right\rangle = (-1)^{l(2l - 1)} \left\|u_{l-{1}/{2}}(\lambda+1)\right\|^2.
\end{equation}
We will use this matrix element below.

\subsection{Norms of the highest weight vectors}\label{ssec:norms}
The following theorem is one of the main results of the paper.
\begin{theorem}\phantomsection \label{Th:Norm}
    For $2l \in \mathbb{Z}_{\le 0}$ norm of the vector of $u_{l}(\lambda)$ is
    \begin{equation}\label{eq:norm}
        \left\|u_{l}(\lambda)\right\|^2 =
        \prod_{m=0}^{1-2l}\frac{\lambda + 1 + m \kappa}{ \lambda +1+2l + m(\kappa + 1)} = 
        \left(\frac{\kappa }{\kappa +1}\right)^{-2 l} \frac{\Gamma \left(\frac{\lambda +1}{\kappa }-2 l\right) \Gamma \left(\frac{\lambda +1+2l}{\kappa +1}\right)}{\Gamma \left(\frac{\lambda +1}{\kappa }\right) \Gamma \left(\frac{\lambda +1+2l}{\kappa +1}-2 l\right)}.
    \end{equation}
\end{theorem}
Note that although the Theorem \ref{Th:FormUm}  provides a formula for 	\(u_{l}\) in terms of the \textit{Wakimoto free field realization}, the norm formula above uses the Shapovalov form in the \textit{Verma module}.

The equality of the central and right sides in formula \eqref{eq:norm} follows from straightforward  computation. 
We will prove the equality between the left side and central (and right) sides in the next subsection. The proof uses a vertex operator $J(z)$ that is similar to the operator $I(z)$ used above. 

\begin{example}\label{Exa:norms} 
	The simplest examples of the formula for the norm are
	\begin{equation}\label{eq:example norms}
		\left\|u_{0}(\lambda)\right\|^2 = 1, ~~ 	\left\|u_{-1/2}(\lambda)\right\|^2=\frac{\lambda+1}{\lambda},
	~~ \left\|u_{-1}(\lambda)\right\|^2=\frac{(\lambda+1)(k + \lambda+3)}{(\lambda-1)(k + \lambda + 2)}.
	\end{equation}
	This expressions can be checked using formulas \eqref{eq:ul Verma}. Note that the formula factorizes into the product of linear terms, and this is absolutely unclear from the definition or formulas~\eqref{eq:ul Verma}.
\end{example}
The norms of $u_{l}(\lambda)$ for $l>0$ can be calculated using the map $\tau\otimes\tau$ that reflects \((\lambda,l) \leftrightarrow (k-\lambda,1/2-l)\) (see Remark \ref{Rem:tau}). Moreover, the resulting formula can be rewritten similar to the right side of \eqref{eq:norm} but with an interchanged numerator and denominator, namely for \(l>0\) we obtain
\begin{equation}
	\left\|u_{l}(\lambda)\right\|^2
	= 
	\left(\frac{\kappa }{\kappa{+}1}\right)^{2 l{-}1}\! \frac{\Gamma \left(\frac{\kappa-\lambda -1}{\kappa }{+}2 l{-}1\right) \Gamma \left(\frac{\kappa-\lambda -2l}{\kappa +1}\right)}{\Gamma \left(\frac{\kappa-\lambda -1}{\kappa }\right) \Gamma \left(\frac{\kappa-\lambda-2l}{\kappa +1}{+}2l{-}1\right)}
	=
	\left(\frac{\kappa}{\kappa{+}1}\right)^{2 l} \frac{\Gamma \left(\frac{\lambda{+}1}{\kappa }\right) \Gamma \left(\frac{2l{+}\lambda{+}1}{\kappa{+}1}{-}2 l\right)}{\Gamma \left(\frac{\lambda{+}1}{\kappa }{-}2 l\right) \Gamma \left(\frac{2l{+}\lambda{+}1}{\kappa {+}1}\right)}.
\end{equation}
This suggests the following to renormalization of half of the highest weight vectors, in order to obtain unified formulas for norms.
\begin{notati} \label{not:tilde u}
    The highest weight vectors \(\tilde u_l(\lambda)\) are defined by 
    \begin{equation}
        \tilde u_l(\lambda) = \begin{cases}
          u_l(\lambda)   , l > 0; \\
           \frac{u_l(\lambda)}{\left\|u_l(\lambda)\right\|^2}  , l < 0.
        \end{cases}
    \end{equation}
\end{notati}
%
\begin{corollary}[from Theorem \ref{Th:Norm}] \phantomsection \label{Cor:tilde U}
	For $2l \in \mathbb{Z}$ the norm of the vector of $\tilde{u}_{l}(\lambda)$ is	
	\begin{equation}
	    \left\|\tilde u_{l}(\lambda)\right\|^2 = \left(\frac{\kappa}{\kappa + 1}\right)^{2 l} \frac{\Gamma \left(\frac{\lambda +1}{\kappa }\right) \Gamma \left(\frac{2 l+\lambda +1}{\kappa +1} - 2 l\right)}{\Gamma \left(\frac{\lambda +1}{\kappa } - 2 l\right) \Gamma \left(\frac{ 2 l+\lambda +1}{\kappa +1}\right)} = \frac{\mathtt{t}^{1,-\frac{1}{\kappa}}_{-2 l}(-\frac{\lambda}{\kappa})}{\mathtt{t}^{1,-\frac{1}{\kappa}}_{-2 l}(-\frac{\lambda+1}{\kappa}) }. 
	\end{equation}  
\end{corollary}
Here we used the following notation 
for the product over an integral \textit{triangle}
\begin{equation}\label{eq:triangle}
	\mathtt{t}^{\epsilon_1,\epsilon_2}_n(\alpha)= 
		\begin{cases}
			\prod_{i,j \ge 0, i+j<n} (\alpha-i\epsilon_1-j\epsilon_2) \quad  &n>0 
			\\ 
			\prod_{i,j > 0, i+j\le-n} (\alpha+i\epsilon_1+j\epsilon_2) \quad   &n<0
		\end{cases}.
\end{equation}
We will also use notation for the product over an integral \textit{segment}
\begin{equation}\label{GammaS}  
	\mathtt{s}_n^{\epsilon_1,\epsilon_2}(\alpha)= \frac{\mathtt{t}^{\epsilon_1,\epsilon_2}_n(\alpha)}{\mathtt{t}^{\epsilon_1,\epsilon_2}_{n-1}(\alpha)} =\left(\epsilon_2 - \epsilon_1 \right)^n\frac{\Gamma(\frac{\alpha+\epsilon_1 - n \epsilon_2}{\epsilon_2-\epsilon_1}+n+1)}{\Gamma(\frac{\alpha+\epsilon_1 - n \epsilon_2}{\epsilon_2-\epsilon_1}+1)} .
\end{equation}
These functions have useful symmetry properties
\begin{align}
	&\mathtt{s}^{\epsilon_1,\epsilon_2}_n(\alpha) \mathtt{s}^{\epsilon_1,\epsilon_2}_{-n}(-\alpha - \epsilon_1 - \epsilon_2) = (-1)^n,
	\\
	\label{eq:t symmetry}
	&\mathtt{t}^{\epsilon_1,\epsilon_2}_n(\alpha) = \mathtt{t}^{\epsilon_1,\epsilon_2}_{-n-1}(-\alpha - \epsilon_1 - \epsilon_2) (-1)^{n(n+1)/2}.
\end{align}

\begin{remark}\label{Rem:RT proof}
	Below, we will give the computational proof of Theorem~\ref{Th:Norm}. However, there is also a representation-theoretic proof based on Kac-Kazhdan theorem (see Theorem \ref{Th:KacKazhdan}). Let us illustrate this with the first non-trivial vector \(u_{-1/2}(\lambda)\). Its norm is equal to \((\lambda+1)/\lambda\), see formula \eqref{eq:example norms}, i.e., it has one zero and one pole.
	
	The existence of the pole at \(\lambda=0\) in \(\left\|u_{-1/2}(\lambda)\right\|^2\) comes from the fact that for \(\lambda=0\) the module \(\mathcal{M}_{\lambda,k}\) has a singular vector \(f_0 v_{\lambda,k}\). Hence, the vector \(v_{1/2}\otimes f_0 v_{\lambda,k}\in \mathcal{L}_{1,1}\otimes \mathcal{M}_{\lambda,k}\) is the highest weight vector with   zero norm. This vector does not agree with normalization \eqref{eq:ul normalization}, actually as one can see from the formula \eqref{eq:ul Verma 1/2} that the vector \(u_{-1/2}(\lambda)\) has a pole at \(\lambda=0\) with residue given by \(v_{1/2}\otimes f_0 v_{\lambda,k}\).
	
	On the other hand, the zero \(\lambda+1\) in \(\left\|u_{-1/2}(\lambda)\right\|^2\) corresponds to the fact that the module \(\mathcal{M}_{\lambda+1,k+1}\) on the right side of decomposition \eqref{eq:W2} has a singular vector. This singular vector (with zero norm) is equal to \(u_{-1/2}(\lambda)\). Namely, one can see from the formula \eqref{eq:ul Verma 1/2} that the vector \(u_{-1/2}(\lambda)\) for \(\lambda=-1\) is equal to \(f_0^\Delta u_{1/2}(\lambda)\). This also implies that decomposition \eqref{eq:W2} fails for this \(\lambda\).
	
	As a more subtle example let us consider the vector \(u_{-1}(\lambda)\). Its expression \eqref{eq:ul Verma 1} has a pole at \(\lambda=0\), while the norm  \eqref{eq:example norms} has neither pole, nor zero at this point. Representation-theoretic reason for this is that, for \(\lambda=0\), the modules on both left and right sides of decomposition \eqref{eq:W1} has singular vectors. Namely, let \(v_{-2}=f_0v_0\in \mathcal{L}_{0,k}\) be the singular vector on the left side. It generates a submodule isomorphic to \(\mathcal{M}_{-2,k}\) and one can then study the submodule in the tensor product \(\mathcal{L}_{0,1}\otimes \mathcal{M}_{-2,k}\subset \mathcal{L}_{0,1}\otimes \mathcal{M}_{0,k}\) bearing in mind the decomposition \eqref{eq:W1}. Then, the vector \(u_1(-2)\) is a highest weight vector with highest weight \((0,k+1)\) and \(f_0^\Delta u_1(-2)\) is a singular vector. The norm of \(f_0^\Delta u_1(-2)\) should have a double zero at \(\lambda=0\) (since both \(f_0v_0\) and \(f_0u_1(-2)\) have	 simple zeroes). On the other hand, this vector is proportional to the residue \(\operatorname{Res}_{\lambda=0}u_1(\lambda)\sim f_0u_1(-2)\).

\end{remark}

\subsection{Calculation of the norms}
Let us define the operator  $J(z):(\mathcal{L}_{0,1}\oplus \mathcal{L}_{1,1})\otimes M_{\lambda, k}\rightarrow \overline{(\mathcal{L}_{0,1}\oplus \mathcal{L}_{1,1})\otimes M_{\lambda - 1, k}}$ by formula
\begin{equation}
	J(z) =  b_1(z)\otimes\mathcal{O}^{(1)}_{0, 1}(z) - b_0(z)\otimes\mathcal{O}^{(1)}_{1, 1}(z) .
\end{equation}
Remark that the operator $J(z)$ can be described as the operator  $I(z)$ dressed by the screening $S(t)$ (see formula~\eqref{omnnu})
\begin{equation}\label{eq:Idress}
    J(z) = \int_0^z dt\, I(z) S(t).
\end{equation}
Note that the contour chosen here works only under some inequalities for the parameters. For other values, the definition is extended via analytic continuation.
\begin{proposition}\phantomsection \label{Th:Jprop}
    1) The operator $J(z)$ commutes with $\widehat{\mathfrak{sl}}(2)_{k+1}^\Delta$;
    
    2) The operator $J(z)$ is $\mathrm{Vir}^{\mathrm{coset}}$ vertex operator $\Phi_{2,1}(z)$.
\end{proposition}
\begin{proof}
	The same as proof of Proposition~\ref{Th:Iprop}.
\end{proof}

We will also need screening current $	S(t) = :\! \re^{-\sqrt{\frac{2}{\kappa}}\varphi(t)}\!\!:\beta(t)$ conjugated by $\re^{-g_0}$. Namely, we have
 \begin{equation}\label{eq:Stilde}
     \tilde{S}(t) =\re^{g_0}S(t)   \re^{-g_0} = :\! \re^{-\sqrt{\frac{2}{\kappa}}\varphi(t)}\!\!: \big(\beta(t) - e^{(1)}(t)\big).
 \end{equation}
Using Proposition \ref{Prop:I conj} we get
\begin{equation}\label{eq:Jconj}
    \re^{g_0}J(z)\re^{-g_0} =\int_0^z dt \, :\! \re^\frac{\varphi(z)}{\sqrt{2\kappa}} \!\!:  b_1(z)   \widetilde S(t).
\end{equation}


Recall that $X^\dagger$ denotes the operator conjugated to an operator $X$ with respect to the Shapovalov form (see Definitions \ref{Def:Shap sl2}, \ref{Def:Shap Vir}).

\begin{proposition} 
	There are the following formulas for the conjugation. 
	\begin{enumerate}
	    \item For operators $b_{i}(z): \mathcal{L}_{j,1} \rightarrow \mathcal{L}_{1-j,1}$,  $i,j\in\{0,1\}$, we have
	    \begin{equation}\label{eq:b dagger}
	        (b_{0}(z))^\dagger = (-1)^j  z^{-1/2} b_{1}(z^{-1}), ~ (b_{1}(z))^\dagger = -(-1)^{j} z^{-1/2} b_{0}(z^{-1});
	    \end{equation}
	    \item For operator $\mathcal{O}_{i,1}^{(0)}(z):\mathcal{M}_{\lambda,k} \rightarrow \mathcal{M}_{\lambda + 1,k}$, where $i\in\{0,1\}$ we have
	    \begin{equation}\label{eq:C kappa lambda}
      (\mathcal{O}_{0,1}^{(0)}(z))^\dagger =  C(\kappa,\lambda)   z^{\frac{- 3}{2\kappa}}\mathcal{O}_{1,1}^{(1)}(z^{-1}), ~  (\mathcal{O}_{1,1}^{(0)}(z))^\dagger = C(\kappa,\lambda) (-1) z^{\frac{- 3}{2\kappa}}\mathcal{O}_{0,1}^{(1)}(z^{-1}).
	    \end{equation}
	\end{enumerate}
\end{proposition}
We will actually compute the proportionality coefficient $ C(\kappa,\lambda)$ in the proof of Theorem~\ref{Th:normrec}.
\begin{proof}
	Let \(b_i(z)=\sum b_{i,n}z^{-n-1/4}\). Then we have \([h_{k}, b_{i,s}]  = (-1)^i b_{i,k+s}\) and
	\begin{equation}
		[h_{k}, (b_{1-i,-s})^\dagger] =-([h_{-k}, b_{1-i,-s}])^\dagger =(-1)^i (b_{1-i,-k-s})^\dagger .
	\end{equation}
	Hence  operators $b_{i}(z)$ and $b_{1-i}(z^{-1})$ satisfy the same commutation relations with Heisenberg algebra generators \(h_k\), \(k \in \mathbb{Z}\). The vertex operator between two Fock modules $\mathcal{F}_\alpha$ and $\mathcal{F}_{\beta}$ is uniquely determined by these relations up to overall constant. This constant is fixed by the action on the highest weight vectors.
%
 
	The proof for the operators $\mathcal{O}^{(0)}_{i,1}(z)$ is similar. Namely it is easy to see that the corresponding operators satisfy the same commutation relations with $\widehat{\mathfrak{sl}}(2)$. Then, the proportionality follows from Proposition \ref{onlyonevertex}.
\end{proof}
\begin{corollary}\phantomsection \label{Cor:conjI}
	The conjugation of the operator $I(z) \colon \mathcal{L}_{j,1}\otimes \mathcal{M}_{\lambda,k} \rightarrow \overline{\mathcal{L}_{1 - j,1}\otimes \mathcal{M}_{\lambda + 1,k}}$ has the form
	\begin{equation}
	   (I(z))^\dagger = C(\kappa,\lambda) (-1)^{j}z^{-\frac{ 3}{2\kappa}- \frac{1}{2}} J(z^{-1}).
	\end{equation}
	Here the proportionality coefficient $C(\kappa,\lambda)$ is defined in \eqref{eq:C kappa lambda}.
\end{corollary}
Now we can compute conjugation of the Proposition~\ref{Prop:I vl}.
\begin{corollary}\phantomsection \label{helpfulnorms}
	For \(l\le 0\) we have the formula for the action of operator $J(z)$ on the highest weight vectors
	\begin{equation}
		J(z) u_{l}(\lambda{+}1) = z^{l-\frac{\lambda+3}{\kappa}}\left( (-1)^{2(l+1/2)(l+1)}C(\kappa,\lambda)^{-1} \frac{\|u_{l}(\lambda + 1)\|^2}{\|u_{l + \frac{ 1}{2}}(\lambda)\|^2} u_{l + \frac{1}{2}}(\lambda) + O(z) \right).
	\end{equation}
\end{corollary}

Equivalently we have the formula for the action of operator $\re^{g_0} J(z) \re^{-g_0}$ on the extremal vectors 
\begin{equation}\label{eq:conjJ}
	\re^{g_0} J(z) \re^{-g_0} v_{l}(\lambda{+}1)= z^{l-\frac{\lambda+3}{\kappa}}\left( (-1)^{2(l+1/2)(l+1)}C(\kappa,\lambda)^{-1} \frac{\|u_{l}(\lambda + 1)\|^2}{\|u_{l + \frac{ 1}{2}}(\lambda)\|^2}v_{l + \frac{1}{2}}(\lambda) + O(z) \right).    
\end{equation}

We will prove Theorem \ref{Th:Norm} by induction on $l$. It is sufficient to prove the following statement.

\begin{theorem}\phantomsection \label{Th:normrec}
	For \(l\in \mathbb{Z}_{<0}\) the ratio of the norms of the highest vectors in coset decomposition is equal to ratio of two Beta functions
	\begin{equation}
		\frac{\|u_{l}(\lambda + 1)\|^2}{\|u_{l+ \frac{1}{2}}(\lambda)\|^2} =  \frac{B(-\frac{1}{\kappa},2l-\frac{\lambda+1}{\kappa}+ 1)}{B(-\frac{1}{\kappa},-\frac{\lambda+1}{\kappa} + 1)}.   
	\end{equation}
\end{theorem}

\begin{proof}
	It follows from the formula (\ref{eq:conjJ}) that
	\begin{equation}\label{eq:2n}
		\left\langle v_{l+\frac{1}{2}}(\lambda) ,\, \re^{g_0}J(1)\re^{-g_0}  v_{l}(\lambda  + 1)\right\rangle =(-1)^{2(l+1/2)(l+1)}C(\kappa,\lambda)^{-1} \frac{\|u_{l}(\lambda + 1)\|^2}{\|u_{l + \frac{ 1}{2}}(\lambda)\|^2}.
	\end{equation}
	On the other hand we can use formula (\ref{eq:Jconj})
	\begin{equation}
		\text{LHS of \eqref{eq:2n}} = \int_0^1 	dt \left\langle v_{l + \frac{1}{2}}(\lambda), :\! \re^\frac{\varphi(1)}{\sqrt{2\kappa}}\!\!:   b_1(1)  \widetilde{S}(t)  v_{l}(\lambda + 1)\right\rangle. 
	\end{equation}
	Finally, we use definition of the $b_1(z)$ and $\widetilde S(t)$. 
	\begin{multline}
		\int_0^1 dt \left\langle v_{l + \frac{1}{2}}\otimes v_{\lambda, k} ,  :\! \re^\frac{\varphi(1)}{\sqrt{2\kappa}} \!\!:    b_1(1)   :\! \re^{-\sqrt{\frac{2}{\kappa}}\varphi(t)} \!\!:  \big(\beta(t) - e^{(1)}(t)\big)  v_{l}\otimes v_{\lambda  + 1, k}\right\rangle =
		\\
		= (-1)^{2(l+1/2)(l+1)+1} \int_0^1 dt \left\langle v_{l+\frac{1}{2}} ,\, :\! \re^{-\frac{\phi(1)}{\sqrt{2}}}\!\!:  \,    :\!\re^{\sqrt{2}\phi(t)} \!\!:   v_{l}\right\rangle \left\langle  v_{\lambda, k} , :\!\re^\frac{\varphi(1)}{\sqrt{2\kappa}} \!\!:   \,  :\!\re^{-\sqrt{\frac{2}{\kappa}}\varphi(t)} \!\!:   v_{\lambda  + 1, k}\right\rangle =
		\\
		=(-1)^{2(l+1/2)(l+1)+1}\int_0^1 dt  (1 - t)^{-1 - \frac{1}{\kappa}} t^{2l - \frac{\lambda +1}{\kappa}} =  (-1)^{2(l+1)(l+1/2)+1} B(-\frac{1}{\kappa},2l-\frac{\lambda+1}{\kappa}+1).
	\end{multline}
	Hence,
	\begin{equation}
		\frac{\|u_{l}(\lambda + 1)\|^2}{\|u_{l+\frac{1}{2}}(\lambda)\|^2} =-C(\kappa,\lambda)  B(-\frac{1}{\kappa},2l-\frac{\lambda+1}{\kappa} + 1).
	\end{equation}
	Using values for \(\|u_{0}(\lambda)\|^2\) and \(\|u_{-1/2}(\lambda)\|^2\) given in Example \ref{Exa:norms} we find that \(C(\kappa,\lambda) = -B(-\frac{1}{\kappa};-\frac{\lambda+1}{\kappa}  + 1)^{-1}\). 
\end{proof}

\begin{remark}\phantomsection \label{Rem:l>0}
	Using conjugated operators we can also write formula for highest vectors with \(l\ge 0\), namely 
	\begin{equation}\label{eq:vl l>0}
		u_l(\lambda)=\re^{g_0^\dagger} v_l(\lambda).
	\end{equation}
	This formula is proven similarly to the second proof of Theorem~\ref{Th:FormUm}. Namely, the analog of Proposition~\ref{Prop:I vl} holds for vectors given by \eqref{eq:vl l>0} and current \(I(z)^\dagger\).
\end{remark}

\section{Matrix elements} \label{sec:main}
\subsection{Main Theorem}
Recall the notations of subsection \ref{ssec:operstatecorr}. For  $i + j  + r \in 2\mathbb{Z}$ we have a map 
 \begin{equation}
     \mathtt{Y}_1\otimes \mathtt{Y}_k(x,z):\big(\mathcal{L}_{j,1}\otimes\mathcal{M}_{\nu,k} \big)\rightarrow \operatorname{Hom} \big(\mathcal{L}_{i,1}\otimes\mathcal{M}_{\lambda,k}, ~ \overline{\mathcal{L}_{r,1}\otimes\mathcal{M}_{\mu,k}} \big).
 \end{equation}
\begin{notati}
	We denote  
	\begin{equation}\label{eq:U nu n def}
	    \widetilde{\mathcal{U}}_{\nu,n}(x,z) = \mathtt{Y}_1\otimes \mathtt{Y}_k\big(\tilde  u_{n}(\nu) ; x,z\big), \qquad \mathcal{U}_{\nu,n}(x,z) = \mathtt{Y}_1\otimes \mathtt{Y}_k\big( u_{n}(\nu) ; x,z\big). 
	\end{equation}
\end{notati}
\begin{example} 
	The simplest examples of the vertex operators \(\widetilde{\mathcal{U}}_{\nu,n}\) has the form
	\begin{subequations}
		\begin{align}
			\widetilde{\mathcal{U}}_{\nu,0}(x,z) &= 1\otimes \mathcal{V}_{\nu}(x,z),
			\\
			\widetilde{\mathcal{U}}_{\nu,1/2}(x,z) &= b(x,z)\otimes \mathcal{V}_{\nu}(x,z), 
			\\
			\label{eq:V-1/2}
			\widetilde{\mathcal{U}}_{\nu,-1/2}(x,z) &= \frac{\nu}{\nu + 1}\Big(\partial_x b(x,z)\otimes \mathcal{V}_{\nu}(x,z)   - \frac{1}{\nu} b(x,z)\otimes \partial_x \mathcal{V}_{\nu}(x,z) \Big).
		\end{align}	
	\end{subequations}
\end{example}
The following proposition is an operator reformulation of the fact that $\tilde u_n(\nu)$ is a highest weight vector with respect to $\widehat{\mathfrak{sl}}(2)_{k+1}$ and coset Virasoro.
\begin{proposition}\phantomsection \label{Th:Vnunprimary}
    We have
    \begin{subequations}
	    \begin{align}
			[e^{\Delta}_r,   \widetilde{\mathcal{U}}_{\nu,n}(x, z)] &= z^r(-x^2 \partial_{x} + (\nu + 2n) x) \widetilde{\mathcal{U}}_{\nu,n}(x, z),
			\\
			[h^{\Delta}_r,   \widetilde{\mathcal{U}}_{\nu,n}(x, z)] &= z^r(-2 x \partial_{x} +  (\nu + 2n)) \widetilde{\mathcal{U}}_{\nu,n}(x, z),
			\\
			[f^{\Delta}_r,  \widetilde{\mathcal{U}}_{\nu,n}(x, z)] &= z^r \partial_{x}\, \widetilde{\mathcal{U}}_{\nu,n}(x, z),
			\\			
			[L^{\mathrm{coset}}_r, \widetilde{\mathcal{U}}_{\nu,n}(x, z)] &=  z^{r + 1}\partial_z\, \widetilde{\mathcal{U}}_{\nu,n}(x, z) + (r + 1)z^l \Delta(P + n b, b) \widetilde{\mathcal{U}}_{\nu,n}(x, z),
		\end{align}    	
    \end{subequations}
    where $b = b_{GKO}$ and $P(\lambda) = P_{GKO}(\lambda)$.   
\end{proposition}
\begin{notati}
    Define three-point function
	\begin{equation} \label{eq:tilde C}
	    \tilde{C}_{m,n,l}(\mu,\nu,\lambda) = \big\langle \tilde{u}_m(\mu), \widetilde{\mathcal{U}}_{\nu,n}(1,1) \tilde{u}_l(\lambda)\big\rangle.
	\end{equation}	
\end{notati}
It follows from fusion rules \eqref{eq:fusion level 1} that \(\tilde{C}_{m,n,l}\) vanishes unless \(l+m+n \in \mathbb{Z}\). 

Note that, since normalization of \(\mathcal{U}_\nu\) is not given, each individual \(\tilde{C}_{m,n,l}\) is not defined.  However, the ratios 
are well defined. For example, one can ask for the ratio with the three-point functions of highest weight vectors
\begin{equation}
	\frac{\tilde{C}_{m,n,l}(\mu,\nu,\lambda)}{\tilde{C}_{\{m\},\{n\},\{l\}}(\mu,\nu,\lambda)}.
\end{equation}
Note that we fixed normalization of \(b_0(z), b_1(z)\) in the formula \eqref{eq:b01} such that
\begin{multline}\label{eq: C base}
	\langle v_{\mu,k},\mathcal{U}_{\nu}(1,1) \,v_{\lambda,k}\rangle
	=\langle \tilde{u}_{0}(\mu)\ \widetilde{\mathcal{U}}_{\nu,0}(1,1) \,\tilde{u}_{0}(\lambda)\rangle
	= \langle \tilde{u}_{{1}/{2}}(\mu)\ \widetilde{\mathcal{U}}_{\nu,0}(1,1) \,\tilde{u}_{{1}/{2}}(\lambda)\rangle 
	\\ = \langle \tilde{u}_{0}(\mu)\ \widetilde{\mathcal{U}}_{\nu,1/2}(1,1) \,\tilde{u}_{{1}/{2}}(\lambda)\rangle 
	= \langle \tilde{u}_{{1}/{2}}(\mu)\ \widetilde{\mathcal{U}}_{\nu,1/2}(1,1) \,\tilde{u}_{0}(\lambda)\rangle.
\end{multline}
Therefore, we can normalize using \(\tilde{C}_{0,0,0}(\mu,\nu,\lambda)\).

\begin{theorem} \phantomsection \label{Th:main}
	There is the following formula for matrix elements \(\tilde{C}_{m,n,l}(\mu,\nu,\lambda)\) 
	\begin{multline}\label{eq:td C t}
		\frac{\tilde{C}_{m,n,l}(\mu,\nu,\lambda)}{\tilde{C}_{0,0,0}(\mu,\nu,\lambda)} 
		= (-1)^{\frac{1}{2} (l-m+n) (l-m+n+1)+4 n (m-n) ( m- n-\frac{1}{2})} 
		\\
		\cdot \frac{\mathtt{t}^{1,-\frac{1}{\kappa}}_{- l - m - n}(\frac{2 + \lambda + \mu+\nu}{-2\kappa})\mathtt{t}^{1,-\frac{1}{\kappa}}_{-l + m - n}(\frac{\lambda - \mu+\nu}{-2\kappa})\mathtt{t}^{1,-\frac{1}{\kappa}}_{- l - m + n}(\frac{\lambda + \mu - \nu}{-2\kappa})\mathtt{t}^{1,-\frac{1}{\kappa}}_{l-m-n}(\frac{ - \lambda + \mu+\nu}{-2\kappa})}{\mathtt{t}^{1,-\frac{1}{\kappa}}_{-2 l}(\frac{\lambda+1}{-\kappa})\mathtt{t}^{1,-\frac{1}{\kappa}}_{-2 m}(\frac{\mu+1}{-\kappa})\mathtt{t}^{1,-\frac{1}{\kappa}}_{-2 n}(\frac{\nu + 1}{-\kappa})},
	\end{multline}    
	where $l+m+n\in\mathbb{Z}$ and $\mathtt{t}^{\epsilon_1,\epsilon_2}_n(\alpha)$ defined by formula (\ref{eq:triangle}).
\end{theorem}

Let us prove the theorem modulo the results to be proven later in this section.

\begin{proof} The theorem is proven by induction. The equation \eqref{eq: C base} serve as base of the induction. The step is based on the following two observations.


     \textbf{Observation 1} The three-point function satisfies recurrence relation on $l$. This relation is formulated in the Proposition \ref{Prop:lrecurrence}. It is proven by considering a four-point conformal block with the insertion of a degenerate field.
        
	\textbf{Observation 2} The three-point function is symmetric with respect to permutation of the indices $n,m,l$. Namely we have
	\begin{subequations}\label{eq:sym}
	    \begin{align}
	    	\label{eq:ln sym}
            (-1)^{2l(l+1)(2l-1)+2n(n+1)(2n-1)+2m(m+1)(2m-1)}
            \frac{\tilde{C}_{m,l,n}(\mu,\lambda,\nu)}{\tilde{C}_{0,0,0}(\mu,\lambda,\nu)}
            =\frac{\tilde{C}_{m,n,l}(\mu,\nu,\lambda)}{\tilde{C}_{0,0,0}(\mu,\nu,\lambda)}\\
            \label{eq:lm sym}
            (-1)^{l-m+4n(m-n^2)}\frac{\tilde{C}_{l,n,m}(\lambda,\nu,\mu)}{\tilde{C}_{0,0,0}(\lambda,\nu,\mu)}
           	=\frac{\tilde{C}_{m,n,l}(\mu,\nu,\lambda)}{\tilde{C}_{0,0,0}(\mu,\nu,\lambda)}. 	
	    \end{align} 
	\end{subequations}  
	We prove them in the Section \ref{ssec:symm3pt}. Essentially these symmetries follow from the symmetry of the map $\mathtt{m}$ (i.e. coinvariants) under the permutation of points.
\end{proof}

Theorem~\ref{Th:main} supersedes all concrete computations of matrix elements which are performed in this paper.

\begin{example}
	Consider particular case \(n=0, \nu=0, l=m, \lambda=\mu\). Then we have
	\begin{equation}
		\tilde{C}_{l,0,l}(\lambda,0,\lambda) =  \|\tilde u_{l}(\lambda)\|^2.
	\end{equation}	
	This case corresponds to the matrix elements of the identity operator.
\end{example}
\begin{example}\phantomsection \label{Exa:I}
	Consider particular case \(n = -\frac{1}{2}, \nu=1 \). The corresponding operator is given by the formula \eqref{eq:V-1/2}. Assume now that \(\mu=\lambda+1\), then  \(\widetilde{\mathcal{U}}_{-1/2,1}\) coincides with $\frac{1}{2}I(z)$ and using formula \eqref{actI} and Notation \ref{not:tilde u} we have
	\begin{equation}
		\tilde{C}_{l-1/2,-1/2,l}(\lambda+1,1,\lambda) \sim  \left\langle \tilde{u}_{l-1/2}(\lambda+1), I(1) \tilde{u}_{l}(\lambda) \right\rangle =  \| \tilde{u}_{l}(\lambda)\|^2.
	\end{equation}
	Note that above we proved formula \eqref{actI} only for \(l\le 0\) but now we have an analog valid for any \(l\). In particular this determines coefficients \(c_{l,\lambda}\) which appeared in the proof of Proposition~\ref{Prop:I vl}. Here \(\sim\) stands for overall factor which does not depend on \(l\) and hidden in the normalization of vertex operator \(\mathcal{V}_{1}\).
	
	Similarly, for \(\mu = \lambda-1\) certain component of  \(\widetilde{\mathcal{U}}_{-1/2,1}\) coincides with $\frac{1}{2}J(z)$. Then using Corollary \ref{helpfulnorms} we have 
	\begin{equation}
		\tilde{C}_{l+1/2,-1/2,l}(\lambda-1,1,\lambda) \sim  \left\langle \tilde{u}_{l+1/2}(\lambda-1), J(1) \tilde{u}_{l}(\lambda) \right\rangle \sim   \| \tilde{u}_{l+1/2}(\lambda-1)\|^2.
	\end{equation}
%
%
\end{example}

\begin{example}\label{Exa:b}
	Consider particular case \(n = \frac{1}{2}, \nu=0, m = l \pm \frac{1}{2}, \lambda=\mu\). Then we have
	\begin{equation}
		\tilde{C}_{l+\frac{1}{2},\frac{1}{2},l}(\lambda,0,\lambda) \sim  	(-1)^{2l(l-\frac{1}{2})}\|\tilde u_{l}(\lambda)\|^2 , ~~ \tilde{C}_{l-1/2,\frac{1}{2},l}(\lambda,0,\lambda) \sim (-1)^{2l(l-\frac{1}{2})} \|\tilde u_{l-\frac{1}{2}}(\lambda)\|^2.
	\end{equation}	
	This case corresponds to matrix elements of $b(x,z)$. We calculate them below during the proof of the main theorem, see Proposition \ref{Prop:b(x,z)}.
\end{example}
 
\begin{remark} 
	In the case $\lambda +\nu -\mu=2N$ or  $\lambda -\nu - 2 -\mu=2N $ one can use bosonizations of the vertex operator given on Propositions \ref{Prop:O&Vnu 1}, \ref{Prop:O&Vnu 2}. This leads to new  Selberg-type integrals, see Theorem~\ref{Th:Selberg}.
\end{remark}

The remaining part of this section is organized as follows.
In section \ref{ssec:symm3pt} we study symmetries of the three-point function and prove the relations \eqref{eq:sym}. Then we focus on the proof of Proposition~\ref{Prop:lrecurrence}. 

The idea is to consider four-point conformal block 
\(\langle \tilde u_{m}(\mu), \widetilde{\mathcal{U}}_{\nu,n}(1,1) b(x, z) \tilde u_{l}(\lambda)\rangle\) 
with insertion of degenerate field $b(x,z)$. We study properties of operator $b(x,z)$ in Section \ref{ssec:matrixelb}. In Sections \ref{ssec:VirasoroConformalBlocks} and \ref{ssec:sl2ConformalBlocks} we recall definitions of conformal blocks, BPZ and KZ equations, and their solutions in terms of hypergeometric functions.  We put all things together in Section~\ref{ssec:coset blocks}. Firstly we prove Proposition~\ref{Prop:lrecurrence} using triviality of monodromy of the conformal block with insertion of $b(x,z)$, this is a variation of an argument used in \cite{zamolodchikov:1989}, \cite{Teschner:1995Liouville}, \cite{Teschner:1997}, \cite{Bershtein:2015Bilinear}. Then we consider generic conformal blocks and use Theorem~\ref{Th:main} to find conformal blocks relations which are equivalent to the blowup relations.  

%
%
%
%

\subsection{Symmetries of three-point functions } \label{ssec:symm3pt}

Recall that we have a unique up to a constant map ($i + j  + r \in 2\mathbb{Z}$)
\begin{equation}
	\mathtt{m} = \mathtt{m}_1 \otimes \mathtt{m}_k : \big(\mathcal{L}_{i,1}\otimes\mathcal{M}_{\lambda,k}(0,0) \big)\otimes \big(\mathcal{L}_{j,1}\otimes\mathcal{M}_{\nu,k}(1,1) \big)\otimes  \big(\mathcal{L}_{r,1}\otimes\mathcal{M}_{\mu,k}(\infty,\infty) \big)   \rightarrow \mathbb{C}.
\end{equation}
The matrix elements $\tilde{C}_{m,n,l}(\mu,\nu,\lambda)$ are actually defined through this map 
\begin{equation}
	\frac{\tilde{C}_{m,n,l}(\mu,\nu,\lambda)}{\tilde{C}_{\{m\},\{n\},\{l\}}(\mu,\nu,\lambda)} 
	= \frac{\mathtt{m}\left(\alpha^{(\infty)}_{0,0}\big(\tilde{u}_m(\mu)\big)\otimes \alpha^{(0)}_{1,1}\big(\tilde{u}_n(\nu)\big)\otimes\alpha^{(0)}_{0,0} \big(\tilde{u}_l(\lambda)\big) \right)}{\mathtt{m}\left(\alpha^{(\infty)}_{0,0}\big(\tilde{u}_{\{m\}}(\mu)\big)\otimes \alpha^{(0)}_{1,1}\big(\tilde{u}_{\{n\}}(\nu)\big)\otimes\alpha^{(0)}_{0,0} \big(\tilde{u}_{\{l\}}(\lambda)\big) \right)},
\end{equation}
where \(\{n\}\) denotes fractional part of \(n\) and \(\lfloor n \rfloor\) denotes floor. Moreover, one can consider generic positions of the insertions and Borel subalgebras in three-point function 
\begin{multline}
	\frac{\mathtt{m}\left(\alpha^{(\infty)}_{y_1,w_1}\big(\tilde{u}_m(\mu)\big)\otimes \alpha^{(0)}_{x_2,z_2}\big(\tilde{u}_n(\nu)\big)\otimes\alpha^{(0)}_{x_3,z_3} \big(\tilde{u}_l(\lambda)\big) \right)}{\mathtt{m}\left(\alpha^{(\infty)}_{y_1,w_1}\big(\tilde{u}_{\{m\}}(\mu)\big)\otimes \alpha^{(0)}_{x_2,z_2}\big(\tilde{u}_{\{n\}}(\nu)\big)\otimes\alpha^{(0)}_{x_3,z_3} \big(\tilde{u}_{\{l\}}(\lambda)\big) \right)}=\frac{\tilde{C}_{m,n,l}(\mu,\nu,\lambda)}{\tilde{C}_{\{m\},\{n\},\{l\}}(\mu,\nu,\lambda)}
	\\
	(1+y_1x_2)^{\lfloor m\rfloor + \lfloor n\rfloor -\lfloor l \rfloor} (1+y_1x_3)^{\lfloor m\rfloor + \lfloor l\rfloor -\lfloor n \rfloor} (x_2-x_3)^{\lfloor n\rfloor + \lfloor l\rfloor -\lfloor m \rfloor}
	\\
	(1-w_1z_2)^{l^2{-}\{l\}^2{-}m^2{+}\{m\}^2{-}n^2{+}\{n\}^2} (1-w_1z_3)^{n^2{-}\{n\}^2{-}l^2{+}\{l\}^2-m^{+}\{m\}^2} (z_2-z_3)^{m^2{-}\{m\}^2{-}n^2{+}\{n\}^2{-}l^2{+}\{l\}^2} 
	.
\end{multline}
We can swap \((x_2,z_2)\) and \((x_3,z_3)\). Since \(x_2-x_3\) and \(z_2-z_3\) change the sign after such swapping we get a symmetry with a sign factor given by
\begin{equation}
	(-1)^{\lfloor n\rfloor + \lfloor l\rfloor +\lfloor m \rfloor+m^2{-}\{m\}^2{+}n^2{-}\{n\}^2{+}l^2{-}\{l\}^2}\frac{\tilde{C}_{m,l,n}(\mu,\nu,\lambda)}{\tilde{C}_{\{m\},\{l\},\{n\}}(\mu,\nu,\lambda)}=\frac{\tilde{C}_{m,n,l}(\mu,\nu,\lambda)}{\tilde{C}_{\{m\},\{n\},\{l\}}(\mu,\nu,\lambda)}.
\end{equation}
By straightforward (and not illuminating) case-by-case check on can see that this formula is equivalent to~\eqref{eq:ln sym}.

In order to see another symmetry, we swap the first and third points and also replace the choice of the isomorphism \(\alpha_{x,z}\) in them. Note that for any \(n,x,z\) the vectors \(\alpha^{(0)}_{x,z}\big(\tilde{u}_n(\lambda)\big)\) 
and  \(\alpha^{(\infty)}_{y,w}\big(\tilde{u}_n(\lambda)\big)\) are 
proportional (here \(y=-x^{-1}\) and \(w=z^{-1}\) as above). Indeed, these vectors are \(\mathrm{Vir}^{\mathrm{coset}} \oplus \widehat{\mathfrak{sl}}(2)_{x,z,k + 1}\) highest weight vectors for the same Borel subalgebra \(\overline{\mathrm{Vir}}_{z,+}\oplus\widehat{\mathfrak{b}}_{x,z}\), where \(\overline{\mathrm{Vir}}_{z,+}\) denotes subalgebra \(\mathbb{C}[[t-z]](t-z)\partial_t\oplus \mathbb{C}C\in \overline{\mathrm{Vir}}\). Furthermore, these vectors share the same highest weight, therefore they are proportional due to decomposition \eqref{eq:W1 W2}. The proportionality coefficient is equal to
\begin{equation}\label{eq:alpha 0 infty}
	\alpha^{(0)}_{x,z}\big(\tilde{u}_n(\nu)\big)= x^{-2\lfloor n \rfloor} (-z^2)^{n^2-\{n\}^2} \alpha^{(\infty)}_{y,w}\big(\tilde{u}_n(\nu)\big).
\end{equation}
Indeed, the change of the basis in \(\mathfrak{b}_x\) from \((h+2xf,e-xh-x^2f)\) to \((-h-2y e, y^2e+yh -f)\) is performed by conjugation of element in \(\exp(\mathfrak{b}_x)\) which gives the first factor \(x^{-2\lfloor n \rfloor}\). And the transformation from \(t-z\) expansion into \(s-w=-t^{-1}z^{-1}(t-z)\) expansion gives the second term \((-z^2)^{n^2-\{n\}^2}\). 

Using all these signs we get 
\begin{equation}
	(-1)^{n^2-\{n\}^2+\lfloor l \rfloor +\lfloor n \rfloor -\lfloor m \rfloor }\frac{\tilde{C}_{l,n,m}(\lambda,\nu,\mu)}{\tilde{C}_{0,0,0}(\lambda,\nu,\mu)}
	=\frac{\tilde{C}_{m,n,l}(\mu,\nu,\lambda)}{\tilde{C}_{0,0,0}(\mu,\nu,\lambda)}.   
\end{equation}
By case by case consideration this is equivalent to \eqref{eq:lm sym}.

\begin{remark}
	There is another way to state \( l \leftrightarrow m\) symmetry \eqref{eq:lm sym}. First, assume that \(n=0\). It is straightforward to get from the commutation relations that the conjugate operator \(\mathcal{V}_{\nu}(x,z)^\dagger\) is proportional \(\mathcal{V}_{\nu}(-x^{-1},z^{-1})\). Therefore the same holds for \(\widetilde{\mathcal{U}}_{\nu,0}(x,z)\). Since the ratio of matrix elements \(	\frac{\tilde{C}_{m,0,l}(\mu,\nu,\lambda)}{\tilde{C}_{\{m\},0,\{l\}}(\mu,\nu,\lambda)}\) is proportional to \(x^{l-m}\) then after the swapping of \(l,m\) we get sign \((-1)^{l-m}\). 
	
	In terms of conjugation of vertex operators, the proof above (namely the formula \eqref{eq:alpha 0 infty}) asserts that \(\widetilde{\mathcal{U}}_{\nu,n}(x,z)^\dagger\) is proportional to \((-1)^n \widetilde{\mathcal{U}}_{\nu,n}(-x^{-1},z^{-1})\) for integer \(n\). For half-integer \(n\) we have a sign as in formula \eqref{eq:b dagger}.
\end{remark}

\begin{remark}
	The automorphism \(\tau\) introduced in Remark \ref{Rem:tau} also gives symmetry of the three-point functions
	\begin{equation}
		\frac{\tilde{C}_{m,n,l}(\mu,\nu,\lambda)}{\tilde{C}_{0,0,0}(\mu,\nu,\lambda)}= (-1)^{n(2n-1)}\frac{\tilde{C}_{-m+\frac{1}{2},n,-l+\frac{1}{2}}(\kappa - 2 - \mu,\nu,\kappa - 2 -\lambda)}{\tilde{C}_{0,0,0}(\kappa - 2 -\mu,\nu,\kappa - 2 -\lambda)}\|\tilde u_m(\mu)\|^2 \|\tilde u_l(\lambda)\|^2.
	\end{equation}
	The norms on the right side appears from the fact that \(\tau\) inverses the norm of highest weight vectors, see Section~\ref{ssec:norms}.
\end{remark}

\subsection{Matrix elements of $b(x, z)$} \label{ssec:matrixelb}
Recall the level 1 vertex operator \(b(x,z)\) defined in sec.~\ref{ssec:integrable}. We can consider it as an operator acting $\mathcal{L}_{i,1}\otimes\mathcal{M}_{\lambda,k} \rightarrow  \overline{\mathcal{L}_{1-i,1}\otimes\mathcal{M}_{\lambda,k}}$. Taking into account decompositions \eqref{eq:W1 W2} we can ask for the \(\mathrm{Vir}^{\mathrm{coset}} \oplus \widehat{\mathfrak{sl}}(2)_{k + 1} \) description.
	
\begin{proposition}\phantomsection \label{Th:bprop}
	1) The operator $b(x,z)$ corresponds to degenerate vertex operator \(X_{1}(x,z)\) for  $\widehat{\mathfrak{sl}}(2)_{k+1}^\Delta$.
	
	2) The operator $b(x,z)$ is $\mathrm{Vir}^{\mathrm{coset}}$ vertex operator $\Phi_{1,2}(z)$.
\end{proposition}	
\begin{proof}
	The proof is similar to the proof of Proposition~\ref{Th:Iprop}. This operator is obtained by operator-state correspondence map \(\mathtt{Y}\otimes \mathtt{Y}\) applied to the vector
	\begin{equation}\label{eq:v_b}
		v_b = v_{{1}/{2}}\otimes v_{0, k} \in \mathcal{L}_{1,1}\otimes\mathcal{L}_{0,k}.
	\end{equation} 
	Clearly this vector is the highest weight vector for  \(\mathrm{Vir}^{\mathrm{coset}} \oplus \widehat{\mathfrak{sl}}(2)_{k + 1} \). Then it remains to show that 
	\begin{equation}
		(f_0^{\Delta})^2v_b=0,\quad 	    ((L^{GKO}_{-1})^2 + (b_{\text{GKO}})^{2}L^{GKO}_{-2})v_b = 0.
	\end{equation}
\end{proof}	
	
\begin{proposition}\label{Prop:b(x,z)}
	 Nontrivial matrix elements of $b(x,z)$ are equal to norms of the highest weight vectors, namely
	 \begin{subequations}\label{eq:b tildeU}
	 	\begin{align}
	 		\tilde{B}^+_m(\mu) &= \left\langle \tilde{u}_{m + {1}/{2}}(\mu) ,  b_0(1)  \, \tilde{u}_{m}(\mu)\right\rangle = (-1)^{m(2m-1)}\|\tilde{u}_{m}(\mu)\|^2,
	 		\\
	 		\tilde{B}^-_m(\mu)& =\left\langle \tilde{u}_{m-{1}/{2}}(\mu),   b_1(1)   \, \tilde{u}_{m}(\mu)\right\rangle= 
    (-1)^{m(2m-1)} \|\tilde{u}_{m-{1}/{2}}(\mu)\|^2,
	 	\end{align}
	 \end{subequations}
	 while all other matrix elements vanish.		 	 
\end{proposition}

We will use these formulas in the proof of (main) Theorem~\ref{Th:main}. On the other hand, they can be viewed as a particular case of this theorem; see Example~\ref{Exa:b}.

\begin{proof}
	Let us start with the vanishing property. Inserting \(h_0^\Delta\) into the matrix element we get 
	\begin{equation}
		( \lambda + 2 m )\left\langle \tilde{u}_{m}(\lambda), b_0(z)  \,  \tilde{u}_{l}(\lambda)\right\rangle = \left\langle \tilde{u}_{m}(\lambda), h_0^\Delta  b_0(z)   \, \tilde{u}_{l}(\lambda)\right\rangle =
		 (\lambda +2l + 1)  \left\langle \tilde{u}_{m}(\lambda),  b_0(z)   \,  \tilde{u}_{l}(\lambda)\right\rangle.
	\end{equation}
	Hence \(m\neq l+1/2\) the matrix element vanishes. Similarly, for \(b_1(z)\).
	
	We will use formula \eqref{eq:ul Wakimoto} in the proof, so it will be convenient to rewrite the matrix elements~\eqref{eq:b tildeU} in terms of (differently normalized) highest weight vectors \(u_{m}(\mu)\) 	
	\begin{subequations}
		\begin{align}
			B^+_{-l}(\lambda ) &= \left\langle u_{-l +  1/2}(\lambda),  b_0(1)   \, u_{-l}(\lambda)\right\rangle = (-1)^{l(2l+1)}\|u_{-l+1/2}(\lambda)\|^2,
			\\
			B^-_{ -l}(\lambda)&=\left\langle  u_{-l-1/2}(\lambda), b_1(1)  \,u_{-l}(\lambda)\right\rangle = (-1)^{l(2l+1)}\|u_{-l}(\lambda)\|^2,
			\\
			B^+_m(\mu ) &= \left\langle u_{m +  1/2}(\mu),  b_0(1)  \, u_{m}(\mu)\right\rangle = (-1)^{m(2m-1)} \|u_{m}(\mu)\|^2,
			\\
			B^-_m(\mu)&=\left\langle u_{m-1/2}(\mu),   b_1(1)  \,u_{m}(\mu)\right\rangle= (-1)^{m(2m-1)}\|u_{m-1/2}(\mu)\|^2,
		\end{align}	
	\end{subequations}
	for \(l,m \ge 0\). For $B^+_{-l}(\lambda )$ we have
	\begin{multline}
		 \left\langle u_{-l + 1/2}(\lambda),  b_0(z)  \, u_{-l}(\lambda)\right\rangle 
		 = \left\langle v_{-l + 1/2}(\lambda) , \re^{-g_0^\dagger} \re^{-g_0}  b_0(z)   \, v_{-l}(\lambda)\right\rangle =
		 \\
		 = (-1)^{l(2l+1)}\left\langle v_{-l + 1/2}(\lambda) , \re^{-g_0^\dagger} \re^{-g_0}   \, z^{-l} ( v_{-l+1/2}(\lambda) + o(1) )\right\rangle = (-1)^{l(2l+1)} z^{-l} \|u_{-l+1/2}\|^2. 
	\end{multline}
	Here the second summand in $v_{-l+1/2}(\lambda) + o(1)$ represents vectors of \(L_0\) degree greater then  $v_{-l+1/2}(\lambda)$ and hence orthogonal to the left vector. We also used commutativity between \(b_0(z)\) and \(g_0\). Using the conjugation we have
	\begin{multline}
	    (-1)^{l(2l+1)}\|u_{-l+1/2}(\lambda)\|^2  = B^+_{-l}\left(\lambda \right) = \left\langle u_{-l + 1/2}(\lambda),  b_0(1)  \, u_{-l}(\lambda)\right\rangle = \\
	    =
	    \left\langle u_{-l}(\lambda),  b_1(1) (-1)^{-h_0+1} \,u_{-l + 1/2}(\lambda)\right\rangle = (-1)^{2l} B^-_{-l+1/2}(\lambda).
	\end{multline}
	For vectors \(u_m(\lambda)\) we can use formula \eqref{eq:vl l>0}. Similarly to the arguments above we have
	\begin{multline}
		 B^-_{m}(\mu)=\left\langle u_{m-1/2}(\mu),  b_1(z)  \,u_{m }(\mu)\right\rangle = \left\langle v_{m-1/2}(\mu), \re^{g_0} b_1(z) \re^{g_0^\dagger}\,  v_{m}(\mu)\right\rangle
		\\
		=\left\langle v_{m-1/2}(\mu), \re^{g_0} \re^{g_0^\dagger} b_1(z)   \,v_{m}(\mu)\right\rangle= (-1)^{m(2m-1)} \| u_{m-1/2}(\mu)\|^2.
	\end{multline}
	Using conjugation we have 
	\begin{multline}
		 B^+_{m}(\mu)= \left\langle u_{m +  1/2}(\mu),  b_0(1)  \, u_{m}(\mu)\right\rangle
		 \\
		 = \left\langle u_{m }(\mu),  b_1(1) (-1)^{-h_0+1}  \, u_{m+1/2}(\mu)\right\rangle
		 =  (-1)^{m(2m-1)} \| u_{m}(\mu)\|^2.
	\end{multline}
%
%
%
\end{proof} 

\begin{remark}
	The proposition is consistent with the isometry $\tau\otimes\tau$ (see Remark \ref{Rem:tau}).
	\begin{multline}
		(-1)^{l(2l-1)}\|u_{-l+1/2}(\lambda)\|^2 = B^+_{-l}(\lambda) =   \big\langle u_{-l +  1/2}(\lambda),  b_0(1)  \, u_{-l}(\lambda)\big\rangle =
		\\
		=  \big\langle\tau\otimes\tau(u_{-l}(\lambda)) ,  \tau\otimes\tau(b_0(1))  \,  \tau\otimes\tau(u_{-l}(\lambda))\big\rangle = B^-_{l + 1/2}(k - \lambda) = (-1)^{l(2l-1)} \|u_{l}(k - \lambda)\|^2.
	\end{multline}
\end{remark}

\subsection{Virasoro conformal blocks}\label{ssec:VirasoroConformalBlocks}

The space of conformal blocks can be defined using the space of coinvariants as in Section \ref{ssec:Virasoro}. Below we define sphere conformal blocks using vertex operators \(\Phi_\Delta\) introduced there.

\begin{define}
	Assume that there are given parameters \(P_i\), \(i\in \{1,\dots,n\}\) and \(P_j'\), \(j\in \{2,\dots,n-2\}\). Assume also that there are given vertex operators  $\Phi_{\Delta_2}(z)\colon \mathbb{M}_{P_1,b}\rightarrow \overline{\mathbb{M}_{P'_2,b}}$, $\Phi_{\Delta_i}(z)\colon\mathbb{M}_{P'_{i-1},b}\rightarrow \overline{\mathbb{M}_{P'_i,b}}$ for $i\in\{3,\dots,n-2\}$, and 	 $\Phi_{\Delta_{n-1}}(z)\colon \mathbb{M}_{P'_{n-2},b}\rightarrow \overline{\mathbb{M}_{P_n,b}}$, where $\Delta_i=\Delta(P_i,b)$. The  
	$n$-point sphere conformal block for Virasoro algebra defined as
	\begin{equation}
		\mathrm{F}_b(\vec{P},\vec{P'};\vec{z}) =\Big \langle v_{P_n,b},\Phi_{\Delta_{n-1}}(z_{n-1})\dots\Phi_{\Delta_{3}}(z_{3})\Phi_{\Delta_{2}}(z_{2}) v_{P_1,b}\Big\rangle,
	\end{equation}
	where \(\vec{P}=(P_1,\dots,P_n)\), \(\vec{P'}=(P_2',\dots,P_{n-2}')\), \(\vec{z}=(z_2,\dots,z_{n-1})\).
\end{define} 
The conformal block is defined up to the choice of normalization of vertex operators (i.e. independent of~\(\vec{z}\) function) to be specified below. This is \(n\)-point conformal block through only dependence on \(n-2\) coordinate is explicit; the remaining two are fixed \(z_1=0\), \(z_n=\infty\).

In the paper, we will mainly use four-point conformal blocks. In this case we can assume that \(z_2=z\), \(z_3=1\) and we have only one intermediate parameter \(P_2'=P'\). In such case, the definition can be also restated in terms of so-called chain vectors.

\begin{define}\phantomsection \label{Def:virchain}
	Let $\Phi_{\Delta_2}(z)\colon \mathbb{M}_{P_1,b} \rightarrow \overline{\mathbb{M}_{P_3,b}}$, $\Delta_2 = \Delta(P_2, b)$ be a Virasoro vertex operator as in sec.~\ref{ssec:Virasoro}. The vector $\mathbb{W}^{b;P_2}_{P_1,P_3}(z) = \Phi_\Delta(z) v_{P_1,b}\in \overline{\mathbb{M}_{P_3,b}} $ is called \emph{Virasoro chain vector}.
\end{define}
It follows from statements in sec. \ref{ssec:Virasoro} that for generic values of parameters the chain vector $\mathbb{W}^{b;P_2}_{P_1,P_3}(z)$ exists and uniquely fixed by its top component, i.e. by scalar product \(\big\langle v_{P_3,b}, \mathbb{W}^{b;P_2}_{P_1,P_3}(z) \big \rangle\). For conformal blocks we have
\begin{equation}\label{eq:Vir conf block Whit}
		\mathrm{F}_b(\vec{P},P';z) =\big\langle  \mathbb{W}^{b;P_3}_{P_4,P'}(1),\mathbb{W}^{b;P_2}_{P_1,P'}(z) \big\rangle. 
\end{equation}
Unless otherwise stated, we assume that Whittaker vectors are normalized by \(\langle v_{P_3,b}, \mathbb{W}^{b;P_2}_{P_1,P_3}(z)  \rangle= z^{\Delta_3-\Delta_1-\Delta_2}\). Hence conformal blocks have the form \(\mathrm{F}_b(\vec{P},P';z) =z^{\Delta'-\Delta_1-\Delta_2}(1+O(z))\).

There is a special case of Virasoro four-point conformal blocks which will be important for us, namely conformal block with one degenerate vertex operator $\Phi_{12}$. In this case differential equation \eqref{eq: Phi12} leads to the equation on function \(\mathrm{F}\).

%
\begin{theorem}\cite{BPZ:1984}
	The four-point conformal block \(	\mathrm{F}_b(\vec{P},P';z) = \langle\Delta_4,\Phi_{\Delta_{3}}(1)\Phi_{1,2}(z)\, \Delta_1\rangle\)
	with degenerate field $\Phi_{12}$ satisfy differential equation
	\begin{equation}\label{eq:hypergeom Vir}
		\left(z(1-z)\partial_z^2 + b^{2}\Big((2z-1)\partial_z + \frac{\Delta_1}{z} + \frac{\Delta_3}{z-1} + \Delta_{12} - \Delta_4\Big)\right)\mathrm{F}(z) = 0.
	\end{equation}
\end{theorem}
There are two solutions of this equation which correspond to two possible values of \(P'\) namely \(P'=P+s b/2\), where \(s=\pm 1\). These solutions are given by hypergeometric functions. For  
 $\vec{a}  = (a_1, a_2, a_3)\in \mathbb{C}^3$ we will write 
\begin{equation}\label{eq:2F1}
	_2F_1(\Vec{a}|z) = \, _2F_1(a_1, a_2;a_3|z).
\end{equation}
It will be convenient to introduce a transformation of the vectors in $\mathbb{C}^3$
\begin{equation}
	\label{eq:hat I}
	\hat{I}\Vec{a} = (a_1 - a_3 + 1, a_2 - a_3 + 1, 2 - a_3).
\end{equation}
It is easy to see that $\hat{I}^2 = 1$. Furthermore, the functions $ {}_2F_1(\Vec{a}|z) $ and $ z^{1 - a_3} {}_2F_1(\hat{I}\Vec{a}|z) $ satisfy the same hypergeometic equation.


\begin{corollary}\phantomsection \label{Cor:Vir bl} 
	The four-point conformal blocks with degenerate field $\Phi_{12}$ has the form
	\begin{subequations}\label{eq:Vir hypergeom}
		\begin{align}
			\mathrm{F}_b(\vec{P},P_1+b/2;z)&= z^{D_+}(1 - z)^E \, _2F_1(\vec{A}|z),
			\\
			\mathrm{F}_b(\vec{P},P_1-b/2;z)&= z^{D_-}(1 - z)^E \, _2F_1(\hat{I}\vec{A}|z).
		\end{align}		
	\end{subequations}
	where 
	\begin{gather}
		A_1 = - b P_1 - b P_4 - b P_3 + 1/2,\;\; A_2 =  - b P_1 + b P_4 - b P_3 + 1/2;\;\;  A_3 =  1 - 2 b P_1;
	\\
		D_s = \Delta(P_1 + s b/2, b) - \Delta(P_1 , 
		b) - \Delta(P_{12} , b),\;\;   E =  \Delta(P_3 +  b/2, b) - \Delta(P_3 , 
		b) - \Delta(P_{12} , b).
	\end{gather}
\end{corollary}
As was noted above we normalized conformal blocks in the formulas~\eqref{eq:Vir hypergeom}  by \(\mathrm{F}_b(\vec{P},P';z) =z^{\Delta'-\Delta_1-\Delta_2}(1+O(z))\).

%

\subsection{$\widehat{\mathfrak{sl}}(2)$ conformal blocks} \label{ssec:sl2ConformalBlocks}

The definition and properties of $\widehat{\mathfrak{sl}}(2)$ conformal blocks are similar to Virasoro ones discussed in the previous section. The main new ingredient is the dependence of the vertex operators on the additional parameter \(x\) which parametrizes Borel subalgebra.
\begin{define}
	Assume that there are given parameters \(\lambda_i\), \(i\in \{1,\dots,n\}\) and \(\lambda_j'\), \(j\in \{2,\dots,n-2\}\). Assume also that there are given vertex operators  $\mathcal{V}_{\lambda_2}(x,z)\colon \mathcal{M}_{\lambda_1,k}\rightarrow \overline{\mathcal{M}_{\lambda'_2,k}}$, $\mathcal{V}_{\lambda_i}(x,z)\colon\mathcal{M}_{\lambda'_{i-1},k}\rightarrow \overline{\mathcal{M}_{\lambda'_i,k}}$ for $i\in\{3,\dots,n-2\}$, and 	 $\mathcal{V}_{\lambda_{n-1}}(x,z)\colon \mathcal{M}_{\lambda'_{n-2},k}\rightarrow \overline{\mathcal{M}_{\lambda_n,k}}$  as in sec.~\ref{ssec:operstatecorr}. The  
	$n$-point sphere conformal block for $\widehat{\mathfrak{sl}}(2)$ algebra defined as
	\begin{equation}
		\Psi_k(\vec{\lambda},\vec{\lambda'};\vec{x},\vec{z}) = \Big\langle v_{\lambda_n,k},\mathcal{V}_{\lambda_{n-1}}(x_{n-1},z_{n-1})\dots\mathcal{V}_{\lambda_{3}}(x_3,z_{3})\mathcal{V}_{\lambda_{2}}(x_2,z_{2})v_{\lambda_1,k}\Big\rangle,
	\end{equation}
	where \(\vec{\lambda}=(\lambda_1,\dots,\lambda_n)\), \(\vec{\lambda'}=(\lambda_2',\dots,\lambda_{n-2}')\), \(\vec{x}=(x_2,\dots,x_{n-1})\), \(\vec{z}=(z_2,\dots,z_{n-1})\).
\end{define}

\begin{define}
	Let $ \mathcal{V}_{\lambda_2}(x, z)\colon \mathcal{M}_{\lambda_1,k}\rightarrow \mathcal{M}_{\lambda_3,k}$ be an \(\widehat{\mathfrak{sl}}(2)\) vertex operator as in sec.~\ref{ssec:operstatecorr}. Vector   $\mathcal{W}^{k;\lambda_2}_{\lambda_1,\lambda_3}(x, z) = \mathcal{V}_{\lambda_2}(x, z)v_{\lambda_1,k} \in \overline{\mathcal{M}_{\lambda_3,k}}$ is called \emph{$\widehat{\mathfrak{sl}}(2)$-chain vector}.
\end{define}

Consider four-point conformal blocks. Similarly to Virasoro case one can assume that \(z_2=z\), \(x_2=x\), \(z_3=1\), \(x_3=1\) and we have only one intermediate parameter \(\lambda_2'=\lambda'\). We have 
\begin{equation}\label{eq:sl2 conf block Whit}
	\Psi_k(\vec{\lambda},\lambda';x,z)=\Big\langle v_{\lambda_4,k},\mathcal{V}_{\lambda_{3}}(1,1)\mathcal{V}_{\lambda_{2}}(x,z)v_{\lambda_1,k}\Big\rangle=\Big\langle \mathcal{W}^{k;\lambda_3}_{\lambda_4,\lambda'}(-1, 1),\mathcal{W}^{k;\lambda_2}_{\lambda_1,\lambda'}(x, z)\Big\rangle. 
\end{equation}
Unless otherwise stated, we assume that Whittaker vectors are normalized by \(\big\langle v_{\lambda_3,k}, \mathcal{W}^{k;\lambda_2}_{\lambda_1,\lambda_3}(x,z) \big \rangle= z^{\Delta_3-\Delta_1-\Delta_2} x^{(\lambda_1+\lambda_2-\lambda_3)/2}\). Hence conformal blocks have the form 
\begin{equation}
	\Psi_k(\vec{\lambda},\lambda';x,z) =z^{\Delta'-\Delta_1-\Delta_2} x^{(\lambda_1+\lambda_2-\lambda')/2}\Big(1+O(x)+O(z)\Big).
\end{equation}

We will need a description of four-point conformal block with insertion of degenerate field of spin $\frac{1}{2}$ which we denote by \(X_1(x,z)\), see \eqref{eq:X nu}. This description follows from Knizhnik-Zamolodchikov equations \cite{KZ:1984} and is well known, see e.g. \cite{Teschner:1997}.
Consider function 
\begin{equation}
	\Psi_k(\vec{\lambda},\lambda';y,x,z)=\Psi_{k,0}(\vec{\lambda},\lambda';y,z)+x\Psi_{k,1}(\vec{\lambda},\lambda';y,z) =\langle v_{\mu,k},\mathcal{V}_{\nu}(y, 1)X_{1}(x,z)\,v_{\lambda,k}\rangle.
\end{equation}
Here \(\vec{\lambda}=(\lambda,\nu, \mu)\). It follows from fusion rules that there are two possible choices of \(\lambda'\) namely \(\lambda'=\lambda\pm 1\).




We can now write Knizhnik-Zamolodchikov equations.

\begin{proposition}
	Functions $\Psi_{k,i}(\vec{\lambda},\lambda';y, z)$ satisfy system of equations
	\begin{subequations}\label{eq:KZ}
		\begin{align}
			\label{kz1}
			\kappa\partial_z   \Psi_{k,0}&= \frac{(- y^2 \partial_y + \nu y) \Psi_{k,1} + \frac{1}{2}(-2y\partial_y + \nu) \Psi_{k,0}}{z - 1} 
			+ \frac{\lambda \Psi_{k,0}}{2 z},
			\\
			\label{kz2}
			\kappa\partial_z   \Psi_{k,1} &= \frac{\partial_y \Psi_{k,0} - \frac{1}{2}(-2y\partial_y + \nu) \Psi_{k,1}}{z - 1} 
			+\frac{\frac12(-\lambda-2)\Psi_{k,1} - \partial_y \Psi_{k,0}}{z}.
		\end{align}		
	\end{subequations}
	\end{proposition}

\begin{corollary}\phantomsection \label{Cor:sl2 bl} 
	The conformal blocks have the following expressions in terms of hypergeometric functions
	\begin{subequations}	
		\begin{align}
			\Psi_k(\vec{\lambda},\lambda+1;y,x,z)& = y^\rho z^{\frac{\lambda }{2 \kappa }} (1-z)^{\frac{\nu}{2 \kappa }}
			\left({}_2F_1\big(\Vec{a}|z\big) -  \frac{a_1}{a_3} \frac{x}{y}\, {}_2F_1\big(\Vec{a}+(1,0,1)|z\big) \right),
			\\
			\Psi_k(\vec{\lambda},\lambda-1;y,x,z)& =y^\rho z^{-\frac{\lambda +2}{2 \kappa }} (1-z)^{\frac{\nu}{2 \kappa }}
			\left( z \frac{a_2-a_3}{a_3-1} \, {}_2F_1\big(\hat{I}(\Vec{a})|z\big)+ \frac{x}{y}\, {}_2F_1\big(\hat{I}(\Vec{a}+(1,0,1))|z\big)  \right).
		\end{align}
	\end{subequations}
	where 
	\begin{equation}
		\vec{a} = (a_1,a_2,a_3)=(\frac{\lambda - \mu + \nu+1}{2\kappa}, \frac{ \lambda + \mu + \nu+3}{2\kappa}, \frac{1 + \lambda}{\kappa}).
	\end{equation}
\end{corollary}

\begin{proof}[Sketch of the proof]
	Inserting \(h_0\) into matrix element $\langle v_{\mu,k}|\mathcal{V}_{\nu}(y, 1)X_{i,1}(z)h_0|v_{\lambda,k}\rangle$ we can fix the dependence of $\Psi_{k,i}(y, z)$ on $y$.
	\begin{equation}
		\Psi_{k,0}(\vec{\lambda},\lambda+1;y, z) = y^\rho f_0(z),\qquad \Psi_{k,1}(\vec{\lambda},\lambda+1;y, z) =  y^{\rho - 1} f_1(z),
	\end{equation}
	where $\rho = \frac{1}{2} (\lambda -\mu +\nu +1)$. 
	
	Then KZ equations~\eqref{eq:KZ} in terms of $f_0(z), f_1(z)$ have the form
	\begin{subequations}
		\begin{align}
		\kappa  f_0'(z) &= f_0(z) \left(\frac{\lambda }{2 z}+\frac{\nu -2 \rho }{2 (z-1)}\right)+ f_1(z) \frac{\nu -\rho + 1}{z-1},
		\\
		\kappa  f_1'(z)  &= f_0(z) \left(\frac{\rho }{z-1}-\frac{\rho }{z}\right)+f_1(z) \left(\frac{-\lambda -2}{2z}-\frac{\nu -2 (\rho -1)}{2 (z-1)}\right).
		\end{align}
	\end{subequations}
	After simple calculation, we get the result.
\end{proof}

Note that we normalized conformal blocks above such that the leading coefficient in $x, z/x$ expansion is equal to $1$. 


\subsection{Coset decomposition and conformal block relations} \label{ssec:coset blocks}

Now, we are ready to come back to coset construction. The vector \(\widetilde{\mathcal{U}}_{\nu,n}(x, z) \tilde u_{l}(\lambda) \)  would be a tensor product of Virasoro and $\widehat{\mathfrak{sl}}(2)$-chain vectors since the operator \(\widetilde{\mathcal{U}}_{\nu,n}\) is a primary field for \(\mathrm{Vir}^{\mathrm{coset}} \oplus \widehat{\mathfrak{sl}}(2)_{k + 1}\). A chain vector in each summand of  \eqref{eq:W1 W2} would be a tensor product of Virasoro and $\widehat{\mathfrak{sl}}(2)$-chain vectors. We fix normalization of such products by 
\begin{multline}\label{eq:chain coset}
         \Big\langle  \mathbb{W}^{b;P(\nu) + n b}_{ P(\lambda) + l b, P(\mu)+ m b}(z)\otimes \mathcal{W}^{k+1,\nu + 2n}_{\lambda + 2l,\mu + 2m }(x, z), \tilde u_{m}(\mu)\Big\rangle 
         \\
         = \|  \tilde u_{m}(\mu)\|^2 z^{\Delta(P(\mu))-\Delta(P(\nu))-\Delta(P(\lambda))}x^{\frac{-\mu+ \nu+\lambda}{2}-m+n+l}.
\end{multline}
Then, the following proposition is just a reformulation of the definitions of $\tilde C_{m,n,l}(\mu,\nu,\lambda)$ and $\tilde B^{s}_{l}(\lambda)$  given in formulas \eqref{eq:tilde C} and \eqref{eq:b tildeU} above.
\begin{proposition}\phantomsection \label{Prop:ChainDecomp}
	We have
\begin{align}\label{eq:Chain coset}
     \widetilde{\mathcal{U}}_{\nu,n}(x, z) \tilde u_{l}(\lambda)  &=\sum_{n\in \mathbb{Z}} \frac{\tilde C_{m,n,l}(\mu,\nu,\lambda)}{\|\tilde u_m(\mu)\|^2}\, \mathbb{W}^{b;  P(\nu) + n b }_{P(\lambda) + l b, P(\mu) + m b}(z)\otimes \mathcal{W}^{k+1,\nu + 2 n}_{\lambda + 2l, \mu + 2m}(x, z), 
     \\
     b(x, z) \tilde u_{l}(\lambda) &= \sum_{s\in\pm 1} \frac{\tilde B^{s}_{l}(\lambda)}{\|\tilde u_{l+{s}/{2}}(\lambda)\|^2}\, \mathbb{W}^{b,P_{1,2}(b)}_{P(\lambda)+lb,P(\lambda)+l+\frac{s}{2}}(z)\otimes \mathcal{W}^{k+1; 1}_{\lambda + 2 l,\lambda + 2l+s}(x, z),
\end{align}
   where $P(\lambda) = P_{GKO}(\lambda),b = b_{GKO}$.
\end{proposition}

\subsubsection{Recurrence relations}\label{sssec:RecurrenceRelations}

Consider the four-point conformal block of the form 
\begin{equation}
	\Big\langle \tilde u_{m}(\mu), \widetilde{\mathcal{U}}_{\nu,n}(1,1)\, b(x, z)\, \tilde u_{l}(\lambda)\Big\rangle.
\end{equation}
Due to Proposition \ref{Prop:ChainDecomp}, it decomposes as a linear combination of Virasoro and \(\widehat{\mathfrak{sl}}(2)\) conformal blocks. According to 
Proposition~\ref{Th:bprop}, these are conformal blocks with the presence of degenerate fields. Therefore we get 

\begin{corollary}The four-point conformal block with presence of fields \(b(x,z)\) has the form 
	\begin{multline}\label{eq:b bl}
		\langle \tilde u_{m}(\mu), \widetilde{\mathcal{U}}_{\nu,n}(1,1) b(x, z) \, \tilde u_{l}(\lambda)\rangle      = \sum_{s\in \pm 1} \frac{\tilde C_{m,n,l+s/2}(\mu,\nu,\lambda)  \tilde B^{s}_{l}(\lambda)}{\|\tilde u_{l+s/2}(\lambda)\|^2} 
		\\
		\times
		\mathrm{F}_b\left(\vec{P}_{\vec{l}},P(\lambda) + lb + sb/2;z\right)
		\Psi_{k+1}\left(\vec{\lambda}_{\vec{l}},\lambda + 2l + s;x,z \right), 
	\end{multline}
	where 
	\begin{equation}
		\vec{P}_{\vec{l}}=(P(\lambda)+lb,P(\nu)+nb,P(\mu)+mb),\quad \vec{\lambda}_{\vec{l}}=(\lambda+2l,\nu+2n,\mu+2m).
	\end{equation}
\end{corollary}

The Virasoro and \(\widehat{\mathfrak{sl}}(2)\) conformal blocks appeared in the formula \eqref{eq:b bl} are given in Corollaries \ref{Cor:Vir bl} and \ref{Cor:sl2 bl} correspondingly (with \(y=1\) in the latter). Therefore we get  
\begin{subequations}
	\begin{multline}
		\mathrm{F}_b\left(\vec{P}_{\vec{l}},P(\lambda) + lb + b/2;z\right)
		\Psi_{k+1}\left(\vec{\lambda}_{\vec{l}},\lambda + 2l + 1;x,z \right)
		=(1 - z)^{\frac{1}{2}(-r_1-r_2+r_3)} z^{\frac{-1 - r_3}{2}} \\ 
		\Big( {}_2F_1(\Vec{a}|z) \,{}_2F_1(-\Vec{a}-\Vec{r}|z) 
		- x\frac{a_1}{a_3} {}_2F_1(\Vec{a}+(1,0,1)|z)\,{}_2F_1(-\Vec{a}-\Vec{r}|z)\Big)
	\end{multline}
	\begin{multline}
		\mathrm{F}_b\left(\vec{P}_{\vec{l}},P(\lambda) + lb - b/2 ;z\right)
		\Psi_{k+1}\left(\vec{\lambda}_{\vec{l}},\lambda + 2l -1;x,z \right)
		= (1 - z)^{\frac{1}{2}(-r_1-r_2+r_3)} z^{\frac{1 + r_3}{2}} \\
		\Big(  \frac{a_2 - a_3}{a_3 - 1} z\, {}_2F_1(\hat{I}\Vec{a}|z)\, {}_2F_1(\hat{I}(-\Vec{a}-\Vec{r})|z) 
		+ x \, {}_2F_1(\hat{I}(\Vec{a}+(1,0,1))|z)\,{}_2F_1(\hat{I}(-\Vec{a}-\Vec{r})|z)\Big),
	\end{multline}
\end{subequations}
where
\begin{subequations} 
	\begin{align}
		a_1&= \frac{(\lambda + 2 l-\mu -2 m+\nu +2 n+1)}{2(\kappa + 1)},& ~~ r_1&= -l+m-n-\frac{1}{2},
		\\
		a_2&= \frac{ (\lambda + 2 l+\mu +2 m+\nu +2 n+3)}{2(\kappa + 1)},& ~~  r_2&= -l-m-n-\frac{1}{2},
		\\
		a_3&=  \frac{(\lambda +2 l+1)}{\kappa + 1},& ~~  r_3&=-2 l-1.
	\end{align}
\end{subequations}

On the other hand, it follows from the definition of \(b(x,z)\) in formula \eqref{eq:b01} that the action \(b(x,z)\) on \(\tilde u_{l}(\lambda)\) is given by Laurent series in \(z\) for \(l\in \mathbb{Z}\) and by Laurent series in \(z\) times \(z^{1/2}\) for \(l \in \mathbb{Z}+1/2\). Similarly singular part of OPE of \(b(x,z)\) and
$\widetilde{\mathcal{U}}_{\nu,n}(1,1)$ is either Laurent polynomial in \(z-1\) if \(n \in \mathbb{Z}\) or Laurent polynomial in \(z-1\) times \((z-1)^{1/2}\) for \(n \in \mathbb{Z}+1/2\). In any case, after factoring out the term \((1 - z)^{\frac{1}{2}(-r_1-r_2+r_3)} z^{\frac{-1 - r_3}{2}}\) we get a linear combination of products of \({}_2F_1\) hypergeometric functions which has trivial monodromy in \(z\) and can have only poles at \(z=0,1,\infty\). Hence it is a rational function in \(z\). This gives strong restriction on the coefficients in this linear combination. Namely
\begin{equation}
	\frac{\|\tilde u_{l+\frac{1}{2}}(\lambda)\|^2}{\|\tilde u_{l-\frac{1}{2}}(\lambda)\|^2}\frac{\tilde B_{l}^-(\lambda)}{\tilde B_{l}^+(\lambda)}\frac{\tilde C_{m,n,l-1/2}(\mu,\nu,\lambda)}{\tilde C_{m,n,l+1/2}(\mu,\nu,\lambda)}  \frac{a_2 - a_3}{a_3 - 1}=p_{\vec{r}}(\vec{a})
\end{equation}
where \(p_{\vec{r}}(\vec{a})\) is a coefficient found in Proposition~\ref{Prop:MonCanc}. Using Corollary~\ref{Cor:tilde U} and Proposition~\ref{Prop:b(x,z)} after straightforward computation we get a recursion formula.

\begin{proposition}\phantomsection \label{Prop:lrecurrence}
	We have
	\begin{multline}
		\frac{\tilde{C}_{m,n,l-\frac{1}{2}}(\mu,\nu,\lambda)}{\tilde{C}_{m,n,l+\frac{1}{2}}(\mu,\nu,\lambda)} 
		\\
		= (-1)^{-l+m-n-\frac{1}{2}}\frac{\mathtt{s}^{1,-\frac{1}{\kappa}}_{- l - m - n+\frac{1}{2}}(-\frac{2 + \lambda + \mu+\nu}{2\kappa})\mathtt{s}^{1,-\frac{1}{\kappa}}_{-l + m - n +\frac{1}{2}}(-\frac{\lambda - \mu+\nu}{2\kappa})\mathtt{s}^{1,-\frac{1}{\kappa}}_{- l - m + n +\frac{1}{2}}(-\frac{\lambda + \mu - \nu}{2\kappa})}{\mathtt{s}^{1,-\frac{1}{\kappa}}_{l-m-n+\frac{1}{2}}(-\frac{ - \lambda + \mu+\nu}{2\kappa}) \mathtt{s}^{1,-\frac{1}{\kappa}}_{-2 l+1}(-\frac{\lambda+1}{\kappa})\mathtt{s}^{1,-\frac{1}{\kappa}}_{-2 l}(-\frac{\lambda+1}{\kappa})}. 
	\end{multline}
\end{proposition}
This finishes the proof of Theorem \ref{Th:main}. Recall that the function \(\mathtt{s}\) stands for the product over the segment and was defined in formula~\eqref{GammaS}.

\subsubsection{Conformal block relations and blowup equations}\label{sssec:BlowupRelations}
Consider now the conformal block of the form 
\begin{equation}
	\Big\langle \tilde u_{0}(\mu_4), \widetilde{\mathcal{U}}_{\mu_3,0}(1,1)\, \widetilde{\mathcal{U}}_{\mu_2,0}(x,z) \,\tilde u_{0}(\mu_1)\Big\rangle.
\end{equation} 
Note that the \(\widehat{\mathfrak{sl}}(2)_1\) part of the vectors \(u_{0}(\mu_4), u_{0}(\mu_1)\) and vertex operators \(\widetilde{\mathcal{U}}_{\mu_3,0}, \widetilde{\mathcal{U}}_{\mu_2,0}\) are trivial. Hence this conformal block is equal to  \(\Psi_k\left(\vec{\mu},\lambda;x, z\right)\). On the other hand, we can use decomposition \eqref{eq:W1 W2} and write
\begin{equation}\label{eq:blowup}
	\Psi_k\left(\vec{\mu},\lambda;x, z\right)
	=
	\sum_{l \in \mathbb{Z}}\frac{\tilde C_{0,0, l}(\mu_4,\mu_3,\lambda)\tilde C_{l,0,0}(\lambda,\mu_2,\mu_1)}{\|\tilde u_{l}(\lambda)\|^2}
	\Psi_{k+1}\left(\vec{\mu},\lambda+2l;x, z\right)\mathrm{F}_b\left(\vec{P},P(\lambda)+lb;z \right), 	
\end{equation}
where 
\begin{equation}
	\vec{P}=(P(\mu_1),P(\mu_2),P(\mu_3),P(\mu_4)),\quad \vec{\mu}=(\mu_1,\mu_2,\mu_3,\mu_4).
\end{equation}
Due to AGT correspondence the function \(\mathrm{F}_b\) is equal (up to simple \(U(1)\) factor) to Nekrasov partition function for \(SU(2)\) theory with \(N_f=4\) \cite{Alday:2010Liouville} and the function \(\Psi_k\) is equal to Nekrasov partition function for \(SU(2)\) theory with presence of a surface defect \cite{Alday:2010Affine}, \cite{Nekrasov:2017bps}, \cite{Nekrasov:2022surface}. In this geometric language the relation \eqref{eq:blowup} is a blowup relation with the presence of the surface defect, suggested in \cite{Nekrasov:2023blowups}, \cite{Jeong:2020} (equations without defect were proven by Nakajima and Yoshioka in the seminal paper~\cite{Nakajima:2005instanton}).

\begin{remark}\label{Rem:norm}
	Note that in relation \eqref{eq:blowup} the coefficients are given by rational functions, essentially the products of triangle functions \(\mathtt{t}\). It appears there is another normalization of conformal blocks in which these coefficients are equal to one. On the gauge theory side, the corresponding Nekrasov functions are called \emph{full partition functions}. This is actually normalization used in \cite{Nekrasov:2023blowups}, \cite{Jeong:2020}, for the blowup relation without surface defect see \cite[sec. 4.4]{Nakajima:2003lectures}. We briefly recall some choice of such normalization in Appendix \ref{sec:Barnes}.	
\end{remark}

As another example let us consider one-point torus conformal blocks defined as 
\begin{equation}
	\Psi_k^{\text{tor}}(\lambda,\lambda';x,q)=\operatorname{Tr} \left(q^{L_0}x^{h_0}\mathcal{V}_\lambda(1,1)\right) |_{\mathcal{M}_{\lambda',k}},\quad \mathrm{F}_b^{\text{tor}}(P,P';q)=\operatorname{Tr} \left(q^{L_0}\Phi_\Delta(z)\right) |_{\mathbb{M}_{P',b}}.
\end{equation}
Then we have 
\begin{multline}\label{eq:blowup tor}
	\left(\frac{\sum_{n \in \mathbb{Z}} q^{n^2} x^{2n}}{\prod_{m = 1}^{\infty} (1 - q^m)}\right)\Psi_k^{\text{tor}}(\nu,\lambda;x,q)=\operatorname{Tr} \left(q^{L_0^{(1)}+L_0^{(2)}}x^{h_0^{(1)}+h_0^{(2)}} \widetilde{\mathcal{U}}_{\nu,0}(x,z)\right) |_{\mathcal{L}_{0, 1} \otimes  \mathcal{M}_{\lambda, k}}
	\\
	=\sum_{l \in \mathbb{Z}}  \frac{\tilde C_{l,0, l}(\lambda,\nu,\lambda)}{\|\tilde u_{l}(\lambda)\|^2} \Psi_{k+1}^{\text{tor}}(\nu,\lambda+2l;x,q) \mathrm{F}_b^{\text{tor}}(P(\nu),P(\lambda)+lb;q),
\end{multline}
see character formulas \eqref{eq:char level 1}. Same as above, due to AGT correspondence this can be viewed as a blowup relation for $SU(2)$ theory with adjoint matter with the presence of the surface defect. 

Clearly one can generalize such relations for more point conformal blocks on sphere or torus. In the next section, we will also consider Whittaker limit of these relations.

\section{Kyiv formula for Painlevé $\mathrm{III}_3$ tau-function}  \label{sec:Kyiv}

In this section we deduce Kyiv formulas for the tau function from the coset (or blowup) relations \eqref{eq:blowup} closely following \cite{Nekrasov:2023blowups}, \cite{Jeong:2020}. We restrict ourselves to the case of Painlev\'e \(\mathrm{III}_3\) which corresponds to Whittaker limit of conformal blocks. First, we recall Hamiltonian of the Painlev\'e \(\mathrm{III}_3\) and show the relation between tau function and generating function of canonical transformation. Then we define 
Whittaker vectors and Whittaker limits of conformal blocks. The Whittaker limit of $\widehat{\mathfrak{sl}}(2)$ conformal block satisfies the non-stationary affine Toda equation. In the classical \(k \rightarrow \infty\) this leads to solution of Hamilton-Jacobi equation for Painlevé $\mathrm{III}_3$ Hamiltonian. Taking the classical limit of relations \eqref{eq:blowup} we get the Kyiv formula.

\subsection{Generating function and tau function for Painlev\'e \(\mathrm{III}_3\) }
Recall some facts about Hamiltonian mechanics (see e.g. \cite{Arnold2013} for the reference).

Consider extended phase space with coordinates $(x,p,z)$ where $x$ is a coordinate, $p$ is a momentum, and $z$ is a time. The degenerate Poisson bracket is defined by $\{x, p\} = 1,\{x, z\} = \{p, z\} =0$. Hamiltonian dynamics is defined by one function which is called Hamiltonian $H(x,p;z)$ and differential equations called Hamilton equations
\begin{equation}\label{eq:hamiltoneq}
     \frac{d p}{d z}= -\left(\frac{\partial H}{\partial x}\right)_{p,z}; ~~~  \frac{d x}{d z}= \left(\frac{\partial H}{\partial p}\right)_{x,z}.
\end{equation}
Here and below by additional indices in partial derivatives, we emphasize the variables that are considered to be fixed. 

Assume that there is another pair of functions \(\alpha, \beta\) on the extended phase space and a function \(S(x,\alpha;z)\) such that 
\begin{equation}\label{eq:p beta}
	p= \left(\frac{\partial S}{\partial x}\right)_{\alpha,z}, \quad \beta= -\left(\frac{\partial S}{\partial \alpha}\right)_{x,z},
\end{equation}
and Hamilton-Jacobi equation holds
\begin{equation}\label{eq:Hamilton Jacobi}
	\left(\frac{\partial S}{\partial z}\right)_{\alpha,x} + H\left(x,  \left(\frac{\partial S}{\partial x}\right)_{\alpha,z}; z\right) = 0.
\end{equation}  
The function \(S(x,\alpha;z)\) is a called a generating function of a canonical transformation from coordinates \((x,p,z)\) to \((\alpha,\beta,z)\). The following result is standard
\begin{theorem}\label{Th:HJ}
	In the assumption above, the Hamiltonian flow equations are equivalent to the condition that \(\alpha,\beta\)  are integrals of motions, i.e. \(\partial_z \alpha=\partial_z \beta=0\).
\end{theorem}
\begin{proof}
	Equations \eqref{eq:p beta} and \eqref{eq:Hamilton Jacobi} imply that   
	\begin{equation}
		\eta= p dx-H dz=\beta d\alpha+dS.
	\end{equation}
	The Hamiltonian equations \eqref{eq:hamiltoneq} correspond to vector field in the kernel of 2-form \(d\eta\) in \((x,p,z)\). In the \((\alpha,\beta,z)\) coordinates the kernel of \(d\eta\) is given by the vector field \(\partial_z\).
\end{proof}

The Hamiltonian of the Painlev\'e $\mathrm{III}_3$ equation has the form (see e.g. \cite{GIL13})
\begin{equation}\label{eq:Painleve}
    H = \frac{1}{z}(x^2 p^2 - x^{-1}z - x).
\end{equation}
\begin{define}
    The function $\tau(x,p;z)$ such that $\dfrac{d \log \tau}{d z} = H(x,p;z)$ is called Painlev\'e $\mathrm{III}_3$ tau-function.
\end{define}

It appears that there is a simple relation between generating function \(S\) satisfying the Hamilton-Jacobi equation and  Painlev\'e $\mathrm{III}_3$ $\tau$-function. Similar relation for the Painlev\'e $\mathrm{VI}$ case given in \cite[sec. 4.3]{Nekrasov:2023blowups}, see also \cite[sec. 7]{Its:2016connection}.
\begin{proposition}\phantomsection \label{Th:tauPainleve}
Suppose function $S(x,\alpha;z)$ satisfies Hamilton-Jacobi equation \eqref{eq:hamiltoneq} for Painlev\'e~$\mathrm{III}_3$ Hamiltonian, functions \(\beta,p\) are defined by relations~\eqref{eq:p beta} and functions $\alpha, \beta$ are integrals of motion. Then 
    \begin{equation}
        \log\tau  =  S - \alpha \beta - 4 z \left(\frac{\partial S}{\partial z}\right)_{\alpha,x} - 2 x \left(\frac{\partial S}{\partial x}\right)_{\alpha,z}.
    \end{equation}
\end{proposition}
\begin{proof}
Let us denote by $F = S - \alpha \beta - 4 z \left(\frac{\partial S}{\partial z}\right)_{\alpha,x} - 2 x \left(\frac{\partial S}{\partial x}\right)_{\alpha,z} = S - \alpha \beta  + 4 z H - 2 x p$. Then
\begin{multline}
    \left(\frac{\partial F}{\partial z}\right)_{\alpha,\beta} = \left(\frac{\partial S}{\partial z}\right)_{\alpha,\beta} + 4 H + 4 z \left(\frac{\partial H}{\partial z}\right)_{\alpha,\beta} - 2 \left(\frac{\partial(x p)}{\partial z}\right)_{\alpha,\beta} = 
    \\
    =3H + 4z \left(\frac{\partial H}{\partial z}\right)_{p,x} + 2 x \left(\frac{\partial H}{\partial x}\right)_{p,z} -  p \left(\frac{\partial H}{\partial p}\right)_{x,z} = H,
\end{multline}
where we used
\begin{equation}
     \left(\frac{\partial S}{\partial z}\right)_{\alpha,\beta}  =  \left(\frac{\partial S}{\partial z}\right)_{\alpha,x} +  \left(\frac{\partial S}{\partial x}\right)_{\alpha,z}  \left(\frac{\partial x}{\partial z}\right)_{\alpha,\beta} =  - H  + p  \left(\frac{\partial H}{\partial p}\right)_{x,z}.
\end{equation}
and formula (\ref{eq:Painleve}) for Hamiltonian.
\end{proof}

\subsection{Whittaker limit of three-point functions}
\begin{define}
    The $\widehat{\mathfrak{sl}}(2)$ Whittaker vector $\mathcal{W}_\lambda^k(x,z)\in \overline{\mathcal{M}_{\lambda,k}}$ is a vector defined by relations 
    \begin{equation}
        e_0 \mathcal{W}_\lambda^k(x,z) = x \mathcal{W}_\lambda^k(x,z),\quad f_1 \mathcal{W}_\lambda^k(x,z) = z x^{-1} \mathcal{W}_\lambda^k(x,z).
    \end{equation}
\end{define}
\begin{define}
    The Virasoro Whittaker vector $\mathbb{W}_P^b(x,z)\in \overline{\mathbb{M}_{P,b}}$ is a vector defined by relations 
    \begin{equation}
        L_1 \mathbb{W}_P^b(z) = z \mathbb{W}_P^b(z),\quad L_2 \mathbb{W}_P^b(z) = 0.
    \end{equation}
\end{define}
\begin{proposition}\phantomsection \label{Th:whtexist}
	(a) If $\lambda, k$ are generic then there exists a unique up to constant $\widehat{\mathfrak{sl}}(2)$ Whittaker vector, moreover  $\big\langle\mathcal{W}_\lambda^k(x,z),v_{\lambda,k}\big\rangle \neq 0$.
	
	(b) If $P,b$ are generic then there exists a unique up to constant Virasoro Whittaker vector, moreover $\big\langle\mathbb{W}_P^b(z),v_{P,b}\big\rangle \neq 0$.
\end{proposition}

Proposition \ref{Th:whtexist}(a) follows from the fact that for generic values of parameters Verma module \(\mathcal{M}_{\lambda,k}\) is irreducible. Hence the  Shapovalov on it is non-degenerate and any eigenvector of the nilpotent subalgebra  \(\widehat{\mathfrak{n}}\) can be found uniquely up to normalization in the completion $\overline{\mathcal{M}_{\lambda,k}}$. The proof of  Proposition \ref{Th:whtexist}(b) is analogous.

On the hand we can consider limit of chain vector and get the Whittaker vectors
\begin{subequations}
	\begin{align}
		&\lim_{\lambda_2\rightarrow \infty} 	z^{\Delta_1+\Delta_2}x^{(-\lambda_1-\lambda_2)/2}\mathcal{W}^{k;\lambda_2}_{\lambda_1,\lambda_3}(\lambda_2^{-1}x, \lambda_2^{-2} z)=\mathcal{W}_{\lambda_{3}}^k(x,z),			
		\\  
		&\lim_{P_2\rightarrow \infty} z^{\Delta_1+\Delta_2} \mathbb{W}^{b;P_2}_{P_1,P_3}(P_2^{-2}z)=\mathcal{W}_{P_{3}}^b(z).
	\end{align}	
\end{subequations}

Using this limit description we can first define the Whittaker limit of conformal blocks (taking the limit of formulas \eqref{eq:sl2 conf block Whit}, \eqref{eq:Vir conf block Whit})
\begin{define}
	The Whittaker limit of \(\widehat{\mathfrak{sl}}(2)\) conformal block is defined as 
	\begin{equation}
		\Psi_k(\lambda; x, z) = \Big\langle \mathcal{W}_\lambda^k(1,1), \mathcal{W}_\lambda^k(x,z)\Big\rangle.
	\end{equation} 
\end{define}
\begin{define}
	Consider the Whittaker vector $\mathbb{W}(z)\in \mathbb{M}_{P,b}$ then we define Virasoro Whittaker conformal block by formula
	\begin{equation}
		\mathrm{F}_b(P; z) = \Big\langle \mathbb{W}_P^b(1),\mathbb{W}_P^b(z) \Big\rangle.
	\end{equation} 
\end{define}

For the coset decomposition, we can take a limit of decomposition \eqref{eq:Chain coset} and get 
\begin{equation}\label{eq:whit decomp}
	v_0\otimes \mathcal{W}(x,z) =  \sum_{l\in \mathbb{Z}} \frac{1}{ \mathtt{t}^{1,-{1}/{\kappa}}_{-2 l}( \frac{\lambda+1}{-\kappa})\|\tilde u_l(\lambda)\|^2}\mathbb{W}_{P(\lambda)+lb}^b(z)\otimes \mathcal{W}_{\lambda+2l}^k(x,z).
\end{equation}   
Here the tensor product of Whittaker vectors \(\mathbb{W}_{P(\lambda)+lb}^b(z)\otimes \mathcal{W}_{\lambda+2l}^k(x,z) \in \overline{\mathbb{M}_{P+lb,b}\otimes \mathcal{M}_{\lambda + 2l, k + 1}}\) is normalized as 
\begin{equation}
	\langle \mathbb{W}_l(z)\otimes \mathcal{W}_l(x,z), \tilde u_l(\lambda)\rangle = \|\tilde u_l(\lambda)\|^2 z^{\Delta(P(\lambda),b)+l^2}x^{-l-\lambda/2},
\end{equation}
c.f. normalization of chain vectors above \eqref{eq:chain coset}. Taking the scalar product of decomposition~\eqref{eq:whit decomp} (or taking the limit of conformal relation \eqref{eq:blowup}) we get

\begin{theorem}\phantomsection \label{Th:blowuplim}
    There is a relation on Whittaker limit of conformal blocks
    \begin{equation}\label{eq:blowup Whittaker}
    	\Psi_k(\lambda; x, z) = \sum_{l \in \mathbb{Z}} \frac{1}{\mathtt{t}^{1,-\frac{1}{\kappa}}_{-2 l}( \frac{\lambda}{-\kappa})\mathtt{t}^{1,-\frac{1}{\kappa}}_{-2 l}( \frac{\lambda+1}{-\kappa})} \mathrm{F}_b(P(\lambda) + lb; z) \Psi_{k+1}(\lambda+ 2l; x, z).
    \end{equation} 
\end{theorem}

\subsection{Kyiv formula}
\begin{proposition}
	The function $\Psi_k(\lambda;x,z)$ satisfies Toda differential equation
	\begin{equation}\label{eq:Toda}
    	z \partial_z \Psi_k(\lambda;x,z) = \frac{1}{\kappa}\Big(   z x^{-1}  +  x +   x^2 \partial_x^2 \Big)\Psi_k(\lambda;x,z).
	\end{equation}    
\end{proposition}
Of course, this is just a Knizhnik-Zamolodchikov equation in the Whittaker limit.
\begin{proof}
	Recall that by Sugawara construction \eqref{Sug} we have
	\begin{equation}
		L_0 = \frac{1}{2\kappa}(e_0 f_0 + f_0 e_0 + \frac{1}{2} h_0^2 + 2 e_{-1} f_1 +  2 f_{-1} e_1 + h_{-1} h_1) + \dots.
	\end{equation}
	Inserting \(L_0\) into the formula for the conformal block we get
	\begin{multline}
		z \partial_z \Psi_k(\lambda;x,z) = \langle \mathcal{W}_\lambda^k(1,1), L_0\mathcal{W}_\lambda^k(x,z)\rangle
		\\
		= \frac{1}{2\kappa}\left(2 x + 2 z x^{-1} + \frac{- 2 x \partial_x(- 2 x \partial_x+ 2)}{2}\right)\Psi_k(\lambda;x,z).    
	\end{multline}
	The proposition is proven.
\end{proof}

Let us consider the limit $\kappa\rightarrow\infty$. The conformal block has asymptotic behavior of the form 
\begin{equation}\label{eq:Psi classical}
	\log\Psi_k(-2\sigma \kappa; - x\kappa^2, z\kappa^4) =  - \kappa S(\sigma;  x, z ) + S_0 (\sigma; x,z) + O(1/\kappa),
\end{equation}
where  $S(\sigma;x,z), S_0(\sigma;x,z)$ are certain functions that do not depend on $\kappa$. The following proposition follows directly from 
the Toda equation \eqref{eq:Toda}.
\begin{proposition}
	The function $S(\sigma;  x, z )$ satisfies equation
	\begin{equation}\label{eq:HJPainleve}
		- z\partial_z S =  x^2 (\partial_x S)^2 - z x^{-1} - x.
	\end{equation}
\end{proposition}
Note that the equation \eqref{eq:HJPainleve} coincides with Hamilton-Jacobi equation \eqref{eq:Hamilton Jacobi} for  Painlev\'e~$\mathrm{III}_3$ Hamiltonian \eqref{eq:Painleve}. 

Now we take the classical limit of the coset decomposition. Note that the algebras $\widehat{\mathfrak{sl}}(2)_{k}$ and $\widehat{\mathfrak{sl}}(2)_{k + 1}$ are both become classical. On the other hand, the coset Virasoro algebra in this limit has $b = -\ri$ and central charge \(c=1\).

\begin{theorem}\phantomsection \label{Th:KyivFormula}
	 Tau function for the Painlev\'e $\mathrm{III}_3$ equation has the following expansion
	 \begin{equation}\label{eq:Kyiv}
         \tau(\sigma,\beta;z)= \sum_{l \in \mathbb{Z}} \frac{1}{\mathtt{t}^{1,0}_{2l}(-2\sigma)^{2}} \re^{  l\beta }\mathrm{F}_1(\ri(\sigma  +  l) ; z).
     \end{equation}
\end{theorem}

The formula \eqref{eq:Kyiv} is called the Kyiv formula, it was first conjectured in \cite{Gamayun:2012Conformal}, \cite{GIL13}. There are several other proofs this formula, namely \cite{Iorgov:2015Isomonodromic}, \cite{Bershtein:2015Bilinear}, \cite{Gavrylenko:2018Fredholm}, \cite{Gavrylenko:2018Pure}; as was mentioned above here we follow the logic of \cite{Nekrasov:2023blowups}, \cite{Jeong:2020}.

Note that usually coefficients in the Kyiv formula are written in terms of Barnes \(G\)-function. However, after a certain redefinition of variables, these coefficients can be made rational, see e.g. \cite[eq. (3.9)]{Bershtein:2017Backlund}. In the formula \eqref{eq:Kyiv} we also used rational coefficients  \(\mathtt{t}^{1,0}_{2l}(-2\sigma)^{2}\). It appears they are equivalent to the ones used in loc. cit. via additional change of variables. If we started from the conformal block in full normalization (see Remark \ref{Rem:norm} and Appendix \ref{sec:Barnes}) we would get Barnes functions without extra changes of variables. 

\begin{proof} 
	We take \(k\rightarrow \infty\) limit of the relation~\eqref{eq:blowup Whittaker}. For the left side we use expansion~\eqref{eq:Psi classical}. For the \(\widehat{\mathfrak{sl}}(2)\) conformal blocks in the right side we have	
	\begin{multline}
	    \log \Psi(-2\sigma \kappa + 2 l, \kappa + 1; -x\kappa^2, z\kappa^4) 
	    \\
	    = -(\kappa+1) S\left(\sigma \left(\frac{\kappa}{\kappa+1}\right) - \frac{l}{\kappa + 1};x\left(\frac{\kappa}{\kappa+1}\right)^2,z\left(\frac{\kappa}{\kappa+1}\right)^4\right)  + S_0 (\sigma; x,z) + O(1/\kappa)
	    \\
	    = -(\kappa+1 - l \partial_{\sigma} - \sigma \partial_{\sigma} - 2 x\partial_x - 4 z \partial_z)S(\sigma;x,z) + S_0 (\sigma; x,z) + O(1/\kappa).
	\end{multline}
	Therefore the leading divergent contributions of the left side and right side are equal to \(\exp\Big(-\kappa S(\sigma;  x, z )\Big)\) and cancel. In the subleading order we get 
	\begin{equation}
		\re^{(1-\sigma\partial_\sigma-2x\partial_x-4z\partial_z)S(\sigma,x;z)}= \sum_{l \in \mathbb{Z}} \frac{1}{\mathtt{t}^{1,0}_{2l}(-2\sigma)^{2}} \re^{ - l \partial_{\sigma}  S(\sigma;  x, z )}\mathrm{F}_1(\ri(\sigma  +  l) ; z).
	\end{equation}
	
	Let \(\beta=-\partial_\sigma S(\sigma,x;z)\), \(p=\partial_x S(\sigma,x;z)\). Now we can apply Theorem~\ref{Th:HJ}. Namely we know that \(S\) satisfies Hamilton-Jacobi equation and assume that \(\sigma,\beta\) are integrals of motion, then \(x,p\) satisfy dynamics \eqref{eq:hamiltoneq} with Painlev\'e \(\mathrm{III}_3\) Hamiltonian. Furthermore, by Proposition~\ref{Th:tauPainleve} the left side is equal to the tau function. Putting all things together we get formula~\eqref{eq:Kyiv}.
\end{proof}


\section{Selberg Integrals} \label{sec:selbint}

In this section, we use free field realization of vertex operators and matrix elements calculated in Theorem \ref{Th:main} for the computation of Selberg-type integrals. In Section \ref{ssec:Un free} we prove the formula for operator $\mathcal{U}_{\nu,n}(x,z)$  in free field realization. This formula can be viewed as an operator analog of the formula for the highest weight vector $u_{n}(\nu)$. In Section~\ref{ssec:Selberg} we use this formula for the computation of Selberg integrals. In Section~\ref{ssec:Integral classical} we rewrite the answer as a product of Gamma functions, find constant term identity, and compare particular cases with identities from \cite{Forrester:1995normalization} and \cite{Petrov:2014}. 

\subsection{Integral representation of vertex operators}\label{ssec:Un free}
We want to find an analog of the Proposition~\ref{Prop:O&Vnu 1} for the primary fields \(\mathcal{U}_{\nu,n}\) defined in formula~\eqref{eq:U nu n def}. A bit informally one can rewrite Proposition~\ref{Prop:O&Vnu 1} in form
\begin{equation}\label{eq:Vnu = O nu}
	\mathcal{V}_{\nu}(x,z) =  \int\limits_{0\le t_N \le \dots \le t_1 \le z}  \!\!\!\! (1 + x \gamma(z))^{\nu} :\!\re^{\frac{\nu} {\sqrt{2\kappa}}\varphi(z)}\!: \prod_{i=1}^N   S(t_i) dt_i, 
\end{equation}
if $- \mu + \nu + \lambda = 2 N$, $N\in \mathbb{Z}_{\ge 0}$. The \(\nu\)-th power is well defined since \(\gamma(z)\) is commutative and we can write \(x\)-expansion using binomial formula.

\begin{theorem}\label{Th:U nu n}
	Assume $- \mu + \nu + \lambda  = 2 N$, where $N\in \mathbb{Z}_{\ge 0}$ and suppose that $n\le 0$. Then we have
	\begin{equation}\label{eq:U nu n = O}
		\mathcal{U}_{\nu,n}(x,z) = \re^{-g_0} :\! \re^{n\sqrt{2}\phi(z)} \!\!: (-1)^{n h_0^{(1)}(h_0^{(1)} - 1)} \re^{g_0} \sum_{r\ge 0} \binom{\nu + 2 n}{r} x^{r} \mathcal{O}^{(N)}_{r; \nu  }(z).
	\end{equation}     
\end{theorem}
Note that the sign factor in formula \eqref{eq:U nu n = O} is just \(1\) for \(n\in \mathbb{Z}\) and coincides with the sign factor in the definition \eqref{eq:b01} of \(b_0(z), b_1(z)\)  for \(n\in \mathbb{Z}+\frac12\). This formula can also be rewritten in form 
\begin{equation}\label{eq:U nu n = O 2}
	\mathcal{U}_{\nu,n}(x,z) = \!\!\! \int\limits_{0\le t_N \le \dots \le t_1 \le z}  \!\!\!\!\!\!\!\!  \Big(\re^{-g_0}:\! \re^{n\sqrt{2}\phi(z)} \!\!: (-1)^{n h_0^{(1)}(h_0^{(1)} - 1)} \re^{g_0} \Big)  (1 + x \gamma(z))^{\nu+2n} :\!\re^{\frac{\nu} {\sqrt{2\kappa}} \varphi(z)}\!\!:\prod_{i=1}^N   S(t_i)dt_i .
\end{equation}


\begin{example}	\label{Exa:U nu 1/2}
	Using the formulas \eqref{eq:V-1/2} and \eqref{eq:Vnu = O nu}we get
	\begin{multline}\label{eq:U nu 1/2}
		\mathcal{U}_{\nu,-{1}/{2}}(x,z)=\Big(b_1(z) \mathcal{V}_{\nu}(x,z)   - \frac{1}{\nu} (b_0(z)+xb_1(z)) \partial_x \mathcal{V}_{\nu}(x,z) \Big)
		\\
		= \int\limits_{0\le t_N \le \dots \le t_1 \le z} \!\!\!\!   \Big(b_1(z) -  b_0(z) \gamma(z) \Big)
		(1 + x \gamma(z))^{\nu-1} : \! \re^{\frac{\nu} {\sqrt{2\kappa}}\varphi(z)}\!\!: \prod_{i=1}^N   S(t_i) dt_i
		\\
		= \int\limits_{0\le t_N \le \dots \le t_1 \le z} \!\!\!\!   \Big(\re^{-g_0}  b_1(z) \re^{g_0}\Big)
		(1 + x \gamma(z))^{\nu-1} :\! \re^{\frac{\nu} {\sqrt{2\kappa}}\varphi(z)}\!\!: \prod_{i=1}^N   S(t_i) dt_i.
	\end{multline}
	So we proved formula \eqref{eq:U nu n = O 2} for \(n=-1/2\).
\end{example}


Theorem~\ref{Th:U nu n} can be viewed as an operator analog of Theorem~\ref{Th:FormUm}. We had two proofs of the latter, one based on explicit computation and another one based on the operator \(I(z)\). Similarly, one can expect two proofs of the Theorem~\ref{Th:U nu n}. 

There is a simplifying step that is common to both proofs. Namely, it is sufficient to prove formula \eqref{eq:U nu n = O} for \(N=0\). Indeed, any vector in \(u \in  L_i\otimes \mathcal{M}_\nu\) is a descendant of the highest weight vector. Therefore the corresponding field \(\mathtt{Y}\otimes \mathtt{Y}(u)\)) is obtained by action of \(e^{(j)}(z), h^{(j)}(z), f^{(j)}(z) \), \(j=1,2\) to the field corresponding to highest weight vector i.e. \(\mathcal{V}_\nu(x,z)\) or \(b(x,z)\mathcal{V}_\nu(x,z)\). However, \(\int S(t) dt\) commutes formally with \(e^{(j)}(z), h^{(j)}(z), f^{(j)}(z) \), \(j=1,2\), therefore commutation relations do not depend on screening insertions.

\begin{proof}[First proof of Theorem \ref{Th:U nu n}] 
	Let \(\mathit{U}_{\nu,n}(x, z)\) denotes the right side of the formula \eqref{eq:U nu n = O}. It is convenient to decompose the proof into three steps. 
	
	\textbf{Step 1.}
	Operator \(\mathit{U}_{\nu,n}(x, z)\) satisfies the following commutation relations
	\begin{subequations}
		\begin{align}
			\label{eq:U nu n e}
			[e^{\Delta}_r,    \mathit{U}_{\nu,n}(x, z)] &= z^r(-x^2 \partial_{x} + (\nu + 2n) x)  \mathit{U}_{\nu,n}(x, z),
			\\
			[h^{\Delta}_r,    \mathit{U}_{\nu,n}(x, z)] &= z^r(-2 x \partial_{x} +  (\nu + 2n))  \mathit{U}_{\nu,n}(x, z),\\
			\label{eq:U nu n f}
			[f^{\Delta}_r,   \mathit{U}_{\nu,n}(x, z)] &= z^r \partial_{x}\, \mathit{U}_{\nu,n}(x, z)	.      
		\end{align}
	\end{subequations}
	It is sufficient to check commutation relations with \(e^{\Delta}_r\) and \(f^{\Delta}_r\).  Recall that we have Lemma~\ref{Lem:efconj} that states
	\begin{equation}\label{eq:conj e,f}
		\re^{g_0}e^\Delta(w) \re^{-g_0} = \beta(w),\quad
		\re^{g_0}f^\Delta(w) \re^{-g_0} =  f^{\Delta}(w) + h^{(1)}(w)\gamma(w)  +\gamma'(w).
	\end{equation}
	Hence it is sufficient to compute commutation relations of the right side \eqref{eq:conj e,f} with 
	\begin{equation}
		 \re^{g_0}\mathit{U}_{\nu,n}(x, z) \re^{-g_0} =  \mathcal{V}_{\nu,n}(x, z)=:\! \re^{n\sqrt{2}\phi(z)} \!\!: (-1)^{n h_0^{(1)}(h_0^{(1)} - 1)}  \sum_{p\ge 0} \binom{\nu + 2 n}{p} x^{p} \mathcal{O}^{(0)}_{p; \nu  }(z).
	\end{equation}
	
	For the commutation with \(\beta_r\) only \(\mathcal{O}^{(0)}_{p; \nu}(z)\) are important. Hence using \eqref{eq:O comm rel e} we have 
	\begin{multline}
		\Big[\beta_r,  \sum_{p\ge 0} \binom{\nu + 2 n}{p} x^{p} \mathcal{O}^{(0)}_{p; \nu  }(z)\Big]	
		= z^r \sum_{p\ge 0} (p+1)\binom{\nu + 2 n}{p+1} x^{p+1} \mathcal{O}^{(0)}_{p; \nu  }(z)
		\\ =z^r\Big(-x^2 \partial_{x} + (\nu + 2n) x\Big)  \sum_{p\ge 0} \binom{\nu + 2 n}{p} x^{p} \mathcal{O}^{(0)}_{p; \nu  }(z).
	\end{multline}
	Hence \eqref{eq:U nu n e} is proven.
	
	For the relation \eqref{eq:U nu n f} note that 
	\begin{equation}
		[\gamma_r,  \mathcal{V}_{\nu,n}(x, z)] =0,\qquad [h^{(1)}_r,   \mathcal{V}_{\nu,n}(x, z)] = z^r  2n \mathcal{V}_{\nu,n}(x, z), \qquad  [f^{(1)}_r,  \mathcal{V}_{\nu,n}(x, z)] =0, 
	\end{equation}
	where in last commutator we used \(n<0\).  Therefore  
	\begin{equation}
		[\sum_s h^{(1)}_{r-s} \gamma_s,   \mathcal{V}_{\nu,n}(x, z)] = z^r  2n \gamma(z) \mathcal{V}_{\nu,n}(x, z).
	\end{equation} 
	It remains to compute (using \eqref{eq:O comm rel f})
	\begin{multline}
	   \Big[f^{(2)}_r, \sum_{p\ge 0} \binom{\nu + 2 n}{p} x^{p} \mathcal{O}^{(N)}_{p; \nu  }(z)\Big] = z^r  \sum_{p\ge 0} (\nu-p)\binom{\nu + 2 n}{p} x^{p} \mathcal{O}^{(0)}_{p+1; \nu  }(z) 
	   \\
	   =z^r\Big (\partial_x -2n \gamma(z)\Big)\Big(\sum_{p\ge 0} \binom{\nu + 2 n}{p} x^{p} \mathcal{O}^{(0)}_{p; \nu  }(z)\Big).
	\end{multline}
	Putting all things together we get \eqref{eq:U nu n f}.
	
	\textbf{Step 2.} We want to show that operator \(\mathit{U}_{\nu,n}(x, z)\) is a field, i..e. belong to the image of \(\mathtt{Y}\otimes \mathtt{Y}\). In case of \(x=0\) the space of fields is generated by 
	\begin{equation}
		:\!P^{(1)}\big[\partial \phi \big] \re^{m\sqrt{2}\phi(z)} (-1)^{m h_0^{(1)}(h_0^{(1)} - 1)}\!:\otimes 
		:\! P^{(2)}\big[\partial \varphi, \beta, \gamma \big] \re^{\frac{\nu} {\sqrt{2\kappa}} \varphi(z)}\!:
	\end{equation}
	where \(P^{(1)}\), \(P^{(2)}\) are two differential polynomials (c.f. formula \eqref{eq:operator-state}). The space of fields for arbitrary \(x\) is obtained by conjugation by \(\re^{x f_0^{\Delta}}\). But due to relation~\eqref{eq:U nu n f} we have 
	\begin{equation}
	    \mathit{U}_{\nu,n}(x, z) = \re^{x f_0^{\Delta}} \mathit{U}_{\nu,n}(0, z)  \re^{- x f_0^{\Delta}},
	\end{equation}
	so we are done.

	\textbf{Step 3.} The space of fields can be identified with the space of states \(L_i\otimes \mathcal{M}_\nu\). We know that $\mathit{U}_{\nu,n}(x, z)$ is a field, in Step 1 we showed that the corresponding vector is the highest weight vector for \(\mathfrak{sl}(2)_{k+1}\) with highest weight \(\nu+2n\). Moreover, it is straightforward to see that the corresponding vector has a conformal dimension equal to the conformal dimension of \(u_n(\nu)\). There exists only one up to proportionally vector with this property in \(L_{2\{n\}}\otimes \mathcal{M}_\nu\). It remains to check the normalization.
\end{proof}

\begin{proof}[Sketch of second proof of Theorem \ref{Th:U nu n}] 
	It is convenient to decompose the proof into three steps. 
	
	\textbf{Step 1'.} Consider $\nu = 1, \mu = \lambda + 1$. As a particular case of the Example \ref{Exa:U nu 1/2} we have 
	\begin{equation}
		\mathcal{U}_{1,-{1}/{2}}(x,z) = \Big(\re^{-g_0}  b_1(z) \re^{g_0}\Big)
		\re^{\frac{1} {\sqrt{2\kappa}}\varphi(z)} =I(z) 
	\end{equation}
	where $I(z)$ is defined by formula \eqref{eq:defI} and we used Proposition~\ref{Prop:I conj}, c.f. Example \ref{Exa:I}.

	\textbf{Step 2'.} One can use the associativity of the operator product expansion in the form 
	\begin{equation}
		\mathtt{Y}(v_2|x,z) \mathtt{Y}(v_1|y,w)
		= \pm \mathtt{Y}\bigg( \mathtt{Y}(v_2|x- y,z - w) v_1\bigg|y,w\bigg).  
	\end{equation}
	Since we work in free field realization, the associativity essentially follows from the similar properties of the lattice algebra \(\mathcal{L}_{0,1}\), \(\beta-\gamma\) system, and the bosonic field \(\varphi\). We omit the details.

	\textbf{Step 3'.} Recall that we can restrict ourselves to the case without screenings, i.e. \(N=0\). We prove formula \eqref{eq:U nu n = O 2} using induction on \(n\). Assuming that the formula is proven for $n$ let us prove it for $n - \frac{1}{2}$. 
	
	It was shown in Proposition \ref{Prop:I vl} that 
	\begin{equation}
		\mathcal{U}_{1,-{1}/{2}}(x,z) u_n(\nu) = (-1)^{n(2n - 1)} z^{d(n,\nu, k)} \Big(u_{n-\frac{1}{2}}(\nu+1) + O(z) \Big) .
	\end{equation}
	Therefore, using associativity we get
	\begin{equation}
		\mathcal{U}_{1,-\frac{1}{2}}(x,z)\mathcal{U}_{\nu,n}(y,w)
		= \pm (z - w)^{d(n,\nu,k)} \left(\mathcal{U}_{\nu+1,n-\frac{1}{2}}(y,w)
		+ O\left(z - w \right)\right)    
	\end{equation}	
	On the other hand, we can explicitly calculate the left side
	\begin{multline}
		\mathcal{U}_{1,-\frac{1}{2}}(x,z)\mathcal{U}_{\nu,n}(y,w) \\
		=  \re^{-g_0}b_1(z)\re^{g_0} :\! \re^{\frac{1}{\sqrt{2\kappa}}\varphi(z)} \!\!: \Big(\re^{-g_0}:\!\re^{n\sqrt{2}\phi(w)} \!\!:(-1)^{n h_0^{(1)}(h_0^{(1)} - 1)} \re^{g_0} \Big)  (1 + y \gamma(w))^{\nu+2n} :\!\re^{\frac{\nu} {\sqrt{2\kappa}} \varphi(w)}\!\!:
		\\
		= \pm (z - w)^{d(n,\nu,k)} \Big(\Big(\re^{-g_0}:\!\re^{(n-\frac{1}{2})\sqrt{2}\phi(w)}\!\!: (-1)^{(n - \frac{1}{2}) h_0^{(1)}(h_0^{(1)} - 1)  } \re^{g_0} \Big) (1 + y \gamma(w))^{\nu+2n}   :\!\re^{\frac{\nu+1} {\sqrt{2\kappa}} \varphi(w)}\!\!: 
		\\
		+ O\left(z - w \right) \Big)
	\end{multline}
So, we are done.
\end{proof}

\subsection{Integral}\label{ssec:Selberg}
Let \(N \in \mathbb{Z}_{\ge 0}\), \(l,m,n \in \frac12\mathbb{Z}_{\ge 0}\) such that \(l+m+n \in \mathbb{Z}\) and \(l+m+n\le N\). We want to compute integral
\begin{multline}\label{eq:S integral}
	S^N_{n,m,l} (\alpha, \beta, g) =
	\int_0^1 \cdots \int_0^1 \prod_{i=1}^N t_i^{\alpha-1}(1-t_i)^{\beta-1}
	\prod_{1 \le i < j \le N} |t_i - t_j |^{2 g} \\ \Big(\prod_{i=1}^{l+m+n}(1-t_i)^{1-2n}t_i^{-2l} \prod_{1 \le i < j \le l+m+n} |t_i - t_j |^{2} \Big) dt_1 \cdots dt_N
\end{multline}	

Note that \(S^N_{0,0,0} (\alpha, \beta, g)\) is a standard Selberg integral, see e.g. \cite{Forrester:2008importance}.


\begin{theorem} \phantomsection \label{Th:Selberg} The Selberg-type integral \eqref{eq:S integral} is equal to 
	\begin{multline} \label{eq:S t}
		{S^N_{n,m,l} (\alpha, \beta, g)}=(-1)^{(l{+}m{+}n)(l{+}m{-}3n{+}1)/2} S^N_{0,0,0}(\alpha, \beta, g) (l+m+n)! 
		\\
		\frac
		{
			\mathtt{t}^{-1,g}_{l{+}m{-}n}(\beta+(N-1)g)
			\mathtt{t}^{-1,g}_{l{-}n{-}m}(1-\alpha-Ng)
			\mathtt{t}^{-1,g}_{{-}l{-}m{-}n}(-(N+1)g)
			\mathtt{t}^{-1,g}_{l{+}n{-}m}(1-\alpha-\beta-(N-1)g)
		}
		{
			\mathtt{t}^{-1,g}_{-2m}(1-\alpha-\beta-(2N-1)g)
			\mathtt{t}^{-1,g}_{2l}(1-\alpha)
			\mathtt{t}^{-1,g}_{-2n}(\beta-g)
		}
	\end{multline}	
\end{theorem}

Recall that the function \(\mathtt{t}\) denotes the product over integer points in the triangle and was defined in the formula~\eqref{eq:triangle}. Note that the overall sign in formula \eqref{eq:S t}, which was a kind of difficulty for the three-point function above, here can be specified by the positivity of integral for big real positive values of \(\alpha,\beta, g\)

\begin{proof} 
	Let us express $\tilde C_{m,n,l}(\mu,\nu,\lambda)$ as an integral. Assume that $m\in \frac12 \mathbb{Z}_{\ge 0}$, $n,l \in \frac12 \mathbb{Z}_{\le 0}$, and $\frac{-\mu+ \nu + \lambda }{2} = N\in\mathbb{Z}_{\ge0}$. Using Theorems~\ref{Th:U nu n},~\ref{Th:FormUm} and Remark~\ref{Rem:l>0} we get
	\begin{multline}
		\Big\langle u_{m}(\mu),\mathcal{U}_{\nu,n}(1,1)\, u_l(\lambda)\Big\rangle 
		\\
		= \sum_{r\ge 0} \binom{\nu + 2 n}{r}  \Big\langle v_{m}(\mu), \re^{g_0} \Big(\re^{-g_0}:\! \re^{n\sqrt{2}\phi(1)}\!\!: (-1)^{n h_0^{(1)} (h_0^{(1)} - 1)} \re^{g_0}\mathcal{O}^{(N)}_{r; \nu}(1)\Big)\re^{-g_0} v_l(\lambda)\Big\rangle
		\\
		=  \sum_{r\ge 0} \binom{\nu + 2 n}{r}  \int\limits_{0\le t_N \le \dots \le t_1 \le 1} \!\!\!\!   \Big\langle v_{m}(\mu),\Big(\gamma^r(1):\!\re^{n\sqrt{2}\phi(1)}\!\!:(-1)^{n h_0^{(1)} (h_0^{(1)} - 1)}:\! \re^{\frac{\nu} {\sqrt{2\kappa}}\varphi(1)}\!\!: \Big)
		\\
		\prod_{i=1}^N \re^{g_0} S(t_i) \re^{-g_0}\, v_l(\lambda)\Big\rangle \prod_{i = 1}^N dt_i
	\end{multline}
	Recall that $\tilde{S}(t) = \re^{g_0} S(t) \re^{-g_0} = (\beta(t) - \re^{\sqrt{2}\phi(t)}):\! \re^{-\frac{2} {\sqrt{2\kappa}}\varphi(t)}\!\!:$, see formula~\eqref{eq:Stilde}. Then we can rewrite integral in the right side
	\begin{multline}\label{eq:subset I}
		\!\!\!\!\!\!  \int\limits_{0\le t_N \le \dots \le t_1 \le 1} \!\!\!\!\!\!\!\!\!\!  \Big \langle v_{m}(\mu),\Big(\gamma^r(1) :\! \re^{n\sqrt{2}\phi(1)}\!\!: (-1)^{n h_0^{(1)} (h_0^{(1)}{-}1)} :\! \re^{\frac{\nu} {\sqrt{2\kappa}}\varphi(1)} \!\!: \Big)
		\\
		\prod_{i=1}^N (\beta(t_i) - :\! \re^{\sqrt{2}\phi(t_i)}\!\!:):\!\re^{-\frac{2}{\sqrt{2\kappa}}\varphi(t_i)}\!\!:  v_l(\lambda)\Big\rangle \prod_{i = 1}^N dt_i
		\\
		= (-1)^{N-r + 4n(m - n)(m - n - \frac{1}{2})}\sum_{I \subset \{1,\dots,N\}}\;\; \int\limits_{0\le t_N \le \dots \le t_1 \le 1}  \Big\langle v_{m}(\mu),\gamma^r(1) :\!\re^{n\sqrt{2}\phi(1)}\!\!:\, :\!\re^{\frac{\nu} {\sqrt{2\kappa}}\varphi(1)}\!\!: 
		\\
		\times\prod_{i=1}^N \Big(  (1- \delta_{i \in I} ) \beta(t_i) +\delta_{i \in I} :\!\re^{\sqrt{2}\phi(t_i)}\!\!:\Big)  :\!\re^{-\frac{2}{\sqrt{2\kappa}}\varphi(t_i)}\!\!: v_l(\lambda)\Big \rangle \prod_{i = 1}^N dt_i.		
	\end{multline}
	The matrix element vanishes unless the number of \(\beta\) is equal to the number of \(\gamma\). Therefore we can assume that \(|I|=N-r\). Furthermore, in order to have non-zero scalar product in \(\widehat{\mathfrak{sl}}(2)_1\) we need $m=  n + |I|+l$. Hence $r=N-m + n + l$ and we have
	\begin{multline}
		\Big\langle v_{m}(\mu),\gamma^r(1) :\!\re^{n\sqrt{2}\phi(1)}\!\!:\, :\!\re^{\frac{\nu} {\sqrt{2\kappa}}\varphi(1)} \!\!:
		\prod_{i=1}^N \Big(  (1- \delta_{i \in I} ) \beta(t_i) +\delta_{i \in I} :\!\re^{\sqrt{2}\phi(t_i)}\!\!:\Big) :\! \re^{-\frac{2}{\sqrt{2\kappa}}\varphi(t_i)}\!\!: v_l(\lambda)\Big\rangle
		\\
		=(-1)^r r! \prod_{i=1}^N t_i^{-\frac{\lambda}{\kappa}}(1 - t_i)^{-\frac{\nu}{\kappa}-1}\prod_{1 \le i<j\le N} (t_i - t_j)^{\frac{2}{\kappa}}\prod_{i  \in I }t_i^{2l}(1-t_i)^{2n+1}\prod_{ i,j \in I, \, i<j}(t_i - t_j)^2.
	\end{multline}

	Now we can replace \((t_i - t_j)^{\frac{2}{\kappa}}\) by \(|t_i - t_j|^{\frac{2}{\kappa}}\) and the integral domain by a cube \([0,1]^N\) with additional factor \(N!\). Then the value of the integral does not depend on the choice of \(I\) and we can assume that \(I=\{1,\dots,N-r\}\) with additional factor \(\binom{N}{r}\). Finally we get integral~\eqref{eq:S integral}, namely
	\begin{multline}
		\frac{\tilde{C}_{m,n,l}(\mu,\nu,\lambda)}{\|\tilde u_{l}(\lambda)\|^2\|\tilde u_{n}(\nu)\|^2} = \Big\langle u_{m}(\mu),\mathcal{U}_{\nu,n}(1,1)\, u_l(\lambda)\Big\rangle
		\\ 
		= 
		(-1)^{N+ 4n(m - n)(m - n - \frac{1}{2})}\binom{\nu + 2n}{N -m + n + l}\frac{1}{(m - n - l)!}S^N_{-n,m,-l}\left(-\frac{\lambda}{\kappa}+1,-\frac{\nu}{\kappa},\frac{1}{\kappa}\right).   
	\end{multline}
	
	Therefore we get 
	\begin{multline}
		\frac{S^N_{-n,m,-l}\left(-\frac{\lambda}{\kappa}+1,-\frac{\nu}{\kappa},\frac{1}{\kappa}\right)}{S^N_{0,0,0}\left(-\frac{\lambda}{\kappa}+1,-\frac{\nu}{\kappa},\frac{1}{\kappa}\right)}
		\\
		=(-1)^{ 4n(m - n)(m - n - \frac{1}{2})} \frac{\tilde{C}_{m,n,l}(\nu{+}\lambda{-}2N,\nu,\lambda)}{\tilde{C}_{0,0,0}(\nu{+}\lambda{-}2N,\nu,\lambda)} \frac{(m-n-l)!}{ \left\|\tilde u_{l}(\lambda)\right\|^2\|\tilde u_{n}(\nu)\|^2} \binom{\nu}{N}\binom{\nu + 2n}{N -m + n + l}^{-1}	    
	\end{multline}
	It implies the result.
\end{proof}

\begin{remark}
	Similarly, for \(n=0\), one can use bosonization of the vertex operator given in Proposition \ref{Prop:O&Vnu 2}. Then we get a relation 
	\begin{equation}
		\frac{S^N_{0,m,-l}\left(-\frac{\lambda}{\kappa}+1,\frac{\nu+2}{\kappa},\frac{1}{\kappa}\right)}{S^N_{0,0,0}\left(-\frac{\lambda}{\kappa}+1,\frac{\nu+2}{\kappa},\frac{1}{\kappa}\right)}=(-1)^{m-l} 		\frac{(m-l)!}{\left\|\tilde u_{l}(\lambda)\right\|^2}
		\frac{\tilde{C}_{m,0,l}(\lambda{-}\nu{-}2{-}2N,\nu,\lambda)}{\tilde{C}_{0,0,0}(\lambda{-}\nu{-}2{-}2N,\nu,\lambda) }.		
	\end{equation}
	One can view this as an additional check of the agreement between three-point functions formula \eqref{eq:td C t} and Selberg integral formula \eqref{eq:S t}.
\end{remark}

\subsection{Symmetries, constant term, particular cases}\label{ssec:Integral classical}

Note that the right side of Selberg integral formula \eqref{eq:S t} agrees with the natural symmetry of the integral
\begin{equation}
	S^N_{n,m,l} (\alpha, \beta, g)=S^N_{l+1/2,m,n-1/2} (\beta,\alpha, g)
\end{equation}
There are also additional symmetries which are not clear from the integral form but follow easily from the answer \eqref{eq:S t}
\begin{align}
	\label{eq:m-n sym}
	\frac{S^N_{n,m,l} (\alpha, \beta, g)}{S^N_{0,0,0}(\alpha, \beta, g)}&= (-1)^{2(m-n)(n+m+l)}	\frac{S^N_{m,n,l} (\alpha, 1-\alpha-\beta-2(N-1)g, g)}{S^N_{0,0,0}(\alpha, 1-\alpha-\beta-2(N-1)g, g)},\\
	\label{eq:m-l sym}
	\frac{S^N_{n,m,l} (\alpha, \beta, g)}{S^N_{0,0,0}(\alpha, \beta, g)}&=	(-1)^{(2m-2l-1)(n+m+l)} \frac{S^N_{n,l+1/2,m-1/2} (1-\alpha-\beta-2(N-1)g,\beta g)}{S^N_{0,0,0}( 1-\alpha-\beta-2(N-1)g,\beta, g)},
\end{align}

Since \(S^N_{0,0,0}(\alpha, \beta, g)\) is the standard Selberg integral, formula \eqref{eq:S t} actually gives an explicit answer for \(S^N_{n,m,l} (\alpha, \beta, g)\). This answer can be written in terms of gamma functions. To present the answer in a more manageable form, we will use notations 
\begin{equation}
	G_{N,R,r}(\alpha,g)	= 
	\left[ 
	\begin{aligned}
		&\prod\limits_{j=0}^{N-1} \! \Gamma \big(\alpha{+}jg{+}\mathds{1}_{j>N-R}(j{-}(N-R)){+}\mathds{1}_{j<2 r}(j{-}2 r)\big)	&\text{ if \(R>0\)}, 
		\\
		&\frac
		{
			\prod\limits_{j=0}^{N-R-1} \!\! \Gamma \big(\alpha{+} jg{+}\mathds{1}_{j<2 r}(j{-}2 r)\big)		}
		{
			\prod\limits_{j=N}^{N-R-1}  \Gamma \big(\alpha{+}jg{+}(j{-}(N{-}R))\big)
		}
		&\text{ if \(R\le 0 \)},
	\end{aligned} 
	\right.
\end{equation}
\begin{equation}
	\tilde{G}_{N,R,r}(\alpha,g)	= 
	\left[ 
	\begin{aligned}
		&\frac
		{
			\prod\limits_{j=-R+1}^{-1} \!\! \Gamma \big(\alpha+ jg+R+j\big)		}
		{
			\prod\limits_{j=-R+1}^{N-1}  \Gamma \big(\alpha+jg+\mathds{1}_{j>N-2 r}(j+2r-N)\big)
		}
		&\text{ if \(R> 0 \)},	
		\\
		&\prod\limits_{j=0}^{N-1} \! \Gamma \big(\alpha+jg+\mathds{1}_{j<-R}(j+R)+\mathds{1}_{j>N-2r}(j+2 r-N)\big)^{-1}	&\text{ if \(R\le 0\)}. 
	\end{aligned} 
	\right.
\end{equation}
\begin{corollary} The integral \ref{eq:S integral} has the from 
	\begin{multline} \label{eq:Selberg integral}
		S^N_{n,m,l}(\alpha,\beta,g)=\Gamma(1+m+l+n) \prod_{j=0}^{N-1} \frac{ \Gamma\big(1+(j+1)g+\mathds{1}_{j>N{-}m{-}l{-}n}(j{-}N{+}m{+}l{+}n)\big)}{\Gamma(1+g)} \\ 
		G_{N,m+n-l,l}(\alpha,g)	G_{N,l+m-n+1,n-1/2}(\beta,g)	\tilde{G}_{N,m-n-l,m}(\alpha+\beta+(N-1)g,g).
	\end{multline}
\end{corollary}

\begin{example}
	Consider particular case \(l+m+n=N\). In this case, the integral becomes the standard Selberg integral with shifted parameters
	\begin{multline}
		S^N_{n,N-n-l,l}(\alpha,\beta,g)=	\int_0^1 \cdots \int_0^1 \prod_{1 \le i \le N} t_i^{\alpha-1-2l}(1-t_i)^{\beta-2n}
		\prod_{1 \le i < j \le N} |t_i - t_j |^{2 g+2} dt_1 \cdots dt_N
		\\
		= \prod_{j=1}^{N-1} \frac{\Gamma(\alpha-2l+j (g+1))\Gamma(\beta+1-2n+j (g+1))\Gamma(1+(j+1) (g+1)}{\Gamma(\alpha+\beta+1-2l-2n+(N+j-1)(g+1))\Gamma(2+g)}.
	\end{multline}
	It is straightforward to compare this with the formula \eqref{eq:Selberg integral}.	
\end{example}

\begin{example} 
	Another simple example is \(m=1,l=n=0\). In this case, the integral reduces to the very particular case of the Aamoto integrals \cite{Aomoto:1987jacobi} and we have from \eqref{eq:S t} that 
	\begin{equation}
		S^N_{0,1,0}(\alpha,\beta,g)=S^N_{0,0,0}(\alpha,\beta,g)  \frac{\beta+(N-1)g}{\alpha+\beta+(2N-2)g}.
	\end{equation}
\end{example}

Using the standard argument (see e.g. \cite{Forrester:2008importance}) the evaluation of the integral \eqref{eq:S integral} is equivalent to the computation of the constant term
\begin{multline}\label{eq:M C.T.}
	M^{N,a,b}_{N',a',b'}(g) = \operatorname{C.T.}\left[\prod_{1\le j \le N} (1-x_j)^a(1-1/x_j)^b
	\prod_{1 \le i \neq j \le N} (1-{x_i}/{x_j})^g \right.\\ \left.
	\Big(\prod_{1\le j \le N'}(1-x_j)^{a'}(1-1/x_j)^{b'} \!\!\!\!\!\prod_{1 \le i \neq j \le N'}\!\!(1-{x_i}/{x_j}) \Big)\right],
\end{multline}	
where  \(\alpha=-b-(N-1)g\), \(\beta=a+b+1\), \(N'=l+m+n\), \(a'=m-n-l\), \(b'=l+1-m-n\). In order to have Laurent polynomial we assume that \(a,b,g\) are non-negative integer numbers and \(N',a',b'\) satisfy conditions
\begin{equation} \label{eq:N' a' b'}
	0\le N'\le N,\;\; -N'\le a' \le N',\;\; 1-N' \le b' \le 1-a',\;\; a+b+a'+b'\ge 0.
\end{equation}
Note that the constant term \eqref{eq:M C.T.} has obvious \(a \leftrightarrow b\), \(a' \leftrightarrow b'\) symmetry which corresponds to \(l \leftrightarrow m\) symmetry \eqref{eq:m-l sym}. The following result follows from Theorem \ref{Th:Selberg}
\begin{corollary}The constant term \eqref{eq:M C.T.} under conditions \eqref{eq:N' a' b'} has the form 
	\begin{multline} \label{eq:M t}
		{M^{N,a,b}_{N',a',b'} (g)}= M^{N,a,b}_{0,0,0}(g) N'! 
		\frac
		{
			\mathtt{t}^{-1,g}_{-N'-a'-b'}(-a-b-g N)
			\mathtt{t}^{-1,g}_{-b'}(-b)
			\mathtt{t}^{-1,g}_{-N'}(-(N+1)g)
			\mathtt{t}^{-1,g}_{-a'}(-a)
		}
		{
			\mathtt{t}^{-1,g}_{-a'-N'}(-a-g N)
			\mathtt{t}^{-1,g}_{-b'-N'}(-b-gN)
			\mathtt{t}^{-1,g}_{-a'-b'}(-a-b)
		}.
	\end{multline}	
\end{corollary}
Note that \(M^{N,a,b}_{0,0,0}(g)\) is the Morris constant term, so the formula \eqref{eq:M t} gives an explicit expression for \({M^{N,a,b}_{N',a',b'} (g)}\). Since now all parameters are integer numbers the rational functions on right side sometimes \eqref{eq:M t} require some care because naively they can lead to 0/0 indeterminacy. 
\begin{example} 
	In the case \(a'=b'=0\) the constant term \eqref{eq:M C.T.} coincides with the one conjectured by Forrester \cite{Forrester:1995normalization} and proven in \cite[Th. 6.2]{Petrov:2014} (more precisely one has to take \(q\rightarrow 1\) limit of \(m=0\) case of this theorem)
	\begin{multline}
		M^{N,a,b}_{N',0,0}(g) = \operatorname{C.T.}\left[\prod_{1\le j \le N} (1-x_j)^a(1-1/x_j)^b
		\prod_{1 \le i \neq j \le N} (1-{x_i}/{x_j})^g 
		\!\!\!\prod_{1 \le i \neq j \le N'}\!\!(1-{x_i}/{x_j}) \right]
		\\ 
		=M^{N,a,b}_{0,0,0}(g) N'! \prod_{j=N-N'}^{N-1}  \frac{(a+b+gj+(j-N+N'))^{\downarrow j-N+N'}(g+gj+(j-N+N'))^{\downarrow j-N+N'}}{(a+gj+(j-N+N'))^{\downarrow j-N+N'}(b+gj+(j-N+N'))^{\downarrow j-N+N'}}.  						
	\end{multline}
	where \(x^{\downarrow k}=\prod_{i=0}^{k-1}(x-i)\).	
\end{example}

\begin{example}
	In the case \(a'=1, b'=0\) the constant term coincides with another particular case of \cite[Th. 6.2]{Petrov:2014}, namely one has to take \(q\rightarrow 1\) limit and set \(n_0=n-m\) in notations of loc. cit
	\begin{multline}
		M^{N,a,b}_{N',1,0}(g) = \operatorname{C.T.}\!\left[\prod_{1\le j \le N} \!\!(1-x_j)^a(1-1/x_j)^b
		\!\!\!\!\prod_{1 \le i \neq j \le N}\!\!\! (1-{x_i}/{x_j})^g 
		\!\!\prod_{1\le j \le N'}\!\!(1-x_j)
		\!\!\!\!\prod_{1 \le i \neq j \le N'}\!\!\!(1-{x_i}/{x_j}) \right]
		\\ 
		=M^{N,a,b}_{0,0,0}(g) N'! \!\! \prod_{j=N-N'}^{N-1} \!\!\! \frac{(a+b+gj+(j{-}N{+}N'{+}1))^{\downarrow j{-}N{+}N'{+}1}(g+gj+(j{-}N{+}N'))^{\downarrow j{-}N{+}N'}}{(a+gj+(j{-}N{+}N'{+}1))^{\downarrow j{-}N{+}N'{+}1}(b+gj+(j{-}N{+}N'))^{\downarrow j{-}N{+}N'}}. 				
	\end{multline}		
\end{example}

It would be interesting to find a more direct proof of the constant term identity \eqref{eq:M t}, for example using the methods of \cite{Petrov:2014}. Perhaps the conditions \eqref{eq:N' a' b'} can be weakened. 

Another possible question is whether the constant term \eqref{eq:M C.T.} has meaning from the point of view of the Calogero-Sutherland model similar to \cite{Forrester:1995normalization}.

\appendix    

\section{Monodromy cancellation} \label{sec:hypmonodr}

Recall notations for hypergeometric function introduced in formula \eqref{eq:2F1}. Recall also transformation \(\hat{I}\) defined in \eqref{eq:hat I}. In order to write monodromy of hypergeometric function we will need also transformations \(\hat{R}\) and \(\hat{S}\)
\begin{equation}
    \hat{R} \Vec{a} = (a_1, a_1 - a_3 + 1, a_1 - a_2 + 1),\quad \hat{S} \Vec{a} = (a_2, a_1, a_3).
\end{equation}
These transformation are not independent, namely they satisfy relations $\hat{I} =\hat{R} \hat{S} \hat{R},  \hat{R}^2 =  \hat{S}^2 = 1,  (\hat{R} \hat{S})^4 = 1$. The group of transformations of $\mathbb{R}^3$ which generated by $\hat{R},\hat{I},\hat{S}$ is isomorphic to dihedral group  $\mathrm{Dih}_4$, i.e. group of symmetries of a square. The following proposition is standard.

\begin{proposition}
	There is a following identity for hypergeometric functions
	\begin{equation}\label{eq:hgeominf}
	   {}_2F_1(\Vec{a}| z) = g(\Vec{a}) (-z)^{-a_1} \, {}_2F_1(\hat{R}\Vec{a}|z^{-1})+	
	   g(\hat{S}\Vec{a})(-z)^{-a_2} {}_2F_1(\hat{R}\hat{S}\Vec{a}|z^{-1}).
	\end{equation}
	where $ g(\Vec{a}) = \dfrac{\Gamma (a_3) \Gamma (a_2-a_1) }{\Gamma (a_2) \Gamma (a_3-a_1)}$.
\end{proposition}

Let $\Vec{r} = (r_1,r_2,r_3)\in \mathbb{Z}^3$. Consider the following products of hypergeometric functions
\begin{subequations}
	\begin{align}
		H^{(1)}_{\Vec{r}}(\Vec{a}|z)&= {}_2F_1(\Vec{a}|z)  {}_2F_1(-\Vec{a}-\Vec{r}|z), 
		\\
		H^{(2)}_{\Vec{r}}(\Vec{a}|z)&=z^{2 + r_3} {}_2F_1(\hat{I}\Vec{a}|z)  {}_2F_1(\hat{I}(-\Vec{a}-\Vec{r})|z).
	\end{align}	
\end{subequations}

\begin{proposition}\phantomsection \label{Prop:MonCanc}
	Assume that \(a_1,a_2,a_3\) are generic. Then the function 
	\begin{equation}
		H_{\Vec{r}}(\Vec{a}|z) = H^{(1)}_{\Vec{r}}(\Vec{a}|z) + p_{\Vec{r}}(\Vec{a}) H^{(2)}_{\Vec{r}}(\Vec{a}|z)
	\end{equation} 
	is a rational function of $z$ with poles in $0, 1$ and  $\infty$, if and only if 
	\begin{multline}\label{eq:p formula}
	    p_{\Vec{r}}(\Vec{a}) = (-1)^{1 + r_3}\frac{  g(\hat{S}\Vec{a})g(-\Vec{a}-\Vec{r})}{ g(\hat{S}\hat{I}\Vec{a})g(\hat{I}(-\Vec{a}-\Vec{r})) } =  (-1)^{1 + r_3} \frac{\Gamma(a_3)}{ \Gamma(a_3 +2 + r_3)} \frac{\Gamma(-a_3 -r_3)}{ \Gamma(-a_3 + 2)}\times
	     \\
	    \times \frac{\Gamma(a_1  + 1 + r_1)}{ \Gamma(a_1)}
	     \frac{\Gamma(-a_2 + 1)}{ \Gamma(-a_2 - r_2)}
	     \frac{\Gamma(a_1-a_3 + 1)}{ \Gamma(a_1 - a_3 + r_1 - r_3)}
	     \frac{\Gamma(a_3  - a_2 + 1 - r_2 + r_3)}{ \Gamma(a_3  - a_2)}.
	\end{multline}
	\end{proposition}
\begin{proof} 
	It follows from the formula \eqref{eq:hgeominf} that
	\begin{multline}\label{eq:H1 infinity}
	       H^{(1)}_{\Vec{r}}(\Vec{a}|z) = (-1)^{-r_1}  g(\Vec{a})  g(-\Vec{a}-\Vec{r}) z^{ r_1}\, {}_2F_1(\hat{R}\Vec{a}|z^{-1})  \,{}_2F_1(\hat{R}(-\Vec{a}-\Vec{r})|z^{-1})+
	       \\
	       +(-1)^{- r_2} g(\hat{S}\Vec{a})g(\hat{S}(-\Vec{a}-\Vec{r}))z^{r_2}  \,{}_2F_1(\hat{R}\hat{S}\Vec{a}|z^{-1})  \,{}_2F_1(\hat{R}\hat{S}(-\Vec{a}-\Vec{r})|z^{-1}) +
	       \\
	       +(-1)^{-a_1 + a_2 - r_1} g(\hat{S}\Vec{a})g(-\Vec{a}-\Vec{r}) z^{ a_1 -  a_2 + r_1} \,{}_2F_1(\hat{R}\hat{S}\Vec{a}|z^{-1})  \,{}_2F_1(\hat{R}(-\Vec{a}-\Vec{r})|z^{-1}) +
	       \\
	       +(-1)^{a_1 - a_2 - r_2} g(\Vec{a})g(\hat{S}(-\Vec{a}-\Vec{r})) z^{ -  a_1 + a_2 + r_2} \,{}_2F_1(\hat{R}\Vec{a}|z^{-1})  \,{}_2F_1(\hat{R}\hat{S}(-\Vec{a}-\Vec{r})|z^{-1}), 
	\end{multline}
	\begin{multline}\label{eq:H2 infinity}
	       H^{(2)}_{\Vec{r}}(\hat{I}\Vec{a}|z) = (-1)^{-r_1 + 2 + r_3}   g(\hat{I}\Vec{a}) g(\hat{I}(-\Vec{a}-\Vec{r})) z^{r_1 } \,{}_2F_1(\hat{R}\hat{I}\Vec{a}|z^{-1})  \,{}_2F_1(\hat{R}\hat{I}(-\Vec{a}-\Vec{r})|z^{-1})+
	       \\
	       +(-1)^{- r_2 + 2 + r_3}  g(\hat{S}\hat{I}\Vec{a})g(\hat{S}\hat{I}(-\Vec{a}-\Vec{r}))z^{r_2 } \,{}_2F_1(\hat{R}\hat{S}\hat{I}\Vec{a}|z^{-1})  \,{}_2F_1(\hat{R}\hat{S}\hat{I}(-\Vec{a}-\Vec{r})|z^{-1}) +
	       \\
	       +(-1)^{-a_1 + a_2 - r_1 + 2 + r_3}  g(\hat{S}\hat{I}\Vec{a})g(\hat{I}(-\Vec{a}-\Vec{r})) z^{ a_1 - a_2 + r_1} \,{}_2F_1(\hat{R}\hat{S}\hat{I}\Vec{a}|z^{-1})  \,{}_2F_1(\hat{R}\hat{I}(-\Vec{a}-\Vec{r})|z^{-1}) +
	       \\
	       +(-1)^{a_1 - a_2 - r_2 + 2 + r_3}  g(\hat{I}\Vec{a})g(\hat{S}\hat{I}(-\Vec{a}-\Vec{r})) z^{ - a_1 + a_2 + r_2 } \,{}_2F_1(\hat{R}\hat{I}\Vec{a}|z^{-1})  \,{}_2F_1(\hat{R}\hat{S}\hat{I}(-\Vec{a}-\Vec{r})|z^{-1}). 
	\end{multline}
	There is a correspondence between terms in right sides of formulas \eqref{eq:H1 infinity} and \eqref{eq:H2 infinity}, namely 
	\begin{align}\label{eq:hypprod}
	    {}_2F_1(\hat{R}\Vec{a}|z^{-1})  {}_2F_1(\hat{R}(-\Vec{a}-\Vec{r})|z^{-1})  &= {}_2F_1(\hat{R}\hat{I}\Vec{a}|z^{-1})  {}_2F_1(\hat{R}\hat{I}(-\Vec{a}-\Vec{r})|z^{-1}),
	    \\
	    {}_2F_1(\hat{R}\hat{S}\Vec{a}|z^{-1})   {}_2F_1(\hat{R}\hat{S}(-\Vec{a}-\Vec{r})|z^{-1}) &={}_2F_1(\hat{R}\hat{S}\hat{I}\Vec{a}|z^{-1})  {}_2F_1(\hat{R}\hat{S}\hat{I}(-\Vec{a}-\Vec{r})|z^{-1}),
	    \\
	    {}_2F_1(\hat{R}\hat{S}\Vec{a}|z^{-1})   {}_2F_1(\hat{R}(-\Vec{a}-\Vec{r})|z^{-1})&= {}_2F_1(\hat{R}\hat{S}\hat{I}\Vec{a}|z^{-1})  {}_2F_1(\hat{R}\hat{I}(-\Vec{a}-\Vec{r})|z^{-1}),
	    \\
	    {}_2F_1(\hat{R}\Vec{a}|z^{-1})   {}_2F_1(\hat{R}\hat{S}(-\Vec{a}-\Vec{r})|z^{-1})&= {}_2F_1(\hat{R}\hat{I}\Vec{a}|z^{-1})  {}_2F_1(\hat{R}\hat{S}\hat{I}(-\Vec{a}-\Vec{r})|z^{-1}),
	\end{align}
	where we used
	\begin{equation}
	    \hat{R} = \hat{S}\hat{R}\hat{I}, ~~~ \hat{R}\hat{S} =  \hat{S}\hat{R}\hat{S}\hat{I}, ~~~    {}_2F_1(\Vec{a}|z) = {}_2F_1(\hat{S}\Vec{a}|z).
	\end{equation}
	
	Due to our assumptions, the third and fourth terms in the right sides of formulas \eqref{eq:H1 infinity} and \eqref{eq:H2 infinity} should cancel each other. The cancellation of the third term gives 
	\begin{equation}
	     p_{\Vec{r}}(\Vec{a}) =  (-1)^{1 + r_3}\frac{  g(\hat{S}\Vec{a})g(-\Vec{a}-\Vec{r})}{ g(\hat{S}\hat{I}\Vec{a})g(\hat{I}(-\Vec{a}-\Vec{r})) } 
	\end{equation}
	and for the forth term we get
	\begin{equation}
	     p_{\Vec{r}}(\Vec{a}) =   (-1)^{1 + r_3}\frac{  g(\Vec{a})g(\hat{S}(-\Vec{a}-\Vec{r}))}{ g(\hat{I}\Vec{a})g(\hat{S}\hat{I}(-\Vec{a}-\Vec{r})) } 
	\end{equation}
	Using definition of function \(g(\Vec{a})\) it is straightforward to see that these formulas are equivalent and equivalent to \eqref{eq:p formula}.
	
	On the other hand,  the function \(H_{\Vec{r}}(\Vec{a}|z)\) can have singularities only at \(z=0,1,\infty\) and these singularities are branching points. For \(p_{\Vec{r}}(\Vec{a})\) given by formula \eqref{eq:p formula} the arguments above shows that the function \(H_{\Vec{r}}(\Vec{a}|z)\) at \(z=0,\infty\) can have only poles. Hence the monodromy of \(H_{\Vec{r}}(\Vec{a}|z)\) at \(z=1\) is trivial. Therefore the singularities at \(z=1\) are also just poles. Hence \(H_{\Vec{r}}(\Vec{a}|z)\) is a rational function.
\end{proof}

\section{Three point functions} \label{sec:Barnes}
We basically follow \cite{Nekrasov:2018bps}, see also \cite[App. E]{Nakajima:2003lectures}.

Let \(\chi\) be a (probably infinite) sum of terms of the form \(\re^{-\tau \xi}\), where \(\tau\) is a formal parameter. By conjugation \(\chi^*\) we denote operation which acts as \(\re^{-\tau \xi} \mapsto \re^{\tau \xi}\). Let us introduce the function (plethystic exponent) 
\begin{equation}\label{eq:gamma}
	\mathsf{E} \left[ \chi\right]= \exp\left(\frac{d}{ds}|_{s=0}\frac{1}{\Gamma(s)}\int_{0}^{+\infty}\frac{d\tau}{\tau} \tau^s \chi^* \right).
\end{equation}
%
%
%
%
%
Then for any finite sum we have
\begin{equation}
	\mathsf{E} \left[ \sum \re^{\tau\xi_i}-\sum \re^{\tau\eta_j}\right] = \frac{\prod \eta_j}{\prod \xi_i}.
\end{equation}
Using the definition \eqref{eq:gamma} the plethystic exponent can be also defined for the infinite sums (under the certain restrictions in order to ensure convergence).

Let us introduce functions 
\begin{align}
	&\mathsf{C}^{\mathrm{Vir}}(u_1,u_2,u_3;q_1,q_2) = 
	\mathsf{E}	\left[ 
	\frac{q_1^2q_2^2(u_1u_2u_3 -u_2^2 - u_3^{2})+u_1u_2^{-1}u_3^{-1}+u_1^{-1}u_2u_3^{-1}+u_1^{-1}u_2^{-1}u_3 - u_1^{-2}}{(1-q_1)(1-q_2)}, 
	\right]
	\\
	&\mathsf{C}^{\mathfrak{sl}}(u_1,u_2,u_3;q_1,q_2)=\mathsf{E}\left[ \frac{q_1^2q_2(u_1u_2u_3-u_1^2-u_2^2)+u_1u_2^{-1}u_3^{-1}+u_1^{-1}u_2u_3^{-1}+u_1^{-1}u_2^{-1}u_3- u_1^{-2}}{(1-q_1)(1-q_2)} \right].
\end{align}
Note that triangle function \(\mathtt{t}\)  introduced in \eqref{eq:triangle} has following expressions
\begin{subequations}\label{eq:triangle pleth}
	\begin{align}
		\mathsf{E}	\left[ \frac{v q_1^n-v}{(1-q_1)(1-q_2/q_1)}+\frac{v q_2^n-v}{(1-q_1/q_2)(1-q_2)} \right] &= (-1)^{n(n+1)/2}\mathtt{t}_n^{\epsilon_1,\epsilon_2}(-\alpha),
		\\
		\mathsf{E}	\left[ \frac{q_2}{q_1}\frac{v q_1^n-v}{(1-q_1)(1-q_2/q_1)}+\frac{v q_2^n-v}{(1-q_1/q_2)(1-q_2)} \right] &= \mathtt{t}_{-n}^{\epsilon_1,\epsilon_2}(\alpha-\epsilon_1),	
	\end{align}	
\end{subequations}
where \(v=\re^{\tau \alpha}, q_1=\re^{\tau\epsilon_1}, q_2=\re^{\tau\epsilon_2}\).
Using this we have 
\begin{multline}\label{eq:ratio of C}
	\frac{\mathsf{C}^{\mathfrak{sl}}(u_1 q_1^l ,u_2 q_1^n ,u_3 q_1^m;q_1,q_1^{-1}q_2)\mathsf{C}^{\mathrm{Vir}}(u_1 q_2^m,u_2 q_2^n,u_3 q_2^l;q_1 q_2^{-1},q_2)}{\mathsf{C}^{\mathfrak{sl}}(u_1,u_2,u_3;q_1,q_1^{-1}q_2)\mathsf{C}^{\mathrm{Vir}}(u_1,u_2,u_3;q_1 q_2^{-1},q_2)}
	\\
	 =\pm \frac{\mathtt{t}^{\epsilon_1,\epsilon_2}_{-m-n-l}(a_1{+}a_2{+}a_3{+}\epsilon_1) \mathtt{t}^{\epsilon_1,\epsilon_2}_{m-n-l}(a_1{+}a_2{-}a_3) \mathtt{t}^{\epsilon_1,\epsilon_2}_{-m+n-l}(a_1{-}a_2{+}a_3) \mathtt{t}^{\epsilon_1,\epsilon_2}_{-m-n+l}({-}a_1{+}a_2{+}a_3)}{\mathtt{t}^{\epsilon_1,\epsilon_2}_{-2l}(2a_1) \mathtt{t}^{\epsilon_1,\epsilon_2}_{-2n}(2a_2+\epsilon_1)  \mathtt{t}^{\epsilon_1,\epsilon_2}_{-2m}(2a_3+\epsilon_1)} 
	\\
	=\pm \frac{\tilde{C}_{m,n,l}(\mu,\nu,\lambda)}{\tilde{C}_{0,0,0}(\mu,\nu,\lambda) \|\tilde{u}_l(\lambda)\|^2}
\end{multline}
where in first transformation we used notations \(u_1=\re^{\tau a_1}, u_2=\re^{\tau a_2}, u_3=\re^{\tau a_3}, q_1=\re^{\tau\epsilon_1}, q_2=\re^{\tau\epsilon_2}\) and in second transformation we used 
Theorems~\ref{Th:main}, Corollary \ref{Cor:tilde U}, and identification of parameters
\begin{equation}\label{eq:identification}
	\frac{\epsilon_1}{\epsilon_2}=-\frac{1}{\kappa}, \quad \frac{a_1}{\epsilon_2}=-\frac{\lambda}{2\kappa}\quad \frac{a_2}{\epsilon_2}=-\frac{\mu}{2\kappa} \quad \frac{a_3}{\epsilon_2}=-\frac{\nu}{2\kappa}.
\end{equation}
The overall signs in formula \eqref{eq:ratio of C} can be easily computed from relations \eqref{eq:triangle pleth}, however they are not illuminating and we omit them for simplicity. 

Using these functions we can renormalize four-point conformal blocks. For the algebra \(\widehat{\mathfrak{sl}}(2)\) we define it as follows
\begin{equation}
	\Psi_k^{\mathrm{full}}\left(\vec{\mu},\lambda;x, z\right)=(-1)^{\lambda}\mathsf{C}^{\mathfrak{sl}}(u_1,u_2,u';q_1,q_2) \mathsf{C}^{\mathfrak{sl}}(u',u_3,u_4;q_1,q_2) \Psi_k\left(\vec{\mu},\lambda;x, z\right),
\end{equation}
where \(u_i=\re^{\tau \lambda_i}\), \(u'=\re^{\tau \lambda'}\), \(q_1=\re^{2 \tau}\), \(q_2=\re^{-2 \kappa \tau}\), c.f. \eqref{eq:identification}. For the Virasoro case we define 
\begin{equation}
	\mathrm{F}_b^{\mathrm{full}}\left(\vec{P},P';z \right) 	=\mathsf{C}^{\mathrm{Vir}}(u_1,u_2,u';q_1,q_2) \mathsf{C}^{\mathrm{Vir}}(u',u_3,u_4;q_1,q_2) \mathrm{F}_b\left(\vec{P},P';z \right),
\end{equation}
where \(u_i=\exp(\tau\big(\frac{2P_i}{b+b^{-1}}-1)\big)\), \(u'=\exp(\tau\big(\frac{2P'}{b+b^{-1}}-1)\big)\), \(q_1=\exp\big(\frac{2b^{-1} \tau}{(b+b^{-1}}\big)\), \(q_2=\exp\big(\frac{2b \tau}{(b+b^{-1}}\big)\).

Note that this normalization factors satisfy relation 
\begin{equation}\label{eq:coset vacuum}
	\mathsf{C}^{\mathfrak{sl}}(u_1,u_2,u_3;q_1,q_2)=\mathsf{C}^{\mathfrak{sl}}(u_1,u_2,u_3;q_1,q_1^{-1}q_2)\mathsf{C}^{\mathrm{Vir}}(u_1,u_2,u_3;q_1 q_2^{-1},q_2).
\end{equation}
The \(q_1,q_2\) parameters here exactly correspond to coset relations, e.g. \eqref{eq:blowup}, Namely, we take, \(q_1=\re^{2 \tau}\), \(q_2=\re^{-2 (k+2) \tau}\), then for \(\mathsf{C}^{\mathfrak{sl}}\) on the right side parameters are \(\re^{2 \tau}, \re^{-2 (k+3) \tau}\) i.e. correspond to \(\widehat{\mathfrak{sl}}(2)_{k+1}\) and for \(\mathsf{C}^{\mathrm{Vir}}\) parameters are \(\re^{2 (k+3)\tau}, \re^{-2 (k+2) \tau}\) i.e. correspond to Virasoro algebra with $b = b_{GKO}=-\ri \sqrt{\frac{k+2}{k+3}}$. Relation on parameters \(u\) also agrees with relation $P(\lambda) = P_{GKO}(\lambda)=-\frac{\lambda +1}{2 (k+2) } b_{GKO}$.

Using this normalization and relations \eqref{eq:ratio of C}, \eqref{eq:coset vacuum} the (coset, blowup) relation on conformal block \eqref{eq:blowup} takes the form
\begin{equation}
	\Psi_k^{\mathrm{full}}\left(\vec{\mu},\lambda;x, z\right)
	=
	\sum_{l \in \mathbb{Z}}
	\Psi_{k+1}^{\mathrm{full}}\left(\vec{\mu},\lambda+2l;x, z\right)\mathrm{F}_b^{\mathrm{full}}\left(\vec{P},P(\lambda)+lb;z \right). 	
\end{equation}

Similarly, one can renormalize conformal blocks and hide coefficients in other (coset, blowup) relations, e.g. in \eqref{eq:blowup tor} or \eqref{eq:blowup Whittaker}.

%

\paragraph{Data Availability}  The authors declare that the data supporting the findings of this study are available within the paper.

\paragraph{Conflict of Interest} The authors have no relevant financial or non-financial interests to disclose.

\bibliographystyle{alpha}
\addcontentsline{toc}{section}{\refname}  
\bibliography{Coset}

\noindent \textsc{School of Mathematics, University of Edinburgh, Edinburgh, UK 
	}

\emph{E-mail}:\,\,\textbf{mbersht@gmail.com}\\

\noindent\textsc{Section de Math\'ematiques, Universit\'e de Gen\'eve, Geneva, Switzerland
	\\	HSE University, Moscow, Russia
	}

\emph{E-mail}:\,\,\textbf{trufaleks2022@gmail.com}\\

\noindent \textsc{Hebrew University of Jerusalem, Jerusalem, Israel
	\\	HSE University, Moscow, Russia
	}

\emph{E-mail}:\,\,\textbf{borfeigin@gmail.com}	        
\end{document}